\documentclass[12pt]{article}
\usepackage{amsthm}
\usepackage{amsmath}
\usepackage{amssymb}
\usepackage{latexsym}
\usepackage{enumitem}
\usepackage{comment}
\usepackage{xcolor}
\usepackage{authblk}
\usepackage{array}
\usepackage{multirow}
\usepackage{float}
\usepackage{stmaryrd} 

\newcommand{\norm}[1]{\left\lVert #1 \right\rVert}

\DeclareMathOperator*{\argmax}{arg\,max}
\DeclareMathOperator*{\argmin}{arg\,min}

\newtheorem{definition}{Definition}[section]
\newtheorem{theorem}{Theorem}[section]
\newtheorem{lemma}{Lemma}[section]
\newtheorem{proposition}{Proposition}[section]

\newtheorem{example}{Example}[section]
\newtheorem{corollary}{Corollary}[section]
\newtheorem{assumption}{Assumption}[section]

\newtheorem{remark}{Remark}[section]

\usepackage[authoryear,round,sort]{natbib}
\usepackage[colorlinks,linkcolor=blue,anchorcolor=blue,citecolor=red]{hyperref}

\DeclareMathOperator{\esssup}{ess\hspace{0.1em}sup}

\allowdisplaybreaks

\title{Group-averaged Markov chains: mixing improvement}
\author[1]{Michael C.H. Choi\thanks{Email: mchchoi@nus.edu.sg, corresponding author}}
\author[1]{Youjia Wang\thanks{Email: e1124868@u.nus.edu}}
\affil[1]{Department of Statistics and Data Science, National University of Singapore, Singapore}
\date{\today}

\begin{document}

\maketitle

\begin{abstract}
For Markov kernels $P$ on a general state space $\mathcal{X}$, we introduce a new class of averaged Markov kernels $P_{da}(G,\nu)$ of $P$ induced by a group $G$ that acts on $\mathcal{X}$ and a probability measure $\nu$ on $G \times G$. Notable special cases are the group-orbit average $\overline{P}$, left-average $P_{la}$, right-average $P_{ra}$ and the independent-double-average $(P_{la})_{ra}$. For $\pi$-stationary $P$ in which $\pi$ is invariant with respect to $G$, we show that in general $P_{da}$ enjoys favorable convergence properties than $P$ based on metrics such as spectral gap or asymptotic variance, and within the family of $P_{da}$ the most preferable kernel is in general $(P_{la})_{ra}$. We demonstrate that $P_{la}, P_{ra}, (P_{la})_{ra}$ are comparable in terms of mixing times, which supports the use of $P_{la}, P_{ra}$ in practice as computationally cheaper alternatives over $(P_{la})_{ra}$. These averaged kernels also admit natural geometric interpretations: they emerge as unique projections of $P$ onto specific $G$-invariant structures under the Kullback–Leibler divergence or the Hilbert–Schmidt norm and satisfy Pythagorean identities. On the other hand, in the general case if $\pi$ is not invariant with respect to $G$, we propose and study a technique that we call state-dependent averaging of Markov kernels which generalizes the earlier results to this setting. As examples and applications, this averaging perspective not only allows us to recast state-of-the-art Markov chain samplers such as Hamiltonian Monte Carlo or piecewise-deterministic Markov processes as specific cases of $P_{da}$, but also enables improvements to existing samplers such as Metropolis-Hastings, achieving rapid mixing in some toy models or when $\pi$ is the discrete uniform distribution. \newline	
\textbf{Keywords}: Markov chains, spectral gap, group-orbit average, permutations, Kullback-Leibler divergence, information projection, Markov chain Monte Carlo, Metropolis-Hastings, Hamiltonian Monte Carlo\newline
\textbf{AMS 2020 subject classification}: 05E18, 60J10, 60J20, 60J22, 65C40, 94A15, 94A17
\end{abstract}

\tableofcontents

\section{Introduction}

This paper centers on the theme of leveraging symmetry, group structure of the target distribution and averaging of Markov kernels to improve Markov chain mixing. When the state space admits a group action and the target distribution $\pi$ is compatible with that action, one can often reorganize transitions so that the chain explores states modulo the symmetry more efficiently. This paper develops a general version of this principle. Given a Markov kernel $P$ on a general Polish state space on which a locally compact group $G$ acts measurably, we introduce \emph{group‐induced averages} of $P$ defined to be
\[
P_{da}(G,\nu):=\mathbb{E}_{(g,h)\sim\nu}(U_gPU_h),
\]
where $U_g[f](x) := f(gx)$ for $f \in L^2(\pi)$ is the permutation operator associated with the action of $g\in G$, and $\nu$ is a probability measure on $G\times G$. Notable special cases include
\begin{align*}
	P_{la} &:= \mathbb{E}_{g \sim \mu}(U_gP), \\
	P_{ra} &:= \mathbb{E}_{g \sim \mu}(PU_g), \\
	(P_{la})_{ra} &:= \mathbb{E}_{(g,h) \sim \mu \otimes \mu}(U_g P U_h),
\end{align*}
that we call respectively the left-average, the right-average and the independent-double-average induced by $G$ and $\mu$ with $\mu$ being the Haar measure. These kernels can readily be shown to be $\pi$-stationary if $P$ is itself $\pi$-stationary. Intuitively, $P_{da}$ mixes the local dynamics of $P$ with global ``orbit moves", thereby facilitating jumps between different parts of the state space that can be otherwise hard to reach using the original dynamics.

We stress that this paper offers a unifying perspective: not only are many kernels of modern samplers can be understood as suitable averages, but it also sheds light on new methods to improve the mixing of existing samplers via leveraging group structure, geometric invariance and averaging.

\paragraph{Contributions and organizations.}
Our main results quantify—in spectral, geometric, and information–theoretic terms—how such averaging improves convergence.

\begin{itemize}[leftmargin=2em]
	\item \textbf{Properties of $P_{da}$.} In Section \ref{sec:prelim}, we begin our paper by properly defining the double-average $P_{da}$ and its special cases. In particular, we derive properties of $P_{da}$ and demonstrate an inheritance of properties from that of $P$.
	
	\item \textbf{Spectral and asymptotic variance improvement.} We prove that group‐induced averaging $P_{da}$ does not decrease the spectral gap and, under a natural misalignment condition between the $G$–invariant functions and the eigenspace corresponding to the gap, averaging strictly increases it (Section~\ref{sec:improvement}). In addition, we compare the asymptotic variance and demonstrate that $P_{da}$ does not increase the asymptotic variance and gives conditions under which averaging strictly decreases it. Results on these metrics also suggest that $(P_{la})_{ra}$ compares favorably with other kernels within the family of $P_{da}$.
	
	\item \textbf{Mixing time comparison.} Working with worst-case mixing times based on $L^p$ distances ($1\le p\le\infty$), we compare the independent-double-average $(P_{la})_{ra}$ against the computationally cheaper one–sided averages $P_{la}$ and $P_{ra}$ (Section~\ref{sec:comparison}), and demonstrate that these times are of comparable order. This gives practical insights on the simulation of these kernels.
	
	\item \textbf{Information projections and Pythagorean identities.} We show that $P_{da}$ are \emph{information projections} onto the corresponding $G$–invariant sets of kernels under $\pi$–weighted Kullback–Leibler (KL) divergence or the Hilbert-Schmidt (HS) norm. We prove Pythagorean identities and identifies the isotropic average $\overline P$ as the closest $G$–invariant kernel to $P$ both in KL and in HS distance (Section~\ref{sec:geometry}). This offers geometric justification that these averaged kernels arise naturally.
	
	\item \textbf{Beyond exact symmetry: artificial group planting.} Even when $\pi$ lacks a natural group invariance, we propose ``artificial group planting’’ strategies in Section \ref{sec:groupplanting} that adjoin a tractable symmetry while preserving the target $\pi$ (via state-dependent averaging or importance–sampling corrections). This extends the scope of our averaging constructions to general settings in which $\pi$ may not possess an inherent symmetric structure.
	
	\item \textbf{Examples and reformulations.} In Section \ref{sec:examples}, we recast several widely used samplers—including Swendsen–Wang, Hamiltonian Monte Carlo and piecewise-deterministic Markov processes—as special cases of $P_{da}$. We also present case studies where averaging yields provable acceleration of classical samplers on bimodal toy targets or when $\pi$ is the discrete uniform distribution, illustrating how to pick $G$ in practice.
\end{itemize}

\subsection{Related works}
Beating ``diffusivity" is the main theme in designing accelerated Markov chains. Classical random-walk type samplers tend to explore large state space of rugged target distribution inefficiently. Adding non-local jumps is a standard approach to deal with this issue. For multi-modal target distributions, jumps can enable the chain to traverse between different modes, such as parallel tempering, simulated annealing and importance sampling \citep{bertsimas1993simulated, neal2001annealed, earl2005parallel}, where different temperatures or more tractable distributions are used as a bridge to build jumps for exploration to escape local traps. Some possibly non-Markovian or particle-based algorithms also lie in this direction, such as the equi-energy sampler, Wang-Landau algorithm and sequential Monte Carlo \citep{kou2006equi, wang2001efficient, del2006sequential}. Another line works on extended state space, such as lifted MCMC, Hamiltonian Monte Carlo, underdamped Langevin diffusion and piecewise-deterministic Markov process (PDMP) \citep{diaconis2000analysis, neal2011mcmc, cheng2018underdamped, davis1984piecewise}, where a more ``deterministic" flow/direction is introduced to counter diffusive wandering, and jumps play a key role in switching between different flows/directions and retaining stationarity. This resonates with the ``hit and run" argument in \citep{andersen2007hit} which unifies many samplers under the same framework, with jumps corresponding to the ``hit" part. 

Specifically for finite Markov chains, various ways of adding jumps are more explicitly studied. Slightly modifying the underlying edges for random walks on graphs can significantly change mixing time, for example see \citep{hermon2018sensitivity, ding2013sensitivity, hermon2018sensitivitycutoff}. Furthermore, if the target distribution is uniform, composing a fixed permutation on the whole state space after each step in most cases can also substantially improve mixing, such as \citep{chung1987random, bordenave2019cutoff, chatterjee2021correction}. A relatively independent topic analogous to lifted MCMC to speed up mixing is non-backtracking Markov chains with less chance to revisit the path, see \citep{alon2007non, ben2017cutoff, diaconis2013spectral}. 

Selecting effective non-local jumps is the key challenge. A promising principle is to exploit symmetries of the target distribution, and use group actions as the natural way to characterize the induced jump maps. Symmetry of distribution is a universal phenomenon especially in problems originating from nature, for example multi-modal distributions with modes arranged in symmetric patterns. Apart from many of the algorithms introduced above, some other algorithms also utilize the symmetries to improve classical samplers. For random walks on graphs targeting the uniform distribution, \citep{boyd2009fastest} uses the automorphism group of the underlying graph to obtain a fastest mixing Markov chain, and the orbital Markov chains in \citep{niepert2012markov, niepert2013symmetry} also use automorphism groups to improve mixing in probabilistic graphical models. In \citep{kou2006equi, andrieu2021peskun, choi2025improving}, a density-preserving jump is introduced, where the isometric involution $\psi$ with $\psi^2=e$ serving as the jump in \citep{andrieu2021peskun, choi2025improving} forms a two-element flipping group $\mathbb Z_2=\{e, \psi\}$ encoding the mirror symmetry of target distribution. Such symmetry appears naturally in many lifted MCMC and PDMP constructions, as well as in some mean-field models without external fields. An intriguing question arises from these two works: the involutive constraint $\psi^2=e$ is restrictive, what if $\psi$ has higher order as $\psi^k=e$? An intuitive answer is to use cyclic group $\langle \psi\rangle$ to organize the jumps. Similar ways of generalizations seem to be an interesting direction, yet systematic study of symmetries under other groups and their implications for designing improved samplers remains underexplored. 

For target distributions lacking exact symmetry, the ``approximate symmetry" phenomenon is pointed out in \citep{ying2025multimodal, ying2023double} that appears in many statistical physics problems such as models under low temperatures, and in which similar group-based jumps may be applied. In \citep{ying2023double}, a Metropolis move is injected between isometric involutions for Ising model with boundary conditions. However, a rigorous theoretical characterization of such phenomenon, as well as a unified approach to deal with general distributions from the group-symmetric perspective are largely open. 

Symmetry under groups and projections in geometry are closely related. Particularly in terms of Markov chains, many samplers can be viewed as projections of a baseline kernel onto a symmetry-defined subset of Markov kernels. A classical example is the Metropolis-Hastings (MH) algorithm, which arises as a projection of the proposal chain onto the set of reversible kernels under suitable $L^1$-type norm on transition matrices \citep{billera2001geometric}. A continuous-time analogue of this geometric viewpoint is developed \citep{diaconis2009characterizations}. Related ``information-projection" constructions for Markov chains appear in \citep{choi2023systematic, wolfer2021information}. More recently, under the $\pi$-weighted KL divergence between Markov chains studied in \citep{wolfer2021information, wang2023information}, the proposed kernel in \citep{choi2025improving} can be seen as the projection onto the set of transition matrices invariant under the flipping group generated by isometric involution as mentioned earlier. Under the same spirit, in this article we show that the group-induced averaged kernels admit a geometric characterization as the projection onto various group-invariant subsets of kernels. Combining group, operator and geometric perspectives together, this construction provides new lens for designing improved Markov chains.  

Finally, there is also a strand of works applying group structures in specific algorithms for reasons other than encoding target symmetries. In classical Monte Carlo, \citep{wendel1957groups} incorporates group orbits into conditional Monte Carlo to reduce estimator variance. In the context of MCMC, \citep{liu2000generalised, khare2011spectral} use group actions as a coordinate-free alternative to classical blocking in Gibbs sampler and extra parametrization in data augmentation algorithm, which claim to have similar effects of ``dimension reduction". \citep{kamatani2023non} uses group elements as the directions in guided Metropolis-Hastings. A natural question remains unsolved in these works: are groups essential here, or could more general transform families do as well or better? Our article offers a symmetry-based rationale --- the most fundamental role of group is to organize invariance and symmetry --- and justify these constructions in a general way through that lens, clarifying when the group structure is intrinsic versus merely convenient. Specifically, we take advantage of the group symmetry in the target distribution and group-invariant structure in the space of Markov kernels to guide the design of our improved samplers.


\subsection{Notations}

We shall adapt the following notations throughout the paper. We write $\llbracket a,b \rrbracket := \{a, a+1, \ldots, b-1, b\}$ with $a,b \in \mathbb{Z}$ and $a \leq b$. We also denote by $\llbracket n \rrbracket := \llbracket 1,n \rrbracket$ for $n \in \mathbb{N}$. For $h \in \mathbb{R}$, we write $h_+ := \max\{h,0\}$. We write that \[
f(n) \in \Theta(h(n)) \iff \exists\, c_1,c_2>0,\; n_0 \in \mathbb{N} :
c_1 h(n) \leq f(n) \leq c_2 h(n),\; \forall n \geq n_0.
\]

\section{Preliminaries}\label{sec:prelim}

Let $X = (X_n)_{n \in \mathbb{N}_0}$ be a time-homogeneous discrete-time Markov chain on a measurable Polish state space $(\mathcal{X},\mathcal{F})$, and we denote by $P$ to be the Markov kernel which describes the one-step transition. Recall that for $P : \mathcal{X} \times \mathcal{F} \to [0,1]$ to be a Markov kernel, for each fixed $A \in \mathcal{F}$, the mapping $x \mapsto P(x,A)$ is $\mathcal{F}$-measurable and for each fixed $x \in \mathcal{X}$, the function $A \mapsto P(x,A)$ is a probability measure on $\mathcal{X}$. Given a function $f: \mathcal{X} \to \mathbb{R}$ and a signed measure $\mu$ on $(\mathcal{X},\mathcal{F})$, $P$ acts on $f$ from the left and $\mu$ from the right by
$$ P[f](x) := \int_{\mathcal{X}} f(y) P(x,dy), \quad \mu P(A) := \int_{\mathcal{X}} P(x,A) \mu(dx), \quad x \in \mathcal{X}, A \in \mathcal{F},$$
whenever the above integrals exist. The set of all Markov kernels on $\mathcal{X}$ is written as $\mathcal{L} := \mathcal{L}(\mathcal{X})$.

We denote by $\mathcal{P} := \mathcal{P}(\mathcal{X})$ to be the set of probability measures with support on $\mathcal{X}$. We say that $\pi \in \mathcal{P}$ is a stationary distribution of $X$ if
$$\int_{\mathcal{X}} P(x,A) \, \pi(dx) = \pi(A), \quad A \in \mathcal{F}.$$
We say that $X$ is reversible if there is a probability measure $\pi \in \mathcal{P}$ such that the \textit{detailed balance} relation is satisfied:
$$\pi(dx)P(x,dy) = \pi(dy) P(y,dx).$$
Let $L^2(\pi)$ be the Hilbert space of real-valued measurable functions on $\mathcal{X}$ that are squared-integrable with respect to $\pi$, endowed with the inner product $\langle f,h \rangle_{\pi} := \int f h \, d\pi$ and the norm $\norm{f}_{\pi} := \langle f,f \rangle_{\pi}^{1/2}$. $P$ can then be viewed as a linear operator on $L^2(\pi)$, in which we still denote the operator by $P$. The operator norm of $P$ on $L^2(\pi)$ is
$$||P||_{L^2 \to L^2} = \sup_{\substack{f \in L^2(\pi) \\ ||f||_{\pi}=1}} ||P[f]||_{\pi}.$$
Similarly, we define $L^2_0(\pi) := \{f \in L^2(\pi);~\langle f, \mathbf{1}\rangle_\pi=0\}$ as the Hilbert space orthogonal to $\mathbf{1}$, and the operator norm is 
\begin{equation*}
    \|P\|_{2\to 2}:=\|P\|_{L_0^2\to L_0^2}=\sup_{\substack{f \in L_0^2(\pi) \\ \|f\|_{\pi}=1}} \|P[f]\|_{\pi}.
\end{equation*}
Let $P^*$ be the adjoint or time-reversal of $P$ on $L^2(\pi)$, and it can be checked that
$$\pi(dx)P^*(x,dy) = \pi(dy) P(y,dx).$$
In this way, we write $\mathcal{L}(\pi) := \{P \in \mathcal{L};~ P^* = P\}$, the set of all $L^2(\pi)$-self-adjoint Markov kernels, and $\mathcal{S}(\pi) := \{P \in \mathcal{L};~ \pi = \pi P\}$, the set of all $\pi$-stationary Markov kernels. For $\pi \in \mathcal{P}$, we also write $\Pi : L^2(\pi) \to L^2(\pi)$ to be the rank-1 projection operator induced by $\pi$, defined to be $\Pi[f] := \pi(f) = \langle 1,f \rangle_\pi$.

Let $G$ be a group that acts on the state space $\mathcal{X}$. For $g \in G$, we define the \textbf{permutation operator} $U_g : L^2(\pi) \to L^2(\pi)$ induced by $g$ to be, 
\begin{align*}
	U_g [f](x) := f(gx).
\end{align*}
A function $f \in L^2(\pi)$ is said to be \textbf{$G$-invariant} if $f = U_g[f]$ for all $g \in G$. $\pi \in \mathcal{P}$ is said to be \textbf{$G$-invariant} if, for all $g \in G$ and $A \in \mathcal{F}$,
\begin{align*}
	\pi(A) = \pi(gA)
\end{align*}
holds, where $gA := \{gx;~x \in A\}$. We denote by $\mathcal{I}(G) := \{\pi \in \mathcal{P};~ G\textrm{-invariant} \, \pi\}$, the set of all $G$-invariant probability measures. It can readily be seen that if $\pi \in \mathcal{I}(G)$, then $U_g$ is an unitary operator on $L^2(\pi)$ with adjoint $U_g^* = U_{g^{-1}} = U_g^{-1}$.

Throughout this article, we make the following assumption:
\begin{assumption}\label{assum:pi and G}
    For $G$ and $\pi$, we assume 
    \begin{itemize}
        \item $\pi$ admits a density denoted by $\pi(x)$ w.r.t. the reference measure $\mathfrak{m}$ on $\mathcal{X}$.

        \item $\dfrac{d\mathfrak{m}\circ g^{-1}}{d\mathfrak{m}}$ exists and equal to $1$ for any $g\in G$ (if $\mathcal{X}=\mathbb R^d$, this is equivalent to $|\mathrm{det}(Dg)|=1$ when $\mathfrak{m}$ is taken to be the Lebesgue measure).
    \end{itemize}
\end{assumption}

Next, $P \in \mathcal{L}$ is said to be \textbf{$(U_g, U_g^{-1})$-invariant} if 
\begin{align*}
	P = U_g P U_g^{-1}.
\end{align*}
$P$ is said to be \textbf{$(G,G^{-1})$-invariant} if $P$ is $(U_g, U_g^{-1})$-invariant for all $g \in G$, and we write $\mathcal{L}(G,G^{-1}) := \{P \in \mathcal{L};~(G,G^{-1})\textrm{-invariant}\, P\}$. Analogously, $P \in \mathcal{L}$ is said to be \textbf{$(U_g, U_g)$-invariant} if 
\begin{align*}
	P = U_g P U_g.
\end{align*}
$P$ is said to be \textbf{$(G,G)$-invariant} if $P$ is $(U_g, U_g)$-invariant for all $g \in G$, and we write $\mathcal{L}(G,G) := \{P \in \mathcal{L};~(G,G)\textrm{-invariant}\, P\}$.

In the spirit of the previous paragraph, we define the notion of left-invariant and right-invariant $P$. For a fixed $g \in G$, $P \in \mathcal{L}$ is said to be \textbf{$U_g$-left-invariant} if 
\begin{align*}
	P = U_g P.
\end{align*}
$P$ is said to be \textbf{$G$-left-invariant} if $P$ is $U_g$-left-invariant for all $g \in G$, and we write $\mathcal{LI}(G) := \{P \in \mathcal{L};~G\textrm{-left-invariant}\, P\}$. Analogously, for a fixed $g \in G$, $P \in \mathcal{L}$ is said to be \textbf{$U_g$-right-invariant} if 
\begin{align*}
	P = P U_g.
\end{align*}
$P$ is said to be \textbf{$G$-right-invariant} if $P$ is $U_g$-right-invariant for all $g \in G$, and we write $\mathcal{RI}(G) := \{P \in \mathcal{L};~G\textrm{-right-invariant}\, P\}$.

\subsection{Group-induced averages $P_{da}$ and its special cases $\overline{P}$, $\widetilde{P}$, $P_{la}$, $P_{ra}$, $(P_{la})_{ra}$}

Several natural notions of averaging over the group $G$ arise, and from now on we assume $G$ is a locally compact topological group equipped with a Haar measure $\mu$. First, we define
\begin{align*}
	\overline{P} &= \overline{P}(G) := \int_G U_g P U_g^{-1} \, \mu(dg) = \mathbb{E}_{g \sim \mu}(U_g P U_g^{-1}), \\
	\widetilde{P} &= \widetilde{P}(G) := \int_G U_g P U_g \, \mu(dg) = \mathbb{E}_{g \sim \mu}(U_g P U_g). 
\end{align*}
Note that $\overline{P}$ is also known as the \textbf{group-orbit average} of $P$ induced by $G$, see for example \cite[Section $2.2$]{boyd2009fastest}. Analogously, we define the \textbf{left-average} (resp.~\textbf{right-average}) of $P$ with respect to $G$ to be
\begin{align*}
	P_{la} &= P_{la}(G) := \int_G U_g P \, \mu(dg) = \mathbb{E}_{g \sim \mu}(U_g P), \\
	P_{ra} &= P_{ra}(G) := \int_G P U_g \, \mu(dg) = \mathbb{E}_{g \sim \mu}(P U_g).
\end{align*}
More generally, for given probability measure $\nu$ on $G \times G$, we define the \textbf{general-double-average} of $P$ with respect to $G$ and $\nu$ to be
\begin{align*}
	P_{da} = P_{da}(G,\nu) := \mathbb{E}_{(g,h) \sim \nu}(U_g P U_h).
\end{align*}
We also let
\begin{align*}
	\mathcal{D}(G,\nu) := \{P \in \mathcal{L};~ P_{da}(G,\nu) = P\}
\end{align*}
to be the set of Markov kernels that are invariant under the general-double-average. In particular, the \textbf{independent-double-average} of $P$ is defined to be the general-double-average of $P$ with respect to $G$ and the product measure $\mu \otimes \mu$, that is,
\begin{align*}
	P_{da}(G, \mu \otimes \mu) = \mathbb{E}_{(g,h) \sim \mu \otimes \mu}(U_g P U_h) = (P_{la})_{ra}.
\end{align*}
From the above equation we see that
\begin{align}\label{eq:Plararala}
	(P_{la})_{ra} = (P_{ra})_{la}.
\end{align}
We also note that
\begin{align}
	((P_{da}(G,\nu))_{la})_{ra} &= \mathbb{E}_{(u,v) \sim \mu \otimes \mu} \mathbb{E}_{(g,h) \sim \nu} (U_u U_g P U_h U_v) \nonumber \\
	&=  \mathbb{E}_{(g,h) \sim \nu} \mathbb{E}_{(u,v) \sim \mu \otimes \mu}  (U_u U_g P U_h U_v) \nonumber \\
	&= \mathbb{E}_{(g,h) \sim \nu} (P_{la})_{ra} \nonumber \\
	&= (P_{la})_{ra}. \label{eq:Pdalara}
\end{align}
In fact, it can readily be checked that the averages introduced thus far are special cases of the general-double-average, and hence $P_{da}$ can be understood as a unified notion:
\begin{itemize}
	\item $h = g^{-1}$, $g \sim \mu$: $P_{da} = \overline{P}$
	\item $h = g$, $g \sim \mu$: $P_{da} = \widetilde{P}$
	\item $h = e$, $g \sim \mu$: $P_{da} = P_{la}$
	\item $h \sim \mu$, $g = e$: $P_{da} = P_{ra}$
\end{itemize}

In the following proposition, we prove that $\overline{P}$ (resp.~$\widetilde{P}, P_{la}, P_{ra}, (P_{la})_{ra}$) is a Markov kernel that belongs to $\mathcal{L}(G,G^{-1})$ (resp.~$\mathcal{L}(G,G), \mathcal{LI}(G), \mathcal{RI}(G), \mathcal{LI}(G) \cap \mathcal{RI}(G)$) under suitable assumptions.

\begin{proposition}\label{prop:Abelian}
	Let $G$ be a locally compact topological group with Haar measure $\mu$ that acts on $\mathcal{X}$. We then have
	\begin{align*}
		\overline{P} &\in \mathcal{L}(G,G^{-1}), \quad P_{la} \in \mathcal{LI}(G), \quad P_{ra} \in \mathcal{RI}(G), \quad (P_{la})_{ra} \in \mathcal{LI}(G) \cap \mathcal{RI}(G).
	\end{align*}
	If $G$ is further assumed to be an Abelian group, then 
	\begin{align*}
		\widetilde{P} \in \mathcal{L}(G,G).
	\end{align*}
\end{proposition}

\begin{proof}
	First, it is trivial to see that $\overline{P}, \widetilde{P}, P_{la}, P_{ra}$ are Markov kernels on $\mathcal{X}$: they map non-negative $f$ to $\overline{P}[f], \widetilde{P}[f], P_{la}[f], P_{ra}[f] \geq 0$. Also, it can readily be checked that $\overline{P}[\mathbf{1}], \widetilde{P}[\mathbf{1}], P_{la}[\mathbf{1}], P_{ra}[\mathbf{1}] = 1$, where $\mathbf{1}$ is the constant function of value $1$.
	
	Let $h \in G$, and consider
	\begin{align*}
		U_h \overline{P} U_h^{-1} = \int U_h U_g P U_g^{-1} U_h^{-1} \, \mu(dg) = \int U_{hg} P U_{hg}^{-1} \, \mu(dhg) = \overline{P},
	\end{align*}
	where the second equality uses $\mu$ is $G$-invariant. Similarly, we see that
	\begin{align*}
		U_h P_{la} &= \int U_h U_g P \, \mu(dg) = \int U_{hg} P \, \mu(dhg) = P_{la}, \\
		P_{ra} U_h &= \int P U_g U_h \, \mu(dg) = \int U_{gh} P \, \mu(dgh) = P_{ra}, \\
		U_h (P_{la})_{ra} &= U_h (P_{ra})_{la} = (P_{ra})_{la} = (P_{la})_{ra}, \\
		(P_{la})_{ra} U_h &= (P_{la})_{ra},
	\end{align*}
	where we use \eqref{eq:Plararala} in the third line above. Finally, we compute that
	\begin{align*}
		U_h \widetilde{P} U_h = \int U_h U_g P U_g U_h \, \mu(dg) = \int U_{hg} P U_{gh} \, \mu(dg) = \int U_{hg} P U_{hg} \, \mu(dhg) = \widetilde{P},
	\end{align*}
	where the third equality utilizes the Abelian property of $G$ and $\mu$ is $G$-invariant.
\end{proof}

$P \in \mathcal{L}$ is said to be a trace-class operator if 
\begin{align*}
	\sum_{e \in \mathcal{B}} \langle |P|[e],e \rangle_\pi < \infty,
\end{align*}
where $|P| := \sqrt{P^*P}$ and $\mathcal{B}$ is a set of orthonormal basis of $L^2(\pi)$. If $P$ is a trace-class operator, we define its trace to be
\begin{align*}
	\mathrm{Tr}(P) := \sum_{e \in \mathcal{B}} \langle P[e],e \rangle_\pi,
\end{align*} 
where the right hand side is independent of the chosen set of basis $\mathcal{B}$. $P \in \mathcal{L}$ is said to be a Hilbert-Schmidt operator if 
\begin{align*}
	\norm{P}_{\mathrm{HS}}^2 := \sum_{e \in \mathcal{B}} \norm{P[e]}_\pi^2 = \sum_{f,e \in \mathcal{B}} |\langle f,P[e]\rangle_\pi |^2 < \infty,
\end{align*}
where $\mathcal{B}$ is a set of orthonormal basis of $L^2(\pi)$ and the right hand side is independent of the chosen $\mathcal{B}$. When $P \in L^2(\pi)$, functional calculus gives that, for $f,g \in L^2(\pi)$,
\begin{align*}
	\langle P[f],h \rangle_\pi = \int_{[-1,+1]} \lambda \, d \langle \mathcal{E}_\lambda [f],h \rangle_\pi,
\end{align*}
where $(\mathcal{E}_\lambda)$ is the spectral measure associated with $P$.

In the following proposition, we summarize properties in which $P_{da}$ (particularly the special cases $\overline{P}, \widetilde{P}, P_{la}, P_{ra}, (P_{la})_{ra}$) inherits from that of $P$.

\begin{proposition}[Inheritance of properties]\label{prop:inheritance of properties}
	Let $G$ be a locally compact topological group with Haar measure $\mu$ that acts on $\mathcal{X}$. Assume further that $\pi \in \mathcal{I}(G)$ is $G$-invariant. We have
	\begin{enumerate}
		\item\label{it:inherit1}($\pi$-stationarity) If $P \in \mathcal{S}(\pi)$, then $P_{da}(G,\nu)$ (and hence $\overline{P}, \widetilde{P}, P_{la}, P_{ra}, (P_{la})_{ra}$) $\in \mathcal{S}(\pi)$ and $(P_{la})^* = P^*_{ra}$.
		
		\item\label{it:inherit2}($\pi$-reversibility) If $P \in \mathcal{L}(\pi)$ and $\nu$ is symmetric in the sense that $(g,h) \overset{D}{=} (h^{-1},g^{-1}) \sim \nu$ where $\overset{D}{=}$ denotes equality in distribution, then $P_{da}(G,\nu) \in \mathcal{L}(\pi)$. In particular, this yields $\overline{P}, \widetilde{P}, (P_{la})_{ra} \in \mathcal{L}(\pi)$.
		
		\item\label{it:inherit5}(compactness) Assume that $G$ is a finite group. If $P$ is a compact operator, then $P_{da}(G,\nu)$ (and hence $\overline{P}, \widetilde{P}, P_{la}, P_{ra}, (P_{la})_{ra}$) are compact operators.
		
		\item\label{it:inherit3}(trace-class operator) Suppose that $P$ is a trace-class operator and $G$ is a finite group. Then
		\begin{align*}
			\mathrm{Tr}(P) = \mathrm{Tr}(\overline{P}),
		\end{align*}
		and hence $\mathrm{Tr}(\overline{P}) < \infty$. Assume further that $P \in \mathcal{L}(\pi)$ (and hence $\overline{P}$) is a non-negative $L^2(\pi)$-self-adjoint operator, then $P$ is trace-class implies that $\overline{P}$ is trace-class.
		
		\item\label{it:inherit4}(Hilbert-Schmidt operator) If $P$ is a Hilbert-Schmidt operator, then
		\begin{align*}
			\norm{P}_{\mathrm{HS}} \geq \norm{\overline{P}}_{\mathrm{HS}}, \quad \norm{P}_{\mathrm{HS}} \geq \norm{P_{la}}_{\mathrm{HS}} \geq \norm{(P_{la})_{ra}}_{\mathrm{HS}}, \quad \norm{P}_{\mathrm{HS}} \geq \norm{P_{ra}}_{\mathrm{HS}}.
		\end{align*}
		Consequently, $\overline{P}, P_{la}, P_{ra}, (P_{la})_{ra}$ are Hilbert-Schmidt operators. If $G$ is assumed to be a finite group, then $P_{da}(G,\nu)$ and hence $\widetilde{P}$ are Hilbert-Schmidt operators.		
	\end{enumerate}
\end{proposition}

\begin{remark}[A two-point example where $\mathrm{Tr}(P_{la}) = \mathrm{Tr}(P_{ra}) > \mathrm{Tr}(P)$]
	We see in Proposition \ref{prop:inheritance of properties} that when $P$ is a trace-class operator the projection $\overline{P}$ is trace-preserving as $\mathrm{Tr}(P) = \mathrm{Tr}(\overline{P})$. On the other hand, this example shows that $P_{la}, P_{ra}$ may not preserve the trace of $P$. We consider a two-point state space $\mathcal{X}=\{1,2\}$ and the two-element group $G=\{e, (12)\}$. Let
	\[
	P=\begin{pmatrix}
		a & b   \\[3pt]
		b   & a
	\end{pmatrix},
	\]
	along with $a,b \in [0,1]$, $a + b = 1$ and $b > a$. Clearly, $\pi$ is the discrete uniform on $\mathcal{X}$. Then $\mathrm{Tr}(P_{la}) = \mathrm{Tr}(P_{ra}) = a + b > 2a = \mathrm{Tr}(P)$.
\end{remark}

\begin{proof}
	We first prove item \eqref{it:inherit1}. We see that, for $A \in \mathcal{F}$,
	\begin{align*}
		\pi P_{da}(A) = \int_{G \times G} \int_{\mathcal{X}} \pi(dx) P(gx,hA) \nu(dg\, dh) = \int_{G \times G} \pi(A) \nu(dg\, dh) = \pi(A),
	\end{align*}
	where the second equality makes use of $\pi \in \mathcal{I}(G)$. 
	For $f,h \in L^2(\pi)$, we note that
	\begin{align*}
		\langle P_{la}[f],h \rangle_\pi &= \mathbb{E}_{g \sim \mu}(\langle U_g P[f],h \rangle_\pi)\\
		&= \mathbb{E}_{g \sim \mu}(\langle f, P^* U_g^{-1}[h] \rangle_\pi) \\
		&= \langle f,P_{ra}[h] \rangle_\pi,
	\end{align*}
	which yields $(P_{la})^* = P^*_{ra}$.
	
	Next, we prove item \eqref{it:inherit2}. For $f,h \in L^2(\pi)$, we see that
	\begin{align*}
		\langle P_{da}(G,\nu)[f],h \rangle_\pi &= \mathbb{E}_{(g,h) \sim \nu}(\langle U_g P U_h[f],h \rangle_\pi)\\
		&= \mathbb{E}_{(g,h) \sim \nu}(\langle f,U_{h^{-1}} P U_{g^{-1}}[h] \rangle_\pi) \\
		&= \mathbb{E}_{(h^{-1},g^{-1}) \sim \nu}(\langle f,U_{h^{-1}} P U_{g^{-1}}[h] \rangle_\pi) \\
		&= \langle f,P_{da}(G,\nu)[h] \rangle_\pi,
	\end{align*}
	where the second equality utilizes $\pi \in \mathcal{I}(G)$ and $P \in \mathcal{L}(\pi)$, and the third equality follows from the symmetry of $\nu$. In particular, the special cases $\overline{P}, \widetilde{P}, (P_{la})_{ra}$ all have symmetric $\nu$.
	
	To prove item \eqref{it:inherit5}, we see that $U_g P U_h$ is a compact operator as product of bounded operators (i.e. $U_g, U_h$) with compact operator (i.e. $P$) remain to be compact operators (see e.g. \cite[Proposition $4.2$]{Conway1990}). Therefore, $P_{da}(G,\nu)$, and hence $\overline{P}, \widetilde{P}, P_{la}, P_{ra}, (P_{la})_{ra}$, are compact operators as the set of compact operators is closed under finite linear combination (see e.g. \cite[paragraph after Theorem $8.1-3$]{K89}).
		
	We move on to prove item \eqref{it:inherit3}. For $g \in G$, we see that
	\begin{align*}
		\mathrm{Tr}(U_g P U_g^{-1}) = \sum_{e \in \mathcal{B}} \langle P[U_g^{-1} e],U_g^{-1}[e] \rangle_\pi = \mathrm{Tr}(P).
	\end{align*}
	Summing up over the Haar measure $\mu$ as $G$ is finite, we see that
	\begin{align*}
		\mathrm{Tr}(\overline{P}) = \mathrm{Tr}(\mathbb{E}_{g \sim \mu}(U_g P U_g^{-1})) = \mathbb{E}_{g \sim \mu} \left(\sum_{e \in \mathcal{B}} \langle P[U_g^{-1} e],U_g^{-1}[e] \rangle_\pi\right) = \mathbb{E}_{g \sim \mu}(\mathrm{Tr}(P)) = \mathrm{Tr}(P).
	\end{align*}
	If $P \in \mathcal{L}(\pi)$ is a non-negative $L^2(\pi)$-self-adjoint operator, then by spectral theorem so does $\overline{P}$, and hence $\overline{P}$ is trace-class if and only if $\mathrm{Tr}(\overline{P}) = \mathrm{Tr}(P) < \infty$, which is true.
	
	Finally, we prove item \eqref{it:inherit4}. We consider
	\begin{align*}
		\norm{\overline{P}}_{\mathrm{HS}}^2 &= \sum_{f,e \in \mathcal{B}} \bigg|\int_G \langle U_g^{-1}[f],PU_g^{-1} [e]\rangle_\pi \mu(dg) \bigg|^2 \\
		&\leq \int_G  \sum_{f,e \in \mathcal{B}} \bigg| \langle U_g^{-1}[f],PU_g^{-1} [e]\rangle_\pi\bigg|^2 \mu(dg)  \\
		&= \int_G \norm{P}_{\mathrm{HS}}^2 \mu(dg) = \norm{P}_{\mathrm{HS}}^2,
	\end{align*}
	where the inequality follows from the Jensen's inequality. Similarly, we compute that
	\begin{align*}
		\norm{P_{la}}_{\mathrm{HS}}^2 &= \sum_{f,e \in \mathcal{B}} \bigg|\int_G \langle U_g^{-1}[f],P [e]\rangle_\pi \mu(dg) \bigg|^2 \\
		&\leq \int_G  \sum_{f,e \in \mathcal{B}} \bigg| \langle U_g^{-1}[f],P[e]\rangle_\pi\bigg|^2 \mu(dg)  \\
		&\leq \int_G \norm{P}_{\mathrm{HS}}^2 \mu(dg) = \norm{P}_{\mathrm{HS}}^2,
        \end{align*}
        and
        \begin{align*}
		\norm{P_{ra}}_{\mathrm{HS}}^2 &= \sum_{f,e \in \mathcal{B}} \bigg|\int_G \langle P^*[f], U_g [e]\rangle_\pi \mu(dg) \bigg|^2 \\
		&\leq \int_G  \sum_{f,e \in \mathcal{B}} \bigg| \langle P^*[f],U_g[e]\rangle_\pi\bigg|^2 \mu(dg)  \\
		&\leq \int_G \norm{P^*}_{\mathrm{HS}}^2 \mu(dg) = \norm{P}_{\mathrm{HS}}^2,
	\end{align*}
	where we make use of the Jensen's and Cauchy-Schwartz inequality as well as $\norm{P}_{\mathrm{HS}} = \norm{P^*}_{\mathrm{HS}}$. If $G$ is a finite group, we first note that as $U_g$ is a bounded operator and $P$ is Hilbert-Schmidt, the product $U_g P U_h$ is a Hilbert-Schmidt operator \cite[Page $267$]{Conway1990}, and so does $P_{da}(G,\nu)$ as it is a finite mixture of Hilbert-Schmidt operators \cite[Theorem $VI.22$]{RS72}.
\end{proof}

\section{Improvement of $P_{da}$ over $P$}\label{sec:improvement}
In this section, we shall demonstrate that $P_{da}$ and its special cases $\overline P, \widetilde P, P_{la}, P_{ra}, (P_{la})_{ra}$ improve upon $P$ from the perspective of mixing time related parameters under suitable assumptions.

\subsection{Comparison of spectral gap}\label{subsec:spectralgap}

The (right) spectral gap of $P$ is defined as 
\begin{equation}\label{eq:rightspectralgap}
    \lambda=\lambda(P):=\inf\left\{\langle f, -L[f]\rangle_\pi: f\in L^2_0(\pi), \|f\|_\pi=1 \right\},
\end{equation}
where $L:=P-I$ is the generator of $P$. The spectral gap $\lambda(P)$ is the gap between $1$ and its second largest eigenvalue of additive reversiblization $\frac{P+P^*}{2}$. Spectral gap is of interest as it plays a key role in bounding the mixing times of $P$, and a larger spectral gap typically implies a smaller upper bound on the mixing time, especially for reversible Markov chains, see \citep{levin2017markov}.

We follow the setting in Proposition \ref{prop:inheritance of properties}, for the given group $G$, we define 
\begin{equation}\label{eq:G-invariant subspace}
    V=V(G):=\left\{f\in L^2(\pi): U_g[f]=f, \enspace\forall g\in G\right\}
\end{equation}
as the $G$-invariant subspace of $L^2(\pi)$, and $$V':=\left\{f\in V: \langle f, \mathbf{1}\rangle_\pi=0\right\}$$ as the subspace of $V$ orthogonal to the constant function. We define 
\begin{equation*}
    W=W(P):=\left\{f\in L^2(\pi): -\frac{L+L^*}{2}[f]=\lambda f\right\}
\end{equation*}
as the eigenspace corresponding to the spectral gap, then we can write $L^2(\pi)=W^\perp\oplus W$. Assume $W$ has a sequence of orthogonal basis functions $\{u_i\}_{i\in \mathcal{J}}$, where $\langle u_i,u_j\rangle_\pi=\delta_{ij}$. For any $f\in L^2(\pi)$, we denote $f_V, f_{V'}, f_W, f_{W^\perp}$ as the projections onto the respective subspaces. 

In the following theorem, we start with $\overline P$ a simple case of $P_{da}(G,\nu)$, and show that it indeed has a larger spectral gap than $P$ under mild conditions, with a (relatively) explicit improvement proposed. We also investigate the sufficient conditions such that such improvement is strict. 

\begin{theorem}\label{thm:spectral improvement of overline P}
    Assume that $P\in \mathcal{S}(\pi)$ (and hence $(1/2)(P+P^*)$) is a compact operator. Let $G$ be a locally compact topological group with Haar measure $\mu$ that acts on $\mathcal{X}$, and assume $\pi$ is $G$-invariant. Let $\lambda_2=\lambda_2(P)$ be the third smallest eigenvalue of $\frac{L+L^*}{2}$, which satisfies $\lambda_2>\lambda$. Let $\mathbf{P}_\Omega$ be the projection operator of $L^2(\pi)$ onto any subspace $\Omega$, then
        \begin{align*}
            \lambda(\overline P)\geq \min \bigg\{&\|\mathbf{P}_V\mathbf{P}_W\mathbf{P}_V\|_{2\to 2}\cdot \lambda(P)+\left(1-\|\mathbf{P}_V\mathbf{P}_W\mathbf{P}_V\|_{2\to 2}\right)\cdot \lambda_2(P), \\
            & \|\mathbf{P}_{V^\perp}\mathbf{P}_W\mathbf{P}_{V^\perp}\|_{2\to 2}\cdot \lambda(P)+\left(1-\|\mathbf{P}_{V^\perp}\mathbf{P}_W\mathbf{P}_{V^\perp}\|_{2\to 2}\right)\cdot \lambda_2(P)\bigg\}\geq \lambda(P),
        \end{align*}
    and $\lambda(\overline P)>\lambda(P)$ if $W\cap V=W\cap V^\perp=\{0\}$.
    
    Specifically, for the cases where $|\mathcal{J}|=1$ such that $W$ has only one basis function $u$ with $\|u\|_\pi=1$, we have
    \begin{align*}
        \lambda(\overline{P})\geq \min \bigg\{&\|u_V\|_\pi^2\cdot \lambda(P)+\left(1-\|u_V\|_\pi^2\right)\cdot \lambda_2(P),\\
        & \left(1-\|u_V\|_\pi^2\right)\cdot \lambda(P)+\|u_V\|_\pi^2\cdot \lambda_2(P)\bigg\}\geq \lambda(P),
    \end{align*}
    where $u_V=\mathbb E_{g\sim \mu}\left[U_g[u]\right]$. In this case, $\lambda(\overline P)>\lambda(P)$ holds as long as $u$ is not $G$-invariant and $u_V\neq 0$.
\end{theorem}
\begin{proof}
    We first show that the projection of $f$ onto subspace $V$ can be written as $f_V=\mathbb E_{g\sim \mu}\left(U_g[f]\right)$. Let $f=f_V+f_{V^{\perp}}$ as the orthogonal decomposition, since $f_V$ is $G$-invariant, then for any $g\in G$, $U_g[f]=f_V+U_g[f_{V^\perp}]$, taking expectation yields
    \begin{equation*}
        \mathbb E_{g\sim \mu}\left(U_g[f]\right)=f_V+\mathbb E_{g\sim \mu}\left(U_g[f_{V^\perp}]\right).
    \end{equation*}
    Observing that for any $\phi\in L^2(\pi)$, 
    \begin{equation*}
        \left\langle \mathbb E_{g\sim \mu}\left(U_g[f_{V^\perp}]\right), \phi\right\rangle_\pi=\left\langle f_{V^\perp}, \mathbb E_{g\sim \mu}\left(U_g^{-1}[\phi]\right)\right\rangle_\pi=0,
    \end{equation*}
    we get $f_V=\mathbb E_{g\sim \mu}\left(U_g[f]\right)$. 
    
    Let $\overline L=\mathbb E_{g\sim \mu}\left(U_gLU_g^{-1}\right)$, it can be readily verified that $\overline L(V)\subseteq V$, $\overline L(V^\perp)\subseteq V^\perp$ and $U_g(V^\perp)\subseteq V^\perp$ for any $g\in G$, hence for any $f\in L_0^2(\pi)$, $\|f\|_\pi=1$,
    \begin{align*}
        \left\langle -\overline L[f], f\right\rangle_\pi&=\left\langle-\overline L[f_V]-\overline{L}[f_{V^\perp}], f_V+f_{V^\perp}\right\rangle_\pi\\
        &=\left\langle -\overline{L}[f_V], f_V\right\rangle_\pi+\left\langle -\overline{L}[f_{V^\perp}], f_{V^\perp}\right\rangle_\pi\\
        &=\left\langle -L[f_V], f_V\right\rangle_\pi+\mathbb E_{g\sim\mu}\left(\langle-LU_g[f_{V^\perp}], U_g[f_{V^\perp}]\rangle_\pi\right),
    \end{align*}
    therefore, recalling that constant function is $G$-invariant, which means $V^\perp$ is orthogonal to $\mathbf{1}$, we have
    \begin{equation}\label{eq:lower bound lambda(bar P), two terms}
        \lambda(\overline P)\geq \min \left\{\inf_{f\in V',\|f\|_\pi=1}\langle -L[f],f\rangle_\pi, \inf_{f\in V^\perp, \|f\|_\pi=1}\langle -L[f],f\rangle_\pi\right\}.
    \end{equation}
    
    It suffices to lower bound the above two terms respectively. Recalling that $\langle -L[f],f\rangle_\pi=\left\langle -\frac{L+L*}{2}[f],f\right\rangle_\pi$, and that $\frac{L+L*}{2}$ preserves the space $W^\perp$, then for any $f\in V'$ with $\|f\|_\pi=1$, 
    \begin{align*}
        \langle -L[f],f\rangle_\pi&=\langle -L[f_W], f_W\rangle_\pi+\langle -L[f_{W^\perp}], f_{W^\perp}\rangle_\pi\\
        &\geq \lambda \|f_W\|_\pi^2+\lambda_2 \|f_{W^\perp}\|_\pi^2\\
        &=\lambda_2-(\lambda_2-\lambda)\cdot\|f_W\|_\pi^2,
    \end{align*}
    where we can decompose $f_W$ as 
    \begin{equation}\label{eq:decomposition of f_W}
        \|f_W\|_\pi^2=\sum_{i\in \mathcal{J}}\left(\langle f,u_i\rangle_\pi\right)^2=\sum_{i\in \mathcal{J}}\left(\langle f,(u_i)_{V}\rangle_\pi\right)^2.
    \end{equation}
    To give an upper bound of the above summation, we define an operator $T:L^2(\pi)\to L^2(\pi)$ such that
    \begin{equation*}
        T[\phi]:=\sum_{i\in \mathcal{J}}\langle \phi_V, (u_i)_V\rangle_\pi\cdot (u_i)_V, 
    \end{equation*}
    then we have 
    \begin{equation*}
        \|f_W\|_\pi^2=\langle f, T[f]\rangle_\pi.
    \end{equation*}
    Next, we observe that $T=\mathbf{P}_V\mathbf{P}_W\mathbf{P}_V$, and hence $T$ is self-adjoint. Actually, for any $\phi\in L^2(\pi)$, 
    \begin{align*}
        \mathbf{P}_V\mathbf{P}_W\mathbf{P}_V[\phi]&=\mathbf{P}_V\left[\sum_{i\in \mathcal{J}}\langle \mathbf{P}_V[\phi], u_i\rangle_\pi\cdot u_i\right]
        =\sum_{i\in \mathcal{J}}\langle \mathbf{P}_V[\phi], u_i\rangle_\pi\cdot \mathbf{P}_V[u_i]\\
        &=\sum_{i\in \mathcal{J}}\langle \phi_V, (u_i)_V\rangle_\pi\cdot (u_i)_V=T[\phi].
    \end{align*}
    Therefore, 
    \begin{equation*}
        \|f_W\|_\pi^2\leq \|T\|_{2\to 2}=\|\mathbf{P}_V\mathbf{P}_W\mathbf{P}_V\|_{2\to 2}\leq 1,
    \end{equation*}
    where the norm of projection operator is bounded by $1$, and hence 
    \begin{equation}\label{eq:spectral gap on V'}
        \inf_{f\in V',\|f\|_\pi=1}\langle -L[f],f\rangle_\pi\geq \|\mathbf{P}_V\mathbf{P}_W\mathbf{P}_V\|_{2\to 2}\cdot \lambda+\left(1-\|\mathbf{P}_V\mathbf{P}_W\mathbf{P}_V\|_{2\to 2}\right)\cdot \lambda_2\geq \lambda.
    \end{equation}
    For the second term in \eqref{eq:lower bound lambda(bar P), two terms}, we similarly have 
    \begin{equation*}
        \inf_{f\in V^\perp, \|f\|_\pi=1}\langle -L[f],f\rangle_\pi\geq \|\mathbf{P}_{V^\perp}\mathbf{P}_W\mathbf{P}_{V^\perp}\|_{2\to 2}\cdot \lambda+\left(1-\|\mathbf{P}_{V^\perp}\mathbf{P}_W\mathbf{P}_{V^\perp}\|_{2\to 2}\right)\cdot \lambda_2\geq \lambda.
    \end{equation*}
    Plugging into \eqref{eq:lower bound lambda(bar P), two terms}, we get the first inequality. Now, we define the cosine of two subspaces
    \begin{equation*}
        \alpha(V,W):=\sup_{\substack{\phi_1\in V, \phi_2\in W, \\\|\phi_1\|_\pi=\|\phi_2\|_\pi=1}} |\langle \phi_1, \phi_2\rangle_\pi|.
    \end{equation*}
    It is well known that 
    \begin{equation*}
        \alpha(V,W)=\sup_{\phi_1\in V, \|\phi_1\|_\pi=1}\|\mathbf{P}_W [\phi_1]\|_\pi=\sup_{\phi_2\in W, \|\phi_2\|_\pi=1}\|\mathbf{P}_V [\phi_2]\|_\pi,
    \end{equation*}
    and $\alpha(V,W)<1$ iff $V\cap W=\{0\}$. With the assumption of $V\cap W=\{0\}$, we have 
    \begin{align*}
        \|\mathbf{P}_V\mathbf{P}_W\mathbf{P}_V\|_{2\to 2}&\leq\sup_{\phi\in V, \|\phi\|_\pi=1}\|\mathbf{P}_V\mathbf{P}_W[\phi]\|_{\pi}\\
        &\leq \alpha(V,W)\cdot\sup_{\phi\in V, \|\phi\|_\pi=1}\|\mathbf{P}_W[\phi]\|_{\pi}\\
        &=\alpha^2(V,W)<1.
    \end{align*}
    Similarly, we can substitute $V^\perp$ to $V$ and get the first part of result. 

    If $|\mathcal{J}|=1$ and $W$ is expanded by $u$, we can rewrite \eqref{eq:decomposition of f_W} as 
    \begin{equation*}
        \|f_W\|_\pi^2=\langle f,u_V\rangle_\pi^2\leq \|u_V\|_\pi^2,
    \end{equation*}
    recalling that $\|u_V\|_\pi^2+\|u_{V^\perp}\|_\pi^2=1$, we get the rest of the result.
\end{proof}

For a non-reversible Markov kernel $P \in \mathcal{S}(\pi)$, the right spectral gap is useful for bounding the mixing time of its continuous-time (Poissonized) version, whereas for the discrete-time chain the second largest singular value is also of interests. We therefore introduce
\begin{equation*}
    \gamma=\gamma(P):=\lambda(\sqrt{PP^*}),
\end{equation*}
which is also referred to as multiplicative spectral gap of $P$ in the literature, while $\lambda(P)$ is called additive spectral gap. Up to constant factors, $\gamma(P)$ plays the same role in upper bounds on the mixing time of non-reversible chains as the usual spectral gap $\lambda(P)$ does for reversible ones, see \citep[Section 1.3]{montenegro2006mathematical}. Moreover, it is easy to see that 
\begin{align*}
    1-\gamma(P)&=\left\|P^*\right\|_{2\to 2}=\left\|P\right\|_{2\to 2}=1-\gamma(P^*)\\
    &=\left\|PP^*\right\|_{2\to 2}^{1/2}=\left\|P^*P\right\|_{2\to 2}^{1/2}.
\end{align*}
The motivation for studying the improvement of $\gamma$ in this article lies in the fact that $P_{la}$ and $P_{ra}$ (and many other cases of $P_{da}(G,\nu)$) are in general non-reversible, even if $P$ is reversible (recall Proposition \ref{prop:inheritance of properties}). Similar to $W$ the eigenspace of the additive spectral gap, we define
\begin{equation*}
    \widetilde{W}=\widetilde{W}(P):=\left\{f\in L^2(\pi): (I-\sqrt{PP^*})f=\gamma(P)f\right\}
\end{equation*}
as the eigenspace of the multiplicative spectral gap, and we also define $\widetilde{W}(P^*)$ in a similar way. 

In the following result, we proceed to study the general case of $P_{da}(G,\nu)$. We show that $\gamma$ is no smaller for $P_{da}(G,\nu)$ for any $\nu$ compared with $P$, and particularly for $P_{la}$, $P_{ra}$ and $(P_{la})_{ra}$ we give tighter bounds for such improvement. The proof is largely based on Theorem \ref{thm:spectral improvement of overline P}.

\begin{theorem}\label{thm:singular value improvement}
    Under the setting and notations in Theorem \ref{thm:spectral improvement of overline P}, we further define $\gamma_2(P)$ as the third smallest eigenvalue of $I-\sqrt{PP^*}$, and analogously for $\gamma_2(P^*)$, then the following statements hold.
    \begin{enumerate}[label=(\roman*)]
        \item For any $\nu\in \mathcal{P}(G\times G)$, $P_{da}(G,\nu)$ satisfies 
        \begin{equation*}
            \gamma(P_{da}(G,\nu))\geq \gamma(P), \quad
            \lambda(P_{da}(G,\nu))\geq \gamma(P).
        \end{equation*}

        \item Particularly, we have $\gamma((P_{la})_{ra})\geq \max\{\gamma (P_{la}), \gamma(P_{ra})\}$. Moreover, 
        \begin{equation*}
            \gamma(P_{la})\geq 1-\sqrt{\beta\left(1-\gamma(P)\right)^2+(1-\beta)\left(1-\gamma_2(P)\right)^2}\geq \gamma(P),
        \end{equation*}
        where $\beta:=\left\|\mathbf{P}_V\mathbf{P}_{\widetilde{W}(P)}\mathbf{P}_V\right\|_{2\to 2}$, and $\gamma(P_{la})>\gamma(P)$ if $\widetilde{W}(P)\cap V=\{0\}$. Similarly,
        \begin{equation*}
            \gamma(P_{ra})\geq 1-\sqrt{\beta'\left(1-\gamma(P)\right)^2+(1-\beta')\left(1-\gamma_2(P)\right)^2}\geq \gamma(P),
        \end{equation*}
        where $\beta':=\left\|\mathbf{P}_V\mathbf{P}_{\widetilde{W}(P^*)}\mathbf{P}_V\right\|_{2\to 2}$, and
        $\gamma(P_{ra})>\gamma(P)$ if $\widetilde{W}(P^*)\cap V=\{0\}$.
    \end{enumerate}
\end{theorem}
\begin{proof}
    For item (i), we have
    \begin{align*}
        1-\gamma(P_{da}(G,\nu))&=\left\|\mathbb E_{(g,h)\sim \nu}\left(U_g P U_h\right)\right\|_{2\to 2}\leq \mathbb E_{(g,h)\sim \nu}\left(\left\|U_g P U_h\right\|_{2\to 2}\right)\\
        &\leq \|P\|_{2\to 2}=1-\gamma(P),
    \end{align*}
    where we recall that $\|U_g\|_{2\to 2}=1$ for any $g\in G$. Moreover, the spectral gap of $P_{da}(G,\nu)$ equals to that of additive reversiblization, i.e.
    \begin{equation}\label{eq:additive reversiblization of P_da}
        K=\frac{1}{2}\left(\mathbb E_{(g,h)\sim \nu}\left(U_gPU_h\right)+\mathbb E_{(g,h)\sim \nu}\left(U_h^{-1}P^*U_g^{-1}\right)\right),
    \end{equation}
    which is reversible, and 
    \begin{equation*}
        \left\|K\right\|_{2\to 2}\leq \frac{1}{2}\left(\left\|P\right\|_{2\to 2}+\left\|P^*\right\|_{2\to 2}\right)=\left\|P\right\|_{2\to 2}.
    \end{equation*}
    
    For item (ii), we define $Q:=\mathbb E_{g\sim \mu}\left(U_g\right)$, then it is easy to check that $Q=\mathbf{P}_V$, and $Q$ is a reversible Markov kernel with $\pi Q=\pi$. Moreover, we can observe that 
    \begin{equation}\label{eq:QP, PQ, QPQ}
        P_{la}=QP, \quad P_{ra}=PQ, \quad (P_{la})_{ra}=QPQ,
    \end{equation}
    and the first inequality comes from item (i). This also implies 
    \begin{align}
        1-\gamma(P_{la})&=\left\|QP\right\|_{2\to 2}=\left\|QPP^*Q\right\|_{2\to 2}^{1/2},\label{eq:1-gamma(P_la)}\\
        1-\gamma(P_{ra})&=\left\|PQ\right\|_{2\to 2}=\left\|QP^*PQ\right\|_{2\to 2}^{1/2}.\label{eq:1-gamma(P_ra)}
    \end{align}
    For \eqref{eq:1-gamma(P_la)}, we have 
    \begin{align*}
        \left\|QPP^*Q\right\|_{2\to 2}&=\sup_{f\in L_0^2(\pi), \|f\|_\pi=1}\langle QPP^*Q[f], f\rangle_\pi=\sup_{f\in L_0^2(\pi), \|f\|_\pi=1}\langle PP^*Q[f], Q[f]\rangle_\pi\\
        &=\sup_{f\in V', \|f\|_\pi=1}\langle PP^*[f],f\rangle_\pi\\
        &\leq 1-\left(\beta \lambda(PP^*)+(1-\beta)\lambda_2(PP^*)\right)\\
        &=\beta \left(1-\gamma(P)\right)^2+(1-\beta)\left(1-\gamma_2(P)\right)^2,
    \end{align*}
    where the inequality comes from \eqref{eq:spectral gap on V'}, and we recall that $PP^*$ and $\sqrt{PP^*}$ share the same eigenspace corresponding to their second largest eigenvalue. The last equality uses the fact that $\lambda(PP^*)=1-(1-\gamma(P))^2$ and $\lambda_2(PP^*)=1-(1-\gamma_2(P))^2$. Now plugging into \eqref{eq:1-gamma(P_la)} we get the estimate for $\gamma(P_{la})$. The condition for equality to hold comes from Theorem \ref{thm:spectral improvement of overline P}. Applying the same argument for \eqref{eq:1-gamma(P_ra)}, we get the result. 
\end{proof}

Theorem \ref{thm:singular value improvement} item (i) shows that all $P_{da}$ demonstrate spectral improvement, which naturally raises the question: which averaging method yields the most substantial improvement? As a direct consequence of Theorem \ref{thm:singular value improvement}, we show that $(P_{la})_{ra}$ is the state-of-the-art in terms of offering the largest multiplicative spectral gap among all general-double-averages.

\begin{corollary}\label{cor:QPQ is the best}
    Under the setting and notations in Theorem \ref{thm:spectral improvement of overline P} and \ref{thm:singular value improvement}, for any $\nu\in \mathcal{P}(G\times G)$, we have 
    \begin{equation*}
        \gamma((P_{la})_{ra})\geq \gamma(P_{da}(G,\nu)).
    \end{equation*}
\end{corollary}
\begin{proof}
    \eqref{eq:Pdalara} shows that for any $\nu\in \mathcal{P}(G\times G)$, the independent-double-average is always equal to $(P_{la})_{ra}$, and the result comes from Theorem \ref{thm:singular value improvement} item (i).
\end{proof}

Corollary \ref{cor:QPQ is the best} demonstrates that $(P_{la})_{ra}$ is optimal when spectral improvement is the convergence assessment metric. Disregarding computational cost temporarily, one can consider using larger groups to achieve further enhancement of its spectral properties. Based on Theorem \ref{thm:singular value improvement}, in the following result we provide a justification for this intuition.
\begin{corollary}[Monotonicity of $\gamma$ with respect to group size]\label{cor:larger group is better}
    Assume $P\in \mathcal{S}(\pi)$. Let $G_1\leq G_2$ be two locally compact topological groups with Harr measure $\mu_1$ and $\mu_2$ respectively that act on $\mathcal{X}$. Assume $\pi$ is $G_2$-invariant (and hence $G_1$-invariant). For $i=1,2$, denote $(P_{la})_{ra}(G_i)$ as the independent-double-average of $P$ under group $G_i$. Then we have 
    \begin{equation*}
        \gamma((P_{la})_{ra}(G_2))\geq \gamma((P_{la})_{ra}(G_1)).
    \end{equation*}
\end{corollary}
\begin{proof}
    Let $V_1:=V(G_1)$ and $V_2:=V(G_2)$ be the two invariant subspaces induced by $G_1$ and $G_2$ respectively, recalling the definition in \eqref{eq:G-invariant subspace}, and hence $V_2\subseteq V_1$. From \eqref{eq:QP, PQ, QPQ}, we can write
    \begin{equation*}
        (P_{la})_{ra}(G_1)=\mathbf{P}_{V_1}P\mathbf{P}_{V_1}, \quad (P_{la})_{ra}(G_2)=\mathbf{P}_{V_2}P\mathbf{P}_{V_2},
    \end{equation*}
    therefore
    \begin{equation*}
        (P_{la})_{ra}(G_2)=\mathbf{P}_{V_2}(P_{la})_{ra}(G_1)\mathbf{P}_{V_2},
    \end{equation*}
    and the desired result follows.
\end{proof}

\subsection{Comparison of asymptotic variance}

Another common metric in assessing the convergence of ergodic Markov chains is the asymptotic variance. The asymptotic variance of $f \in L^2_0(\pi)$ with respect to $P$ is, for any initial distribution $\mu$,
\begin{align}
	v(f,P) &:= \lim_{n \to \infty} \dfrac{1}{n} \mathrm{Var}_\mu\left(\sum_{i=1}^n f(X_i)\right)\nonumber\\
    &=\|f\|_\pi^2+2\sum_{k=1}^\infty \left\langle P^k[f], f\right\rangle_\pi.\label{eq:asympvar power sum version}
\end{align}
A useful variational characterization of asymptotic variance for $P \in \mathcal{L}(\pi)$ \cite{S18} is given by
\begin{align}\label{eq:asympvvar}
	v(f,P) = \sup_{\phi \in L^2_0(\pi)} 4 \langle f,\phi \rangle_\pi - 2 \langle (I-P) [\phi], \phi \rangle_\pi - \langle f,f \rangle_\pi.
\end{align}
From this definition we readily check that, for $G$-invariant $\pi$ and $g \in G$,
\begin{align*}
	v(f,P) = v(U_g f, U_g P U_g^{-1}).
\end{align*} 

In the next result, we show that for any reversible $P_{da}(G,\nu)$, it can lead to an asymptotic variance that is no greater than that of $P$, under suitable assumptions. We also investigate situations where $v(f,P) = v(f, P_{da}(G,\nu))$ and the worst-case asymptotic variance, where we adapt and recall the notations as in Section \ref{subsec:spectralgap}. 
\begin{theorem}\label{thm:asympvar improvement}
    Let $P\in \mathcal{L}(\pi)$ be $\pi$-reversible. Let $G$ be a locally compact topological group with Haar measure $\mu$ that acts on $\mathcal{X}$, and assume that $\pi$ is $G$-invariant. Let $A:=-L$ for simplicity of presentation. We further assume that marginals of $\nu$ on both coordinates are the Harr measure $\mu$, that is, $g,h\sim\mu$ and $(g,h)\overset{D}{=}(h^{-1}, g^{-1})$ so that $P_{da}(G,\nu) \in \mathcal{L}(\pi)$ (recall Proposition \ref{prop:inheritance of properties}). The following statements hold:
    \begin{enumerate}[label=(\roman*)]
        \item If $f\in V'$, 
        \begin{equation*}
            v(f, P_{da}(G,\nu)) = v(f, P)-2\left\|\mathbf{P}_{A^{-1/2}V^\perp}A^{-1/2}[f]\right\|_\pi^2,
        \end{equation*}
        and $v(f, P_{da}(G,\nu))=v(f, P)$ iff $f\in AV'\cap V'$, where $AV':=\{A\phi: \phi\in V'\}$.

        \item Assume further that $P$ is compact. The worst-case asymptotic variance of $P_{da}(G,\nu)$ is at least no larger than that of $P$, that is,
        \begin{equation*}
            \sup_{f\in L_0^2(\pi), \|f\|_\pi=1} v(f, P_{da}(G,\nu))=\frac{2-\lambda(P_{da}(G,\nu))}{\lambda(P_{da}(G,\nu))}\leq \frac{2-\lambda(P)}{\lambda(P)}=\sup_{f\in L_0^2(\pi), \|f\|_\pi=1} v(f, P).
        \end{equation*}
    \end{enumerate}
\end{theorem}
\begin{proof}
    Assume $f\in V'$. We first notice that for $\phi\in L_0^2(\pi)$, 
    \begin{align*}
        \left\langle P_{da}(G,\nu)[\phi_{V'}],\phi_{V^\perp}\right\rangle_\pi&=\left\langle\mathbb E_{(g,h)\sim \nu}\left(U_gPU_h\right)[\phi_{V'}], \phi_{V^\perp}\right\rangle_\pi\\
        &=\left\langle P[\phi_{V'}], \mathbb E_{g\sim \mu}\left(U_g\right)[\phi_{V^\perp}]\right\rangle_\pi\\
        &=0,
    \end{align*}
    hence
    \begin{align*}
        \left\langle P_{da}(G,\nu)[\phi], \phi\right\rangle_\pi&=\left\langle P_{da}(G,\nu)[\phi_{V'}], \phi_{V'}\right\rangle_\pi+\left\langle P_{da}(G,\nu)[\phi_{V^\perp}], \phi_{V^\perp}\right\rangle_\pi\\
        &=\left\langle P[\phi_{V'}], \phi_{V'}\right\rangle_\pi+\left\langle P_{da}(G,\nu)[\phi_{V^\perp}], \phi_{V^\perp}\right\rangle_\pi\\
        &\leq \left\langle P[\phi_{V'}], \phi_{V'}\right\rangle_\pi+\|\phi_{V^\perp}\|_\pi^2,
    \end{align*}
    where the equality holds iff $\phi\in V'$, since the spectrum of $P_{da}(G,\nu)$ is bounded away from 1, according to Theorem \ref{thm:singular value improvement}. Therefore, from \eqref{eq:asympvvar} we have 
    \begin{equation*}
        v(f, P_{da}(G,\nu))=\sup_{\phi\in V'} 4\langle f, \phi\rangle_\pi -2\langle -L[\phi], \phi\rangle_\pi-\langle f,f\rangle_\pi.
    \end{equation*}
    For $\phi\in L_0^2(\pi)$, define
    \begin{equation*}
        H(\phi):=2\langle f,\phi\rangle_\pi-\langle A[\phi], \phi\rangle_\pi, 
    \end{equation*}
    where $A=-L$ is positive on $L_0^2(\pi)$. Then we have
    \begin{align*}
        v(f,P)-v(f, P_{da}(G,\nu))&=2\left(\sup_{\phi\in L_0^2(\pi)}H(\phi)-\sup_{\phi\in V'}H(\phi)\right)\\
        &=:2\left(H_{\text{max}}-\overline H_{\text{max}}\right),
    \end{align*}
    and we denote $\phi_*\in L_0^2(\pi)$ and $\overline \phi_*\in V'$ as the unique maximum points to attain the corresponding maximal values of $H$. For any $v\in L_0^2(\pi)$, we have 
    \begin{equation*}
        \frac{d}{d\varepsilon}H(\phi+\varepsilon v)\bigg|_{\varepsilon=0}=2\langle f-A[\phi], v\rangle_\pi,
    \end{equation*}
    hence $\phi_*=A^{-1}[f]$, and 
    \begin{equation}\label{eq:H_max}
        H_{\text{max}}=\langle f, A^{-1}[f]\rangle_\pi=\|A^{-1/2}[f]\|_\pi^2.
    \end{equation}
    If we further constrain $v\in V'$, then $\overline \phi_*$ should satisfy $f-A\overline \phi_*\perp V'$. Next, we define the $A$-weighted metric on $L_0^2(\pi)$ as 
    \begin{equation*}
        \langle u,v\rangle_A:=\langle A[u],v\rangle_\pi, \quad \forall u,v\in L_0^2(\pi),
    \end{equation*}
    and denote $\|\cdot\|_A$ as the induced norm, $\perp_A$ as the induced orthogonal relationship, and $\mathbf{P}_\cdot^A$ as the induced projection operator. Then we have 
    \begin{equation*}
        f-A[\overline \phi_*]\perp V' \iff A^{-1}[f]-\overline \phi_*\perp_A V',
    \end{equation*}
    which implies
    \begin{equation*}
        \overline \phi_*=\mathbf{P}_{V'}^A A^{-1}[f].
    \end{equation*}
    Now, we show that $\mathbf{P}_{V'}^A=A^{-1/2}\mathbf{P}_{A^{1/2}V'}A^{1/2}$. Denote the RHS as $R$, then $R[v]=v$ for $v\in V'$, and $R^2=R$. Moreover, for any $\phi\in L_0^2(\pi)$ and $w\in V'$, 
    \begin{align*}
        \langle \phi-R[\phi], w\rangle_A&=\left\langle A[\phi]-A^{1/2}\mathbf{P}_{A^{1/2}V'}A^{1/2}[\phi], w\right\rangle_\pi\\
        &=\left\langle \left(I-\mathbf{P}_{A^{1/2}V'}\right)A^{1/2}[\phi], A^{1/2}[w]\right\rangle_\pi\\
        &=0.
    \end{align*}
    Therefore, $\overline \phi_*=A^{-1/2}\mathbf{P}_{A^{1/2}V'}A^{-1/2}[f]$, and we have 
    \begin{equation*}
        \overline H_{\text{max}}=\left\langle f, A^{-1/2}\mathbf{P}_{A^{1/2}V'}A^{-1/2}[f] \right\rangle_\pi=\left\|\mathbf{P}_{A^{1/2}V'}A^{-1/2}[f]\right\|_\pi^2,
    \end{equation*}
    comparing with \eqref{eq:H_max}, and recalling that $A^{1/2}V'\perp A^{-1/2}V^\perp$, we get the result.

    For the worst-case asymptotic variance, we use \citep[equation $(3)$]{S18}, and the result comes from taking $f$ to be the eigenfunction corresponding to the respective spectral gaps. A tighter inequality for $\overline P$ can be obtained as a corollary of Theorem \ref{thm:spectral improvement of overline P}.
\end{proof}

\begin{remark}\label{remark:asympvar QP=PQ=QPQ}
    It is easy to see that $\overline P, \widetilde{P}$ and $(P_{la})_{ra}$ satisfy the assumption that $g,h\sim \mu$ and $(g,h)\overset{D}{=}(h^{-1}, g^{-1})$. $P_{la}$ and $P_{ra}$ generally fail to satisfy due to their non-reversibility. However, it can still be shown that for any $f\in V'$, if $P$ is reversible, then
    \begin{equation}\label{eq:asympvar QP=PQ=QPQ}
        v(f, P_{la})=v(f, P_{ra})=v(f, (P_{la})_{ra}),
    \end{equation}
    which comes from an observation that for $k \in \mathbb{N}$,
    \begin{align*}
        \left\langle \left((P_{la})_{ra}\right)^k[f], f\right\rangle_\pi&=\left\langle (QPQ)^k[f], [f]\right\rangle_\pi\\
        &=\left\langle (QP)^kQ[f], [f]\right\rangle_\pi\\
        &=\left\langle (QP)^k[f], [f]\right\rangle_\pi=\left\langle (PQ)^k[f], [f]\right\rangle_\pi,
    \end{align*}
    where we have used the notations in \eqref{eq:QP, PQ, QPQ}, and combining with \eqref{eq:asympvar power sum version} yields \eqref{eq:asympvar QP=PQ=QPQ}.
\end{remark}

\begin{remark}
    If $f\notin V'$, then the asymptotic variance $v(f, P_{da}(G,\nu))$ may be worse. Here is a simple example on the state space $\mathcal{X}=\{1,2,3\}$ with uniform stationary distribution $\pi(i)=\frac{1}{3}$. Take the two-element group $G=\{e, (12)\}$. Let
    \begin{gather*}
        P=\begin{pmatrix}
              0.09 & 0.5   & 0.41\\[3pt]
              0.5   & 0.12 & 0.38\\[3pt]
              0.41 & 0.38 & 0.21
        \end{pmatrix},
     \qquad  
        \overline P=\widetilde P=\begin{pmatrix}
                        0.105 & 0.5   & 0.395\\[3pt]
                        0.5   & 0.105 & 0.395\\[3pt]
                        0.395 & 0.395 & 0.21
                    \end{pmatrix},\\
        (P_{la})_{ra}=\begin{pmatrix}
                0.3025 & 0.3025 & 0.395 \\[3pt]
                0.3025 & 0.3025 & 0.395 \\[3pt]
                0.395 & 0.395 & 0.21 \\[3pt]
        \end{pmatrix}
    \end{gather*}
    then $V'=\mathrm{span}\left\{(1,1,-2)^T\right\}$. Take $f=\left(1,-0.5,-0.5\right)^T$, we have 
    \begin{equation*}
        v(f,P)\approx 0.2353, \quad v(f, \overline P)=v(f, \widetilde{P})\approx 0.2486, \quad v(f, (P_{la})_{ra})\approx 0.4610
    \end{equation*}
    and in this case asymptotic variances increase.
\end{remark}

\subsection{Comparison of the Cheeger's constant}

In this subsection, we focus on comparing the Cheeger's constant between $P$ and $\overline{P}$. For $\mathcal{F}$-measurable set $A$, we write $\mathbf{1}_A$ to be the indicator function of the set $A$. The Cheeger's constant of $P \in \mathcal{L}(\pi)$ is defined to be
\begin{align}\label{eq:Cheegerdef}
	\Phi(P) := \inf_{A;~ 0 < \pi(A) \leq \frac{1}{2}} \dfrac{\langle (I-P)[\mathbf{1}_A], \mathbf{1}_A \rangle_\pi }{\pi(A)}.
\end{align}

Our result in this subsection demonstrates that the two reversible averages $\overline{P}$ and $(P_{la})_{ra}$ have the Cheeger's constant at least as large as that of $P$.

\begin{theorem}
	Let $G$ be a locally compact topological group with Haar measure $\mu$ that acts on $\mathcal{X}$, and assume that $\pi$ is $G$-invariant. For $P \in \mathcal{L}(\pi)$, we have
	\begin{align*}
		\Phi(\overline{P}) \geq \Phi(P).
	\end{align*}
    If we further assume $P$ is non-negative (i.e. $\langle Pf,f \rangle_{\pi} \geq 0$ for all $f \in L^2(\pi)$), then 
    \begin{equation*}
        \Phi((P_{la})_{ra})\geq \Phi(P).
    \end{equation*}
\end{theorem}

\begin{proof}
	For $\mathcal{F}$-measurable set $A$ with $0 < \pi(A) \leq \frac{1}{2}$, we first see that
	\begin{align*}
		\dfrac{\langle (I-\overline{P})[\mathbf{1}_A], \mathbf{1}_A \rangle_\pi }{\pi(A)} &= \int \dfrac{\langle (I-P)[\mathbf{1}_{gA}], \mathbf{1}_{gA} \rangle_\pi}{\pi(gA)} \, \mu(dg),
	\end{align*}
	where $\pi(gA) = \pi(A)$ follows from $G$-invariant $\pi$ and $U_g^{-1} \mathbf{1}_A = \mathbf{1}_{gA}$. Taking infimum over both sides with respect to the set $A$ and noting \eqref{eq:Cheegerdef} leads to
	\begin{align*}
		\Phi(\overline{P}) \geq \int \Phi(P) \, \mu(dg) = \Phi(P).
	\end{align*}
    
    Moreover, if $P$ is non-negative, it can be readily verified that the mapping $f\mapsto \langle Pf, f\rangle_\pi$ is convex in $f$. Recalling the notations in \eqref{eq:QP, PQ, QPQ}, we have 
    \begin{align*}
        \left\langle (P_{la})_{ra}\mathbf{1}_A, \mathbf{1}_A\right\rangle_\pi&=\left\langle QPQ\mathbf{1}_A, \mathbf{1}_A\right\rangle_\pi=\left\langle PQ\mathbf{1}_A, Q\mathbf{1}_A\right\rangle_\pi\\  
        &=\left\langle P\left[\mathbb E_{g\sim \mu}\left(\mathbf{1}_{gA}\right)\right], \mathbb E_{g\sim \mu}\left(\mathbf{1}_{gA}\right) \right\rangle_\pi\\
        &\leq \mathbb E_{g\sim \mu}\left(\left\langle P[\mathbf{1}_{gA}], \mathbf{1}_{gA}\right\rangle_\pi\right),
    \end{align*}
    and hence
    \begin{align*}
        \frac{\langle (I-(P_{la})_{ra})[\mathbf{1}_A], \mathbf{1}_A \rangle_\pi }{\pi(A)}\geq \frac{\mathbb E_{g\sim \mu}\left(\langle (I-P)[\mathbf{1}_A], \mathbf{1}_A\rangle_\pi\right)}{\pi(A)},
    \end{align*}
    Taking infimum over $A$ on both sides yields the result. 
\end{proof}

\section{Pythagorean identities, distance to isotropy and the group-induced averages as projections under the KL divergence}\label{sec:geometry}

The main aim of this section is to demonstrate that the group-induced averages $P_{da}$ can be understood as projections of $P$ under the $\pi$-weighted Kullback-Leibler (KL) divergence and suitable assumptions. This offers a geometric interpretation and justifies that the group-induced averages arise naturally. In addition, this allows us to define a notion of ``distance to isotropy" of a given Markov kernel $P$ on $\mathbb{R}^d$ under KL divergence and the group $G = \mathrm{SO}(d)$. This distance measures the KL divergence from the closest isotropic Markov kernel, $\overline{P}$, to $P$.

Recall that, for $P,M \in \mathcal{L}(\mathcal{X})$ and $\pi$ be a probability measure on $\mathcal{X}$, the \textbf{$\pi$-weighted Kullback-Leibler divergence} of \( P \) from \( M \), averaged over \( \pi \), is defined as
\begin{align*}
	D^\pi_{KL}(P \| M) := 
	\begin{cases}
		\displaystyle \int_\mathcal{X} \pi(dx) \int_\mathcal{X} P(x, dy) \log \left( \frac{dP(x, \cdot)}{dM(x, \cdot)}(y) \right), 
		& \text{if } P(x, \cdot) \ll M(x, \cdot) \text{ for } \pi\text{-a.e. } x, \\[1.5em]
		+\infty, & \text{otherwise.}
	\end{cases}
\end{align*}
Here, \( \frac{dP(x, \cdot)}{dM(x, \cdot)} \) denotes the Radon-Nikodym derivative of \( P(x, \cdot) \) with respect to \( M(x, \cdot) \), defined for \( \pi \)-almost every \( x \in \mathcal{X} \). When $\mathcal{X}$ is a finite state space, the $\pi$-weighted KL divergence of $P$ from $M$ is given by
\begin{align*}
	D^\pi_{KL}(P \| M) := \sum_{x,y \in \mathcal{X}} \pi(x) P(x,y) \log \left( \dfrac{P(x,y)}{M(x,y)} \right),
\end{align*}
where the usual convention of $0 \log \frac{0}{a} := 0$ applies for $a \in [0,1]$.
\begin{theorem}[Bisection properties]\label{thm:bisection}
	Let $G$ be a locally compact topological group with Haar measure $\mu$ that acts on $\mathcal{X}$, and $\pi$ is assumed to be $G$-invariant. Under Assumption \ref{assum:pi and G}, assume $P$ and $M$ admit a transition density w.r.t. $\mathfrak{m}$ at any starting state $x$. Under these assumptions, we have, for $g,h \in G$,
	\begin{align*}
		D^\pi_{KL}(P \| M) &= D^\pi_{KL}(U_g P U_h \| U_g M U_h).
	\end{align*}
\end{theorem}
\begin{proof}
    According to Assumption \ref{assum:pi and G} that $\dfrac{d\mathfrak{m}\circ g^{-1}}{d\mathfrak{m}}=1$ for any $g\in G$, we have 
	\begin{align*}
		D^\pi_{KL}(P \| M) &= \int_\mathcal{X\times \mathcal{X}} \pi(x) P(x, y) \log \left( \frac{P(x, y)}{M(x, y)} \right)\mathfrak{m}(dx)\mathfrak{m}(dy) \\
		&= \int_\mathcal{X\times \mathcal{X}}  \pi(x) P(gx, h^{-1}y) \log \left( \frac{P(gx, h^{-1}y)}{M(gx, h^{-1}y)} \right)\mathfrak{m}(dx)\mathfrak{m}(dy) \\
		&= D^\pi_{KL}(U_g P U_h \| U_g M U_h),
	\end{align*}
    then the result follows.
\end{proof}

Making use of Theorem \ref{thm:bisection}, we establish the Pythagorean identities under KL divergence.

\begin{theorem}[Pythagorean identity under KL divergence]\label{thm:pythKL}
	Assume that $\pi,G,P,M,\mathcal{X}$ satisfy the assumptions as stated in Theorem \ref{thm:bisection}, and $M \in \mathcal{D}(G,\nu)$. Assume that $P_{da}(G,\nu) \in \mathcal{D}(G,\nu)$. We have
	\begin{align*}
		D^\pi_{KL}(P \| M) &= D^\pi_{KL}(P \| P_{da}) + D^{\pi}_{KL}(P_{da} \| M).
	\end{align*}
	In particular, assuming $\pi$ is absolutely continuous with respect to the Lebesgue measure so that we take $M = \Pi$, we see that
	\begin{align*}
		D^\pi_{KL}(P \| \Pi) \geq D^{\pi}_{KL}(P_{da} \| \Pi).
	\end{align*}
\end{theorem}
\begin{proof}
    Using Theorem \ref{thm:bisection} and $M \in \mathcal{D}(G,\nu)$, we see that
	\begin{align*}
		D^\pi_{KL}(P \| M) &= \int_{G \times G} D^\pi_{KL}(U_g P U_h \| M) \, \nu(dg dh) \\
		&= \int_{G \times G} D^\pi_{KL}(U_g P U_h \| P_{da}) \, \nu(dg dh) + D^\pi_{KL}(P_{da} \| M) \\
		&\quad + \int_{G \times G} \int_\mathcal{X\times \mathcal{X}} \pi(x) (P(gx,h^{-1}y)- P_{da}(x,y)) \log \left(\dfrac{P_{da}(x,y)}{M(x,y)}\right) \mathfrak{m}(dx)\mathfrak{m}(dy) \nu(dgdh) \\
		&= \int_{G \times G} D^\pi_{KL}(P \| P_{da}) \, \nu(dg dh) + D^{\pi}_{KL}(P_{da} \| M) + 0 \\
		&= D^\pi_{KL}(P \| P_{da}) + D^{\pi}_{KL}(P_{da} \| M)
	\end{align*}
	where in the third equality we make use of $P_{da} \in \mathcal{D}(G,\nu)$ and Theorem \ref{thm:bisection}, and in addition the triple integral vanishes by interchanging the order of integration. 
\end{proof}

By recalling that $\overline{P}, \widetilde{P}, P_{la}, P_{ra}, (P_{la})_{ra}$ are special cases of $P_{da}$, we arrive at the following corollary in view of Proposition \ref{prop:Abelian} and Theorem \ref{thm:pythKL}:

\begin{corollary}[Pythagorean identities under KL divergence]\label{cor:pythKL}
	Assume that $\pi,\nu,G,P,M,\mathcal{X}$ satisfy the assumptions as stated in Theorem \ref{thm:bisection}. We have
	\begin{align*}
		D^\pi_{KL}(P \| M) &= D^\pi_{KL}(P \| \overline{P}) + D^{\pi}_{KL}(\overline{P} \| M), \quad M \in \mathcal{L}(G,G^{-1}), \\
		D^\pi_{KL}(P \| M) &= D^\pi_{KL}(P \| P_{la}) + D^{\pi}_{KL}(P_{la} \| M), \quad M \in \mathcal{LI}(G), \\
		D^\pi_{KL}(P \| M) &= D^\pi_{KL}(P \| P_{ra}) + D^{\pi}_{KL}(P_{ra} \| M), \quad M \in \mathcal{RI}(G), \\
		D^\pi_{KL}(P \| M) &= D^\pi_{KL}(P \| (P_{la})_{ra}) + D^{\pi}_{KL}((P_{la})_{ra} \| M), \quad M \in \mathcal{LI}(G) \cap \mathcal{RI}(G).
	\end{align*}
	Using the last equality above and by replacing $P$ with $P_{da}$, we note that, in view of \eqref{eq:Pdalara},
	\begin{align*}
		D^\pi_{KL}(P_{da} \| M) &= D^\pi_{KL}(P_{da} \| (P_{la})_{ra}) + D^{\pi}_{KL}((P_{la})_{ra} \| M), \quad M \in \mathcal{LI}(G) \cap \mathcal{RI}(G).
	\end{align*}
	If $G$ is further assumed to be Abelian, then
	\begin{align*}
		D^\pi_{KL}(P \| M) &= D^\pi_{KL}(P \| \widetilde{P}) + D^{\pi}_{KL}(\widetilde{P} \| M), \quad M \in \mathcal{L}(G,G).
	\end{align*}
\end{corollary}

We discuss several interesting consequences of Theorem \ref{thm:pythKL}. First, if $P \in \mathcal{S}(\pi)$ (and hence $P_{da}$ by Proposition \ref{prop:inheritance of properties}) are $\pi$-stationary, then we see that the group-induced averages $\overline{P}, \widetilde{P}, P_{la}, P_{ra}, (P_{la})_{ra}$ are at least closer to $\Pi$ than that of $P$ when measured by the KL divergence under suitable assumptions. This is similar to results presented in Section \ref{sec:improvement}, in which we can understand these inequalities as rearrangement or data-processing inequalities in this context. In view of this, it is therefore advantageous to consider these group-induced averages over the original $P$ as candidate MCMC samplers to approximately sample from $\pi$.

A natural question thus arises: among the general-double-averages $P_{da}$ and the specific cases $\overline{P}, \widetilde{P}, P_{la}, P_{ra}, (P_{la})_{ra}$, which one is the closest to $\Pi$ based upon KL divergence? By applying the Pythagorean identities in Corollary \ref{cor:pythKL}, we note that $(P_{la})_{ra}$ is the closest one:
\begin{corollary}[$(P_{la})_{ra}$ as the closest Markov kernel]\label{cor:PlaracloseKL}
	Assume that $\pi,G,P,M,\mathcal{X}$ satisfy the assumptions as stated in Theorem \ref{thm:bisection}. We have
	\begin{align*}
		D^{\pi}_{KL}(P_{da} \| \Pi) &\geq D^{\pi}_{KL}((P_{la})_{ra} \| \Pi).
	\end{align*}
	In particular,
	\begin{align*}
		D^{\pi}_{KL}(\overline{P} \| \Pi) &\geq D^{\pi}_{KL}((P_{la})_{ra} \| \Pi), \quad D^{\pi}_{KL}(\widetilde{P} \| \Pi) \geq D^{\pi}_{KL}((P_{la})_{ra} \| \Pi), \\
		D^{\pi}_{KL}(P_{la} \| \Pi) &\geq D^{\pi}_{KL}((P_{la})_{ra} \| \Pi), \quad D^{\pi}_{KL}(P_{ra} \| \Pi) \geq D^{\pi}_{KL}((P_{la})_{ra} \| \Pi).
	\end{align*}
\end{corollary}

Another consequence concerns the special case of $G = \mathrm{SO}(d)$ and $\mathcal{X} = \mathbb{R}^d$, in which it follows from Theorem \ref{thm:pythKL} that the unique projection of $P$ onto $\mathcal{L}(G,G^{-1})$ is given by $\overline{P}$. The set $\mathcal{L}(G,G^{-1})$ can be interpreted as the set of Markov kernels that are isotropic under $G$, and hence the KL divergence $D^\pi_{KL}(P \| \overline{P})$ can be understood as the \textbf{distance to isotropy} of $P$. 

If one further assumes that $G$ is Abelian so that $\widetilde{P} \in \mathcal{L}(G,G)$ by Proposition \ref{prop:Abelian}, a similar interpretation holds for $\widetilde{P}$ being the unique projection of $P$ onto $\mathcal{L}(G,G)$, and the KL divergence $D^\pi_{KL}(P \| \widetilde{P})$ can be interpreted as the \textbf{distance to the set of $(G,G)$-invariant Markov kernels} of $P$.

\subsection{Projections under the Hilbert-Schmidt and Frobenius norm}

Apart from the KL divergence investigated in the previous section, in this subsection we shall consider projections under the Hilbert-Schmidt (HS) norm for HS operators and the Frobenius norm in the finite state space setting. Recall that, for two HS operators $P,M$ on $L^2(\pi)$ and two matrices $P,M$, we define the HS inner product and Frobenius inner product to be respectively
\begin{align*}
	\langle P,M \rangle_{\mathrm{HS}} &:= \mathrm{Tr}(P^*M), \\
	\langle P,M \rangle_{\mathrm{F}} &:= \mathrm{Tr}(P^TM), \quad \norm{P}_{\mathrm{F}}^2 = \mathrm{Tr}(P^TP) = \sum_{x,y \in \mathcal{X}} P(x,y)^2.
\end{align*}

With these notations in mind, the main result in this subsection gives Pythagorean identities under squared-HS and squared-Frobenius norm, thereby offering natural geometric interpretations of the group-induced averages $P_{da}$.

\begin{theorem}[Pythagorean identities under squared-HS and squared-Frobenius norm]\label{thm:pythHS}
	Let $G$ be a finite group with Haar measure $\mu$ that acts on $\mathcal{X}$ and $\pi$ is assumed to be $G$-invariant. Assume that $P,M$ (and hence $P_{da}$ by Proposition \ref{prop:inheritance of properties}) are HS operators and $M,P_{da} \in \mathcal{D}(G,\nu)$, where the measure $\nu$ satisfies $(g,h) \overset{D}{=} (g^{-1},h^{-1}) \sim \nu$. We have
	\begin{align*}
		\norm{P-M}_{\mathrm{HS}}^2 &= \norm{P-P_{da}}_{\mathrm{HS}}^2 + \norm{P_{da}-M}_{\mathrm{HS}}^2.
	\end{align*}
	If $\mathcal{X}$ is finite, we also have
	\begin{align*}
		\norm{P-M}_{\mathrm{F}}^2 &= \norm{P-P_{da}}_{\mathrm{F}}^2 + \norm{P_{da}-M}_{\mathrm{F}}^2.
	\end{align*}
\end{theorem}

\begin{proof}
	First, we decompose
	\begin{align*}
		\norm{P-M}_{\mathrm{HS}}^2 &= \mathrm{Tr}((P-P_{da}+P_{da}-M)^*(P-P_{da}+P_{da}-M)) \\
		&= \norm{P-P_{da}}_{\mathrm{HS}}^2 + \norm{P_{da}-M}_{\mathrm{HS}}^2 + 2 \mathrm{Tr}((P-P_{da})^*(P_{da}-M)),
	\end{align*}
	and hence it suffices to show that the trace of the rightmost term is zero. Using that $U_g^* = U_g^{-1}$ and the cyclic property of trace, we consider, for $g,h \in G$,
	\begin{align*}
	\mathrm{Tr}((U_g P U_h- U_g P_{da} U_h)^*(P_{da}-M)) 
		&= \mathrm{Tr}((P-P_{da})^*U_{g^{-1}}(P_{da}-M)U_{h^{-1}}) \\
		&= 	\mathrm{Tr}((P-P_{da})^*(U_{g^{-1}}P_{da}U_{h^{-1}} - U_{g^{-1}}MU_{h^{-1}})).
	\end{align*} 
	Summing over the Haar measure $\mu$ as $G$ is finite and by the linearity of the trace, it leads to 
	\begin{align*}
		0 = \mathrm{Tr}((P_{da}-P_{da})^*(\overline{P}-M))
		&= \mathrm{Tr}((P-P_{da})^*(P_{da}-M)),
	\end{align*} 
	as desired, where we make use of $(g,h) \overset{D}{=} (g^{-1},h^{-1}) \sim \nu$  and $P_{da},M \in \mathcal{D}(G,\nu)$.
	
	We proceed to consider the finite state space case, which is similar to the considerations above, except we now consider transpose instead of adjoint. Precisely, we note that
	\begin{align*}
		\norm{P-M}_{\mathrm{F}}^2 &= \mathrm{Tr}((P-P_{da}+P_{da}-M)^T(P-P_{da}+P_{da}-M)) \\
		&= \norm{P-P_{da}}_{\mathrm{F}}^2 + \norm{P_{da}-M}_{\mathrm{F}}^2 + 2 \mathrm{Tr}((P-P_{da})^T(P_{da}-M)),
	\end{align*}
	and hence it suffices to show that the trace of the rightmost term is zero. Using that $U_g^T = U_g^{-1}$ and the cyclic property of trace, we consider, for $g,h \in G$,
	\begin{align*}
		\mathrm{Tr}((U_g P U_h- U_g P_{da} U_h)^T(P_{da}-M)) 
		&= \mathrm{Tr}((P-P_{da})^TU_{g^{-1}}(P_{da}-M)U_{h^{-1}}) \\
		&= 	\mathrm{Tr}((P-P_{da})^T(U_{g^{-1}}P_{da}U_{h^{-1}} - U_{g^{-1}}MU_{h^{-1}})).
	\end{align*} 
	Summing over the Haar measure $\mu$ as $G$ is finite and by the linearity of the trace, it leads to 
	\begin{align*}
		0 = \mathrm{Tr}((P_{da}-P_{da})^T(\overline{P}-M))
		&= \mathrm{Tr}((P-P_{da})^T(P_{da}-M)),
	\end{align*} 
	as desired, where we make use of $(g,h) \overset{D}{=} (g^{-1},h^{-1}) \sim \nu$ and $P_{da},M \in \mathcal{D}(G,\nu)$.
\end{proof}

By recalling that $\overline{P}, \widetilde{P}, P_{la}, P_{ra}, (P_{la})_{ra}$ are special cases of $P_{da}$ and noting that $(g,h) \overset{D}{=} (g^{-1},h^{-1}) \sim \nu$ in these averages, we apply Theorem \ref{thm:pythHS} to obtain the following two corollaries. This is analogous to Corollary \ref{cor:pythKL} and \ref{cor:PlaracloseKL}, and demonstrates that $(P_{la})_{ra}$ is the closest to $\Pi$ among these averages under HS and Frobenius norm.

\begin{corollary}[Pythagorean identities under squared-HS and squared-Frobenius norm]\label{cor:pythHS}
	Assume that $\pi,\nu,G,P,M,\mathcal{X}$ satisfy the assumptions as stated in Theorem \ref{thm:pythHS}. We have
	\begin{align*}
		\norm{P-M}_{\mathrm{HS}}^2 &= \norm{P-\overline{P}}_{\mathrm{HS}}^2 + \norm{\overline{P}-M}_{\mathrm{HS}}^2, \quad M \in \mathcal{L}(G,G^{-1}), \\
		\norm{P-M}_{\mathrm{HS}}^2 &= \norm{P-\widetilde{P}}_{\mathrm{HS}}^2 + \norm{\widetilde{P}-M}_{\mathrm{HS}}^2, \quad M \in \mathcal{L}(G,G), \\
		\norm{P-M}_{\mathrm{HS}}^2 &= \norm{P-P_{la}}_{\mathrm{HS}}^2 + \norm{P_{la}-M}_{\mathrm{HS}}^2, \quad M \in \mathcal{LI}(G), \\
		\norm{P-M}_{\mathrm{HS}}^2 &= \norm{P-P_{ra}}_{\mathrm{HS}}^2 + \norm{P_{ra}-M}_{\mathrm{HS}}^2, \quad M \in \mathcal{RI}(G), \\
		\norm{P-M}_{\mathrm{HS}}^2 &= \norm{P-(P_{la})_{ra}}_{\mathrm{HS}}^2 + \norm{(P_{la})_{ra}-M}_{\mathrm{HS}}^2, \quad M \in \mathcal{LI}(G) \cap \mathcal{RI}(G), \\
		\norm{P_{da}-M}_{\mathrm{HS}}^2 &= \norm{P_{da}-(P_{la})_{ra}}_{\mathrm{HS}}^2 + \norm{(P_{la})_{ra}-M}_{\mathrm{HS}}^2, \quad M \in \mathcal{LI}(G) \cap \mathcal{RI}(G).
	\end{align*}
	If $\mathcal{X}$ is finite, then we also have
	\begin{align*}
		\norm{P-M}_{F}^2 &= \norm{P-\overline{P}}_{F}^2 + \norm{\overline{P}-M}_{F}^2, \quad M \in \mathcal{L}(G,G^{-1}), \\
		\norm{P-M}_{F}^2 &= \norm{P-\widetilde{P}}_{F}^2 + \norm{\widetilde{P}-M}_{F}^2, \quad M \in \mathcal{L}(G,G), \\
		\norm{P-M}_{F}^2 &= \norm{P-P_{la}}_{F}^2 + \norm{P_{la}-M}_{F}^2, \quad M \in \mathcal{LI}(G), \\
		\norm{P-M}_{F}^2 &= \norm{P-P_{ra}}_{F}^2 + \norm{P_{ra}-M}_{F}^2, \quad M \in \mathcal{RI}(G), \\
		\norm{P-M}_{F}^2 &= \norm{P-(P_{la})_{ra}}_{F}^2 + \norm{(P_{la})_{ra}-M}_{F}^2, \quad M \in \mathcal{LI}(G) \cap \mathcal{RI}(G), \\
		\norm{P_{da}-M}_{F}^2 &= \norm{P_{da}-(P_{la})_{ra}}_{F}^2 + \norm{(P_{la})_{ra}-M}_{F}^2, \quad M \in \mathcal{LI}(G) \cap \mathcal{RI}(G).
	\end{align*}
\end{corollary}

\begin{corollary}[$(P_{la})_{ra}$ as the closest Markov kernel]
	Assume that $\pi,\nu,G,P,\mathcal{X}$ satisfy the assumptions as stated in Theorem \ref{thm:pythHS}. We have
	\begin{align*}
		\norm{P_{da}-\Pi}_{\mathrm{HS}} &\geq \norm{(P_{la})_{ra}-\Pi}_{\mathrm{HS}}.
	\end{align*}
	In particular,
	\begin{align*}
		\norm{\overline{P}-\Pi}_{\mathrm{HS}} &\geq \norm{(P_{la})_{ra}-\Pi}_{\mathrm{HS}}, \quad \norm{\widetilde{P}-\Pi}_{\mathrm{HS}} \geq \norm{(P_{la})_{ra}-\Pi}_{\mathrm{HS}}, \\
		\norm{P_{la}-\Pi}_{\mathrm{HS}} &\geq \norm{(P_{la})_{ra}-\Pi}_{\mathrm{HS}}, \quad \norm{P_{ra}-\Pi}_{\mathrm{HS}} \geq \norm{(P_{la})_{ra}-\Pi}_{\mathrm{HS}}. 
	\end{align*}
	If $\mathcal{X}$ is finite, then we also have
	\begin{align*}
		\norm{P_{da}-\Pi}_{F} &\geq \norm{(P_{la})_{ra}-\Pi}_{F}.
	\end{align*}
	In particular,
	\begin{align*}
		\norm{\overline{P}-\Pi}_{F} &\geq \norm{(P_{la})_{ra}-\Pi}_{F}, \quad \norm{\widetilde{P}-\Pi}_{F} \geq \norm{(P_{la})_{ra}-\Pi}_{F}, \\
		\norm{P_{la}-\Pi}_{F} &\geq \norm{(P_{la})_{ra}-\Pi}_{F}, \quad \norm{P_{ra}-\Pi}_{F} \geq \norm{(P_{la})_{ra}-\Pi}_{F}. 
	\end{align*}
\end{corollary}

\section{Mixing time comparison between $P_{la}, P_{ra}, (P_{la})_{ra}$}\label{sec:comparison}

From Section \ref{sec:improvement} and \ref{sec:geometry}, $(P_{la})_{ra}$ can be understood as the optimal chain among all double-averages, from both spectral and geometrical perspectives. In this section, we show that the mixing times of $P_{la}$ and $P_{ra}$ are nearly identical to that of $(P_{la})_{ra}$. This equivalence allows us to adopt $P_{la}$ and $P_{ra}$ in practice, achieving comparable mixing times at reduced computational cost. For instance, if we compare the Markov kernels $(P_{la})_{ra}$ and $P_{la}$, at each iteration the former needs to conduct both left and right averaging while only left averaging is needed for the latter case. In this sense, $P_{la}$ or $P_{ra}$ has a reduced computational cost per iteration when compared with $(P_{la})_{ra}$.

We shall use the $L^p$ distance to quantify the mixing times, which is defined as follows. For $1\leq p< \infty$, let $\|f\|_{p,\pi}:=\left(\int |f|^p d\pi\right)^{1/p}$ be the $L^p$ norm of $f$ under $\pi$, and define $\|f\|_{\infty,\pi}:=\lim_{p\to \infty} \|f\|_{p,\pi}$. For a Markov kernel $P$ on state space $\mathcal{X}$ with stationary distribution $\pi$, its worst-case $L^p$ distance to $\pi$ at time $t\in \mathbb N$ is defined as
\begin{equation*}
	d_p(P, t):=\pi\text{-}\mathop{\esssup}\limits_{x\in\mathcal{X}}\left\|\frac{d P^t(x,\cdot)}{d \pi}-1\right\|_{p,\pi}, \quad 1\leq p\leq \infty,
\end{equation*}
and the corresponding mixing time is 
\begin{equation*}
    t_{\mathrm{mix},p}(P,\varepsilon):=\inf\left\{t\in \mathbb N: d_p(P,t)\leq \varepsilon\right\}, \quad 1\leq p\leq \infty, \enspace \varepsilon>0.
\end{equation*}
For $p=1$, it covers the classical worst-case total variation (TV) mixing time up to a universal constant. If $P^t(x,\cdot)$ is not absolutely continuous w.r.t. $\pi$, set $d_1(P,t)=2$ and $d_p(P,t)=\infty$ for $p>1$, and their corresponding mixing times are set to be $\infty$.

A useful characterization of $L^p$ mixing times is via operator norm, i.e. \citep{chen2008cutoff}: if $P(x,\cdot)$ admits a density w.r.t. $\pi$, then
\begin{equation}\label{eq:operator norm for mixing time}
    d_p(P,t)=\left\|P^t-\Pi\right\|_{L^q\to L^\infty},\quad \frac{1}{p}+\frac{1}{q}=1.
\end{equation}
Then, the main result of this section is presented as follows.
\begin{theorem}\label{thm:mixing time equivalence within QP,PQ,QPQ}
    Let $P\in \mathcal{S}(\pi)$, $G$ be a finite group acting on $\mathcal{X}$, and assume that $\pi$ is $G$-invariant. For any $1\leq p\leq \infty$ and $\varepsilon>0$, we have
    \begin{align}
        t_{\mathrm{mix},p}((P_{la})_{ra}, 2\varepsilon)&\leq t_{\mathrm{mix},p}(P_{la}, \varepsilon)\leq t_{\mathrm{mix},p}\left((P_{la})_{ra}, \frac{\varepsilon}{2}\right)+1,\label{eq:mixing time Pla and QPQ}\\
        t_{\mathrm{mix},p}((P_{la})_{ra}, 2\varepsilon)&\leq t_{\mathrm{mix},p}(P_{ra}, \varepsilon)\leq t_{\mathrm{mix},p}\left((P_{la})_{ra}, \frac{\varepsilon}{2}\right)+1.\label{eq:mixing time Pra and QPQ}
    \end{align}
\end{theorem}
\begin{proof}
    If $P$ is not absolutely continuous to $\pi$ at some point $x$, then $QP$, $PQ$ and $QPQ$ are likewise not, in which case the mixing times are all $\infty$. Suppose $\pi(A)=0$ and $P(x,A)>0$ for some set $A$, then the claim is given by 
    \begin{align*}
        QP(x,A)&\geq Q(x,x)P(x,A)=|G|^{-1}\cdot P(x,A)>0, \\
        PQ(x,A)&\geq P(x,A)\cdot|G|^{-1}>0,\\
        QPQ(x,A)&\geq Q(x,x)QP(x,A)>0.
    \end{align*}
    It suffices to consider the case that $P(x,\cdot)$ has a density w.r.t. $\pi$, thus three kernels above all admit a density. Following the notations in \eqref{eq:QP, PQ, QPQ}, for the left-hand-side in \eqref{eq:mixing time Pla and QPQ}, recalling \eqref{eq:operator norm for mixing time}, we have 
    \begin{align*}
        \left\|(QPQ)^t-\Pi\right\|_{L^q\to L^\infty}&=\left\|\left((QP)^t-\Pi\right)(Q-\Pi)\right\|_{L^q\to L^\infty}\\
        &\leq \left\|(QP)^t-\Pi\right\|_{L^q\to L^\infty}\cdot\left\|Q-\Pi\right\|_{L^q\to L^q}\\
        &\leq 2\left\|(QP)^t-\Pi\right\|_{L^q\to L^\infty},
    \end{align*}
    where in the last inequality, we have used the well-known fact that for any Markov operator $K$, 
    \begin{equation*}
        \left\|K\right\|_{L^1\to L^1}\leq 1, \quad \left\|K\right\|_{L^\infty\to L^\infty}\leq 1, 
    \end{equation*}
    and by Riesz-Thorin Interpolation Theorem \citep{stein2011functional}, if $1<q<\infty$, 
    \begin{equation*}
        \left\|K\right\|_{L^q\to L^q}\leq \left\|K\right\|_{L^1\to L^1}^{1/q}\cdot \left\|K\right\|_{L^\infty\to L^\infty}^{1-1/q}\leq 1.
    \end{equation*}
    As a result, we arrive at
    \begin{equation*}
        \left\|Q-\Pi\right\|_{L^q\to L^q}\leq \left\|Q\right\|_{L^q\to L^q}+\left\|\Pi\right\|_{L^q\to L^q}\leq 2.
    \end{equation*}
    For the right-hand-side of \eqref{eq:mixing time Pla and QPQ}, we similarly have 
    \begin{align*}
        \left\|(QP)^t-\Pi\right\|_{L^q\to L^\infty}&=\left\|\left((QPQ)^{t-1}-\Pi\right)(P-\Pi)\right\|_{L^q\to L^\infty}\\
        &\leq \left\|(QPQ)^{t-1}-\Pi\right\|_{L^q\to L^\infty}\cdot\left\|P-\Pi\right\|_{L^q\to L^q}\\
        &\leq 2\left\|(QPQ)^{t-1}-\Pi\right\|_{L^q\to L^\infty}.
    \end{align*}
    
    The proof for \eqref{eq:mixing time Pra and QPQ} is similar. Precisely, we see that 
    \begin{align*}
        \left\|(QPQ)^t-\Pi\right\|_{L^q\to L^\infty}&=\left\|(Q-\Pi)\left((PQ)^t-\Pi\right)\right\|_{L^q\to L^\infty}\\
        &\leq \left\|Q-\Pi\right\|_{L^\infty\to L^\infty}\cdot\left\|(PQ)^t-\Pi\right\|_{L^q\to L^\infty}\\
        &\leq 2\left\|(PQ)^t-\Pi\right\|_{L^q\to L^\infty},
    \end{align*}
    and 
    \begin{align*}
        \left\|(PQ)^t-\Pi\right\|_{L^q\to L^\infty}&=\left\|\left((P-\Pi)(QPQ)^{t-1}-\Pi\right)\right\|_{L^q\to L^\infty}\\
        &\leq \left\|P-\Pi\right\|_{L^\infty\to L^\infty}\left\|(QPQ)^{t-1}-\Pi\right\|_{L^q\to L^\infty}\\
        &\leq 2\left\|(QPQ)^{t-1}-\Pi\right\|_{L^q\to L^\infty},
    \end{align*}
    as desired.
\end{proof}

Theorem \ref{thm:mixing time equivalence within QP,PQ,QPQ} and Remark \ref{remark:asympvar QP=PQ=QPQ} collectively justify the use of $P_{la}$ and $P_{ra}$ as viable alternatives to $(P_{la})_{ra}$, offering similar performance in terms of both mixing time and asymptotic variance, with less computational cost per iteration as extra benefits.

\section{$\pi$ without group invariance: artificial group planting}\label{sec:groupplanting}
In many problems, it is generally hard to determine the natural group symmetry of the target distribution $\pi$, particularly for continuous state spaces (e.g. $\mathcal{X}=\mathbb R^d$). This difficulty limits the direct application of the averaged chains developed in previous sections. To circumvent it, we propose two strategies by deliberately selecting a group by hand --- a procedure we term \textbf{artificial group planting}:
\begin{enumerate}
    \item \textbf{Importance sampling correction}: Given a group $G$ and Haar measure $\mu$, we sample from an auxiliary $G$-invariant distribution that approximates $\pi$, then correct the bias via importance sampling. Specifically, we take the auxiliary distribution to be $\pi_G$ with density given by $\pi_G(x):=\mathbb E_{g\sim \mu}\left(\pi(gx)\right)$.

    \item \textbf{State-dependent averaging}: Given a group $G$, we take the averaging procedure to be state-dependent, where the distribution of $g$ depends on the current state $x$ rather than following the Haar measure. Specifically, we provide generalized versions of $P_{la}$, $P_{ra}$ and $(P_{la})_{ra}$, and focus our analysis on these three chains. 
\end{enumerate}

\subsection{Importance sampling correction}
Given the target distribution $\pi$ and $f\in L^2(\pi)$, a common goal is to evaluate the expectation 
\begin{equation*}
    I(f)=\int_{\mathcal{X}}f(x)\pi(dx).
\end{equation*}
The importance sampling scheme introduces an auxiliary distribution $\pi_0$ which is usually more tractable than $\pi$ with $\pi\ll \pi_0$, then generates samples $\{X_i\}_{i=1}^n$ from $\pi_0$, and uses 
\begin{equation*}
    \widehat I_n(f):=\frac{\sum_{i=1}^n f(X_i)\phi(X_i)}{\sum_{i=1}^n \phi(X_i)}, \quad \text{where }\phi\propto \frac{d\pi}{d\pi_0}
\end{equation*}
as the estimator of $I(f)$. According to \citep{chatterjee2018sample}, the sample size sufficient and necessary for $\widehat I_n(f)$ to fully approximate $I(f)$ is the same order with 
\begin{equation}\label{eq:sample size for importance sampling}
    N=\exp \left(D_{KL}(\pi\|\pi_0)\right),
\end{equation}
where $D_{KL}\left(\pi\|\pi_0\right)=\int\log\frac{d\pi}{d\pi_0}d\pi$ is the classical KL divergence from $\pi_0$ to $\pi$.

In this subsection, after a finite group $G$ with uniform distribution $\mu$ is selected, we take the auxiliary distribution $\pi_0$ to be 
\begin{equation}\label{eq:pi_G}
    \pi_0(x)=\pi_G(x):=\mathbb E_{g\sim \mu}\left(\pi(gx)\right),
\end{equation}
which is easy to be verified as a $G$-invariant probability density. The reason behind the choice \eqref{eq:pi_G} lies in two aspects:
\begin{itemize}
    \item $\pi_G$ not only enables our averaged kernels to apply, but also minimizes the sample size required in \eqref{eq:sample size for importance sampling} among all $G$-invariant distributions. 
    
    \item One can design Markov chains targeting $\pi_G$ without evaluating the sum $\sum_{g\in G}$, allowing the approach to remain computationally feasible even if $G$ is very large. 
\end{itemize}

The first point is supported by the following result, which is a direct consequence of the Pythagorean identity under KL divergence in Section \ref{sec:geometry}. 
\begin{theorem}
    Let $G$ be a finite group acting on $\mathcal{X}$, and let $\mu$ be the uniform distribution on $G$. Under Assumption \ref{assum:pi and G}, for any distribution $\pi_0$ satisfying $\pi\ll \pi_0$ and $\pi_0(gx)=\pi_0(x)$ for any $g\in G$, we have 
    \begin{equation*}
        D_{KL}\left(\pi\|\pi_G\right)\leq D_{KL}\left(\pi\|\pi_0\right).
    \end{equation*}
\end{theorem}
\begin{proof}
    For any $f\in L^2(\pi)$, let $\Pi_0[f](x):=\pi_0(f)$ and $\Pi_G[f](x):=\pi_G(f)$. Recalling $\mathcal{RI}(G)$ defined in Section \ref{sec:prelim}, we have
    \begin{align*}
        \pi_0(gx)=\pi_0(x), \quad \forall x\in \mathcal{X}, g\in G\quad &\Longrightarrow \quad  \Pi_0 U_g=\Pi_0, \quad \forall g\in G\\
        &\Longrightarrow \quad \Pi_0\in \mathcal{RI}(G),
    \end{align*}
    then by Corollary \ref{cor:pythKL}, observing that $\Pi_G=(\Pi)_{la}$, we have
    \begin{equation*}
        D_{KL}\left(\pi\|\pi_G\right)=D_{KL}^{\pi}\left(\Pi\|\Pi_G\right)\leq D_{KL}^{\pi}\left(\Pi\|\Pi_0\right)=D_{KL}\left(\pi\|\pi_0\right),
    \end{equation*}
    then the result follows.
\end{proof}

In the second point, recall that  
\begin{equation*}
    \pi_G(x)=\frac{1}{|G|}\sum_{g\in G}\pi(gx),
\end{equation*}
although many samplers targeting $\pi_G$ do not require the normalizing constant of $\pi$ --- they only need, for instance, ratios like $\pi_G(x)/\pi_G(y)$ or gradients such as $\nabla_x \log \pi_G(x)$ --- they still face the potentially prohibitive cost of computing the sum $\sum_{g\in G}$ when $G$ is exponentially large. This obstacle can be avoided by algorithms that replace the exact sum with an unbiased estimate, and a prominent example is the pseudo-marginal Metropolis-Hastings (PMMH) and its many variants, see \citep{andrieu2009pseudo} and a more recent survey \citep{sherlockpseudo}. Here we use a simple algorithm to illustrate how such approach can be applied in our setting. Let the joint distribution of $(x,g)$ to be
\begin{equation}\label{eq:joint distribution of (x,g)}
    \widetilde \pi(x,g):=\frac{\pi(gx)}{|G|}, \quad (x,g)\in \mathcal{X}\times G,
\end{equation}
then the marginal of $\widetilde \pi$ at $x$ is $\pi_G$, marginal at $g$ is $\mu$, and $\pi(gx)$ can be seen as an unbiased estimator of $\pi_G(x)$. Next, we perform the following updating procedure: Starting from $(x,g)$,
\begin{enumerate}[label=(\roman*)]
    \item Draw $x'$ from some proposal chain $q(x,\cdot)$;

    \item Draw $g'$ from the uniform distribution $\mu(g)=\frac{1}{|G|}$;

    \item Accept $(x',g')$ with probability 
    \begin{equation*}
        \alpha\left((x,g),(x',g')\right)=\min \left\{1, \frac{\pi(g'x')q(x',x)}{\pi(gx)q(x,x')}\right\}.
    \end{equation*}
\end{enumerate}
It is easy to see that $\widetilde \pi$ is the stationary distribution of such algorithm, then we get a sampler of its marginal $\pi_G$. According to \citep{andrieu2009pseudo}, this algorithm is ergodic and converges to $\widetilde \pi$ under mild conditions, and more explicit convergence properties can be found in \citep{andrieu2015convergence}.

\subsection{State-dependent averaging}\label{subsec:state dependent averaging}
If $\pi$ is not $G$-invariant, the standard averaged kernels generally fail to preserve stationarity with respect to $\pi$. To address this limitation and construct $\pi$-invariant averaged kernels with the desired improved properties, we propose a state-dependent averaging scheme based on previous sections. While similar constructions appear in \citep{khare2011spectral, kamatani2023non} for specific algorithmic purposes,  we develop in this subsection a general theoretical framework towards arbitrary Markov kernels, with fundamentally different motivations. We also provide more explicit theoretical guarantees in mixing improvement, as well as a geometric interpretation in terms of information projection. 

For convenience of practical implementation, we focus on the finite group case in this subsection. One first select a finite group $G$ acting on $\mathcal{X}$, with a slight abuse of notations with \eqref{eq:QP, PQ, QPQ}, we define, for $f \in L^2(\pi)$, $Q = Q(G,\pi) : L^2(\pi) \to L^2(\pi)$ to be
\begin{equation}\label{eq:state-dependent Q}
    Q[f](x):=Z_G(x)^{-1}\cdot\sum_{g\in G}f(gx)\pi(gx), \quad \text{where } Z_G(x):=\sum_{g\in G}\pi(gx).
\end{equation}
If $\pi$ is $G$-invariant, then $Q$ defined in \eqref{eq:state-dependent Q} coincides with the one in \eqref{eq:QP, PQ, QPQ}, hence \eqref{eq:state-dependent Q} can be understood as a generalization of it. Indeed, the $Q$ in \eqref{eq:state-dependent Q} is also a Markov kernel with an updating procedure as follows: 
\begin{equation}\label{eq:jump step Q}
    \textbf{Starting from }x, \textbf{ randomly draw }g\in G \textbf{ with probability }\frac{\pi(gx)}{Z_G(x)}, \textbf{ then update }x\leftarrow gx. 
\end{equation}
If the $G$ selected is not too large, this procedure is easy to implement, since the normalizing constant of $\pi$ is not required. Here we stress that $Q$ is generally non-ergodic because the chain remains in some group orbit, therefore one cannot use only $Q$ to sample from $\pi$. We now present some basic properties of $Q$ which are similar to those in \citep[Section 4]{khare2011spectral}.
\begin{lemma}\label{lem:properties of Q}
    Let $G$ be a finite group that acts on $\mathcal{X}$. Under Assumption \ref{assum:pi and G}, for the Markov kernel $Q=Q(G,\pi)$, the following statements hold.
    \begin{enumerate}[label=(\roman*)]
        \item $Q$ is reversible in $L^2(\pi)$, and thus $\pi$-stationary.

        \item $Q$ is the projection operator onto $V$, i.e. $Q=\mathbf{P}_V$, recalling \eqref{eq:G-invariant subspace}.
    \end{enumerate}
\end{lemma}
\begin{proof}
    We only deal with the case of $\mathcal{X}=\mathbb R^d$, and the finite state space case is similar. 
    
    For item (i), for any $u,v\in L^2(\pi)$, we have 
    \begin{align*}
        \langle Q[u], v \rangle_\pi&=\int_{\mathcal{X}} \frac{\sum_{g\in G}u(gx)\pi(gx)}{Z_G(x)}\cdot v(x)\pi(x)\mathfrak{m}(dx)
        =\sum_{g\in G}\int_{\mathcal{X}}\frac{u(gx)v(x)\pi(gx)\pi(x)}{Z_G(x)}\mathfrak{m}(dx)\\
        &=\sum_{g\in G}\int_{\mathcal{X}}\frac{u(x)v(g^{-1}x)\pi(x)\pi(g^{-1}x)}{Z_G(x)}\mathfrak{m}(dx)\\
        &=\sum_{g\in G}\int_{\mathcal{X}}\frac{v(gx)u(x)\pi(gx)\pi(x)}{Z_G(x)}\mathfrak{m}(dx)\\
        &=\langle u, Q[v]\rangle_\pi,
    \end{align*}
    where we have used the fact that $Z_G(gx)=Z_G(x)$ and $\dfrac{d\mathfrak{m}\circ g^{-1}}{d\mathfrak{m}}=1$ for any $g\in G$. 

    For item (ii), for any $f\in V$, it is easy to verify $Q[f]=f$ from definition. Then, it suffices to show that $Q[f]\in V$ for any $f\in L^2(\pi)$. Actually, for any $f\in L^2(\pi)$ and $h\in G$, 
    \begin{align*}
        Q[f](hx)&=Z_G(hx)^{-1}\cdot \sum_{g\in G}f(ghx)\pi(ghx)\\
        &=Z_G(x)^{-1}\cdot \sum_{g\in G}f(gx)\pi(gx)=Q[f](x),
    \end{align*}
    then the result follows.
\end{proof}

For a Markov kernel $P$, we can extend the notions introduced in \eqref{eq:QP, PQ, QPQ} by setting $P_{la}=QP$, $P_{ra}=PQ$ and $(P_{la})_{ra}=QPQ$ with $Q=Q(G,\pi)$ under the same notations in this new situation. For general-double-averages, it is usually difficult to guarantee their $\pi$-stationarity. Therefore, in this subsection we restrict our attention to the three special averages $P_{la}$, $P_{ra}$ and $(P_{la})_{ra}$ and investigate their properties. In the following theorem, we show that most results in Section \ref{sec:improvement}, \ref{sec:geometry} and \ref{sec:comparison} carry over directly to these three kernels, and to avoid repetition, we omit the corresponding detailed formulas. Going beyond previous sections, under the same notations we define
\begin{equation*}
    \mathcal{LI}(G) = \mathcal{LI}(G,\pi):=\{P\in \mathcal{L}: QP=P\}, \quad \mathcal{RI}(G) = \mathcal{RI}(G,\pi) :=\{P\in \mathcal{L}: PQ=P\}.
\end{equation*}

\begin{theorem}\label{thm:state-dependent kernels}
    Assume $P\in \mathcal{S}(\pi)$. Let $G$ be any finite group acting on $\mathcal{X}$. Under Assumption \ref{assum:pi and G}, without assuming any group invariance of $\pi$, the arguments in the following results still hold for $P_{la}$, $P_{ra}$ and $(P_{la})_{ra}$ defined in this subsection (with $P_{da}$ in the original results replaced by these three kernels):
    \begin{enumerate}[label=(\roman*)]
        \item Improvement of multiplicative spectral gap: Theorem \ref{thm:singular value improvement} (assuming $P$ is compact) and Corollary \ref{cor:larger group is better}.

        \item Improvement of asymptotic variance: Theorem \ref{thm:asympvar improvement} (assuming $P$ is compact) and Remark \ref{remark:asympvar QP=PQ=QPQ} (both assuming $P$ is reversible).

        \item Pythagorean identities: Theorem \ref{thm:pythKL} (assuming $P$ and $M$ admit a transition density w.r.t. $\mathfrak{m}$ at any starting state $x$) and Theorem \ref{thm:pythHS} HS part (assuming $P$ is a HS operator). Precisely, we have
        \begin{align}
        	D_{KL}^\pi(P\|M) &= D_{KL}^\pi(P\|(P_{la})_{ra})+D_{KL}^\pi((P_{la})_{ra}\|M), \quad  M\in \mathcal{LI}(G)\cap \mathcal{RI}(G)\cap \mathcal{S}(\pi),\label{eq:Pythagorean KL for QPQ} \\
        	D_{KL}^\pi(P\|M) &= D_{KL}^\pi(P\|P_{la})+D_{KL}^\pi(P_{la}\|M), \quad  M\in \mathcal{LI}(G)\cap \mathcal{S}(\pi),\label{eq:Pythagorean KL for QP}\\
        	D_{KL}^\pi(P\|M) &= D_{KL}^\pi(P\|P_{ra})+D_{KL}^\pi(P_{ra}\|M), \quad  M\in \mathcal{RI}(G)\cap \mathcal{S}(\pi),\label{eq:Pythagorean KL for PQ}
        \end{align}
        and 
        \begin{align}
            \norm{P-M}_{\mathrm{HS}}^2 &=\norm{P-(P_{la})_{ra}}_{\mathrm{HS}}^2+\norm{(P_{la})_{ra}-M}_{\mathrm{HS}}^2, \quad M\in \mathcal{LI}(G)\cap \mathcal{RI}(G)\cap \mathcal{S}(\pi),\label{eq:Pythagorean HS for QPQ}\\
            \norm{P-M}_{\mathrm{HS}}^2 &=\norm{P-P_{la}}_{\mathrm{HS}}^2+\norm{P_{la}-M}_{\mathrm{HS}}^2, \quad M\in \mathcal{LI}(G)\cap \mathcal{S}(\pi),\label{eq:Pythagorean HS for QP}\\
            \norm{P-M}_{\mathrm{HS}}^2 &=\norm{P-P_{ra}}_{\mathrm{HS}}^2+\norm{P_{ra}-M}_{\mathrm{HS}}^2, \quad M\in \mathcal{RI}(G)\cap \mathcal{S}(\pi).\label{eq:Pythagorean HS for PQ}
        \end{align}

        \item Comparable mixing times: Theorem \ref{thm:mixing time equivalence within QP,PQ,QPQ}.
    \end{enumerate}
\end{theorem}
\begin{proof}
    The proofs for item (i), (ii) and (iv) are essentially the same with the original ones in previous sections. We only prove item (iii), for which we provide an alternative argument in the KL case, as the bisection property used earlier may not hold in this setting. 

    For Pythagorean identity under $\pi$-weighted KL divergence, we only deal with $\mathcal{X}=\mathbb R^d$, and the finite case is similar. We start with $(P_{la})_{ra}=QPQ$. Recalling that 
    \begin{align*}
        \{P\in \mathcal{S}(\pi): QPQ=P\}&=\{P\in \mathcal{S}(\pi): QP=P\}\cap \{P\in \mathcal{S}(\pi): PQ=P\}\\
        &=\mathcal{LI}(G)\cap \mathcal{RI}(G)\cap \mathcal{S}(\pi),
    \end{align*}
    for any $M\in \mathcal{LI}(G)\cap \mathcal{RI}(G)\cap\mathcal{S}(\pi)$, we first show that for any $g,h\in G$, 
    \begin{equation}\label{eq:properties of QMQ=M}
        M(gx, hy)=\frac{\pi(hy)}{\pi(y)}M(x,y), \quad \mathfrak{m}\text{-a.e. }x,y\in \mathcal{X}.
    \end{equation}
    According to Lemma \ref{lem:properties of Q}, it can be easily verified that 
    \begin{align}
        QP=P &\iff P[f]\in V, \quad \forall f\in L^2(\pi)\nonumber\\
        &\iff P(gx,y)=P(x,y), \quad \mathfrak{m}\text{-a.e. }y\in \mathcal{X},\enspace \forall x\in \mathcal{X}, g\in G,\label{eq:properties of QM=M}
    \end{align}
    based upon which we get 
    \begin{align}
        PQ=P &\iff QP^*=P^* \nonumber\\
        &\iff P^*(gx,y)=P^*(x,y), \quad \mathfrak{m}\text{-a.e. }y\in \mathcal{X},\enspace \forall x\in \mathcal{X}, g\in G\nonumber\\
        &\iff P(x, hy)=\frac{\pi(hy)}{\pi(y)}P(x,y), \quad \mathfrak{m}\text{-a.e. }x,y\in \mathcal{X},\label{eq:properties of MQ=M}
    \end{align}
    then \eqref{eq:properties of QMQ=M} follows. Next, we can decompose 
    \begin{align}
        D_{KL}^\pi(P\|M)&=D_{KL}^\pi(P\|QPQ)+D_{KL}^\pi(QPQ\|M)\nonumber\\
        &\quad +\int_{\mathcal{X}\times \mathcal{X}}\pi(x)\left(P(x,y)-QPQ(x,y)\right)\log \left(\frac{QPQ(x,y)}{M(x,y)}\right)\mathfrak{m}(dx)\mathfrak{m}(dy),\label{eq:cross term in KL(P||M)}
    \end{align}
    then it suffices to show the rightmost term above is $0$. We first rewrite $QPQ(x,y)$ in a more explicit form, and the reason that $QPQ$ also admits a density is shown later. For any $f\in L^2(\pi)$, we have
    \begin{align*}
        QP[f](x)&=\sum_{g\in G}\frac{P[f](gx)\pi(gx)}{Z_G(x)}=\sum_{g\in G}\frac{\pi(gx)\int_{\mathcal{X}}f(y)P(gx,y)\mathfrak{m}(dy)}{Z_G(x)}\\
        &=\int_{\mathcal{X}}f(y)\cdot\frac{\sum_{g\in G}P(gx,y)\pi(gx)}{Z_G(x)}\mathfrak{m}(dy),
    \end{align*}
    and similarly
    \begin{align*}
        PQ[f](x)&=\int_{\mathcal{X}}P(x,y)Q[f](y)\mathfrak{m}(dy)=\int_{\mathcal{X}}P(x,y)\frac{\sum_{g\in G}f(gy)\pi(gy)}{Z_G(y)}\mathfrak{m}(dy)\\
        &=\sum_{g\in G}\int_{\mathcal{X}}\frac{P(x,y)f(gy)\pi(gy)}{Z_G(y)}\mathfrak{m}(dy)=\sum_{g\in G}\int_{\mathcal{X}}\frac{P(x,g^{-1}y)f(y)\pi(y)}{Z_G(y)}\mathfrak{m}(dy)\\
        &=\int_{\mathcal{X}}f(y)\cdot \frac{\sum_{g\in G}P(x,gy)\pi(y)}{Z_G(y)}\mathfrak{m}(dy),
    \end{align*}
    therefore $QP$ and $PQ$ also admits a density at $x$ which can be written as
    \begin{equation}\label{eq:detailed form of QP and PQ}
        QP(x,y)=\frac{\sum_{g\in G}P(gx,y)\pi(gx)}{Z_G(x)}, \quad PQ(x,y)=\frac{\sum_{g\in G}P(x,gy)\pi(y)}{Z_G(y)}.
    \end{equation}
    We then see that
    \begin{align*}
        QPQ[f](x)&=\int_{\mathcal{X}}QP(x,y)Q[f](y)\mathfrak{m}(dy)=\int_{\mathcal{X}}\frac{\sum_{g\in G}P(gx,y)\pi(gx)}{Z_G(x)}\cdot \frac{\sum_{h\in G}f(hy)\pi(hy)}{Z_G(y)}\mathfrak{m}(dy)\\
        &=\sum_{g,h\in G}\int_{\mathcal{X}}\frac{f(hy)P(gx,y)\pi(gx)\pi(hy)}{Z_G(x)Z_G(y)}\mathfrak{m}(dy)\\
        &=\sum_{g,h\in G}\int_{\mathcal{X}}\frac{f(y)P(gx,h^{-1}y)\pi(gx)\pi(y)}{Z_G(x)Z_G(y)}\mathfrak{m}(dy)\\
        &=\int_{\mathcal{X}}f(y)\cdot \sum_{g,h\in G}\frac{\pi(gx)}{Z_G(x)}P(gx,hy)\frac{\pi(y)}{Z_G(y)}\mathfrak{m}(dy),
    \end{align*}
    hence the density of $QPQ(x,\cdot)$ is
    \begin{equation}\label{eq:detailed form of QPQ}
        QPQ(x,y)=\sum_{g,h\in G}\frac{\pi(gx)}{Z_G(x)}P(gx,hy)\frac{\pi(y)}{Z_G(y)}.
    \end{equation}
    Plugging \eqref{eq:detailed form of QPQ} into \eqref{eq:cross term in KL(P||M)}, we have 
    \begin{align*}
        &\quad\int_{\mathcal{X}\times \mathcal{X}}\pi(x)QPQ(x,y)\log \left(\frac{QPQ(x,y)}{M(x,y)}\right)\mathfrak{m}(dx)\mathfrak{m}(dy)\\
        &=\sum_{g,h\in G}\int_{\mathcal{X}\times \mathcal{X}}\pi(x)\cdot\frac{\pi(gx)}{Z_G(x)}P(gx,hy)\frac{\pi(y)}{Z_G(y)}\log \left(\frac{QPQ(x,y)}{M(x,y)}\right)\mathfrak{m}(dx)\mathfrak{m}(dy)\\
        &=\int_{\mathcal{X}\times \mathcal{X}}\sum_{g,h\in G}\pi(g^{-1}x)\cdot\frac{\pi(x)}{Z_G(x)}P(x,y)\frac{\pi(h^{-1}y)}{Z_G(y)}\log \left(\frac{QPQ(x,y)}{M(x,y)}\right)\mathfrak{m}(dx)\mathfrak{m}(dy)\\
        &=\int_{\mathcal{X}\times \mathcal{X}}\pi(x)P(x,y)\log \left(\frac{QPQ(x,y)}{M(x,y)}\right)\mathfrak{m}(dx)\mathfrak{m}(dy), 
    \end{align*}
    where in the second equality we have used \eqref{eq:properties of QMQ=M}. Then we obtain \eqref{eq:Pythagorean KL for QPQ}.
    
    For $P_{la}=QP$, for any $M\in \mathcal{LI}(G)\cap \mathcal{S}(\pi)$, using \eqref{eq:detailed form of QP and PQ}, we similarly have 
    \begin{align*}
        &\quad\int_{\mathcal{X}\times \mathcal{X}}\pi(x)QP(x,y)\log \left(\frac{QP(x,y)}{M(x,y)}\right)\mathfrak{m}(dx)\mathfrak{m}(dy)\\
        &=\sum_{g\in G}\int_{\mathcal{X}\times \mathcal{X}}\pi(x)\cdot\frac{\pi(gx)}{Z_G(x)}P(gx,y)\log \left(\frac{QP(x,y)}{M(x,y)}\right)\mathfrak{m}(dx)\mathfrak{m}(dy)\\
        &=\int_{\mathcal{X}\times \mathcal{X}}\sum_{g\in G}\pi(g^{-1}x)\cdot\frac{\pi(x)}{Z_G(x)}P(x,y)\log \left(\frac{QP(x,y)}{M(x,y)}\right)\mathfrak{m}(dx)\mathfrak{m}(dy)\\
        &=\int_{\mathcal{X}\times \mathcal{X}}\pi(x)P(x,y)\log \left(\frac{QP(x,y)}{M(x,y)}\right)\mathfrak{m}(dx)\mathfrak{m}(dy), 
    \end{align*}
    where we have used \eqref{eq:properties of QM=M} in the second equality. For $P_{ra}=PQ$, combining \eqref{eq:properties of MQ=M}, \eqref{eq:cross term in KL(P||M)} and \eqref{eq:detailed form of QP and PQ}, the argument is similar. We thus obtain \eqref{eq:Pythagorean KL for QP} and \eqref{eq:Pythagorean KL for PQ}.

    Next, we prove the Pythagorean identity under squared-HS norm. We first show that $P_{la}=QP$, $P_{ra}=PQ$ and $(P_{la})_{ra}=QPQ$ are all HS operators. Assume $L^2(\pi)=V\oplus V^\perp$ admits a set of orthonormal basis $\{f_i\}_{i\in \mathcal{V}_1}\cup \{f_i\}_{i\in \mathcal{V}_2}$, where $f_i\in V$ for $i\in \mathcal{V}_1$ and $f_i\in V^\perp$ for $i\in \mathcal{V}_2$. For any other set of basis $\{e_j\}_{j\in \mathcal{J}}$, we have 
    \begin{align*}
        \norm{QP}_{\mathrm{HS}}^2&=\sum_{i\in \mathcal{V}_1\cup \mathcal{V}_2, j\in \mathcal{J}}\left|\langle f_i, QP[e_j]\rangle_\pi\right|^2=\sum_{i\in \mathcal{V}_1\cup \mathcal{V}_2, j\in \mathcal{J}}\left|\langle Q[f_i], P[e_j]\rangle_\pi\right|^2\\
        &=\sum_{i\in \mathcal{V}_1, j\in \mathcal{J}}\left|\langle f_i, P[e_j]\rangle_\pi\right|^2\leq \norm{P}_{\mathrm{HS}}^2,
    \end{align*}
    and thus 
    \begin{align*}
        &\norm{PQ}_{\mathrm{HS}}^2 =\norm{QP^*}_{\mathrm{HS}}^2\leq \norm{P^*}_{\mathrm{HS}}^2=\norm{P}_{\mathrm{HS}}^2,\\
        &\norm{QPQ}_{\mathrm{HS}}^2 \leq \norm{PQ}_{\mathrm{HS}}^2\leq \norm{P}_{\mathrm{HS}}^2.
    \end{align*}
    Now,  for any $M\in \mathcal{LI}(G)\cap \mathcal{RI}(G)\cap\mathcal{S}(\pi)$, we can decompose
    \begin{align*}
        \norm{P-M}_{\mathrm{HS}}^2&=\norm{P-QPQ}_{\mathrm{HS}}^2+\norm{QPQ-M}_{\mathrm{HS}}^2\\
        &\quad +2\mathrm{Tr}((P-QPQ)^*(QPQ-M)),
    \end{align*}
    where the last term is 0, since by cyclic property of trace, 
    \begin{align*}
        \mathrm{Tr}((P-QPQ)^*M)&=\mathrm{Tr}((P^*-QP^*Q)QMQ)\\
        &=\mathrm{Tr}((QP^*Q-QP^*Q)M)=0,
    \end{align*}
    and this yields \eqref{eq:Pythagorean HS for QPQ}. For any $M\in \mathcal{LI}(G)\cap\mathcal{S}(\pi)$, we also have
    \begin{align*}
        \mathrm{Tr}((P-QP)^*M)&=\mathrm{Tr}((P^*-P^*Q)QM)\\
        &=\mathrm{Tr}((P^*Q-P^*Q)M)=0,
    \end{align*}
    and for any $M\in \mathcal{RI}(G)\cap\mathcal{S}(\pi)$, 
    \begin{align*}
        \mathrm{Tr}((P-PQ)^*M)&=\mathrm{Tr}((P^*-QP^*)MQ)\\
        &=\mathrm{Tr}((QP^*-QP^*)M)=0,
    \end{align*}
    which leads to \eqref{eq:Pythagorean HS for QP} and \eqref{eq:Pythagorean HS for PQ}. 
\end{proof}

\begin{remark}
    The Pythagorean identity under squared-Frobenius norm in this setting may not hold. We present two counterexamples on state space $\mathcal{X}=\{1,2,3\}$ with $\pi=(0.3,0.5,0.2)$. Let $G=\{e,(12)\}$, and 
    \begin{gather*}
        P=\begin{pmatrix}
              0.6 & 0.3 & 0.1\\[3pt]
              0.2 & 0.7 & 0.1\\[3pt]
              0.1 & 0.3 & 0.6
        \end{pmatrix},
        \qquad
        M = Q=\begin{pmatrix}
            0.375 & 0.625 & 0 \\[3pt]
            0.375 & 0.625 & 0 \\[3pt]
            0     & 0     & 1
        \end{pmatrix},
    \end{gather*}
    then one readily verifies that $M\in \mathcal{LI}(G)\cap \mathcal{RI}(G)\cap\mathcal{S}(\pi)$. We have 
    \begin{equation*}
       \norm{P-M}_{\mathrm{F}}^2=0.4725, 
    \end{equation*}
    and
    \begin{align*}
        \norm{P-QPQ}_{\mathrm{F}}^2+\norm{QPQ-M}_{\mathrm{F}}^2&=0.45625,\\
        \norm{P-QP}_{\mathrm{F}}^2+\norm{QP-M}_{\mathrm{F}}^2&=0.46250,\\
        \norm{P-PQ}_{\mathrm{F}}^2+\norm{PQ-M}_{\mathrm{F}}^2&=0.45625.
    \end{align*}
    This yields LHS $>$ RHS in the Pythagorean identities. Next, under the same $\pi$ and $G$, we take 
    \begin{gather*}
        P=\begin{pmatrix}
              \frac{2}{3} & \frac{1}{10} & \frac{7}{30}\\[4pt]
              \frac{3}{50} & \frac{22}{25} & \frac{3}{50}\\[4pt]
              \frac{7}{20} & \frac{3}{20} & \frac{1}{2}
        \end{pmatrix},
        \qquad
        M=Q=\begin{pmatrix}
            0.375 & 0.625 & 0 \\[3pt]
            0.375 & 0.625 & 0 \\[3pt]
            0     & 0     & 1
        \end{pmatrix},
    \end{gather*}
    then we compute that
    \begin{equation*}
        \norm{P-M}_{\mathrm{F}}^2\approx 0.9780, 
    \end{equation*}
    and 
    \begin{align*}
        \norm{P-QPQ}_{\mathrm{F}}^2+\norm{QPQ-M}_{\mathrm{F}}^2&\approx 0.9966,\\
        \norm{P-QP}_{\mathrm{F}}^2+\norm{QP-M}_{\mathrm{F}}^2&\approx 0.9791,\\
        \norm{P-PQ}_{\mathrm{F}}^2+\norm{PQ-M}_{\mathrm{F}}^2&\approx 0.9832,
    \end{align*}
    which leads to LHS $<$ RHS in the Pythagorean identities.
\end{remark}

A key assumption in Theorem \ref{thm:state-dependent kernels} item (iii) the KL-Pythagorean identities is that $P(x,\cdot)$ and $M(x,\cdot)$ should admit a density w.r.t. the reference measure $\mathfrak{m}$. This assumption breaks down for many practical MCMC schemes on continuous state space (e.g. $\mathcal{X}=\mathbb R^d$) that include an explicit rejection step, because such kernels place positive point mass on the current state, i.e. $P(x,\{x\})>0$. Typical examples are the various variants of the Metropolis-Hastings algorithm. Now, on continuous state space with reference measure $\mathfrak{m}$, we focus on the set of $P$ which can be decomposed into the continuous part and discrete part, i.e.
\begin{align}
    P(x,dy)&=P_c(x,y)\mathfrak{m}(dy)+\sum_{z\in A_x}\rho(x,z) \delta_z(dy)\label{eq:continuous and discrete decomposition of P}\\
    &=: P_c(x,dy)+P_d(x,dy),\nonumber
\end{align}
where $A_x\subset \mathcal{X}$ is a finite set depending on $x$, $\rho(x,z)>0$, and $P_c$, $P_d$ can be seen as two sub-stochastic kernels. We use $P_d(x,y)$ to denote the transition probability of $P_d$ from $x$ to $y\in A_x$, and $P_d(x,y)=0$ if $y\notin A_x$ (and so is $M_d$ appearing later). Next, we show that Pythagorean identities also hold for $P$ satisfying \eqref{eq:continuous and discrete decomposition of P}. 
\begin{corollary}
    Assume $P\in \mathcal{S}(\pi)$. Let $G$ be any finite group acting on $\mathcal{X}$. Under Assumption \ref{assum:pi and G}, suppose the state space $\mathcal{X}$ is continuous, $\mathfrak{m}$ has no positive point mass, and $P$ has the form of \eqref{eq:continuous and discrete decomposition of P} with $A_{gx}=gA_x$, then for $M$ satisfying
    \begin{equation}\label{eq:continuous and discrete decomposition of M}
        M(x,dy)=M_c(x,y)\mathfrak{m}(dy)+\sum_{z\in GA_x}r(x,z)\delta_z(dy)=:M_c(x,dy)+M_d(x,dy),
    \end{equation}
    where $GA_x:=\{gz: g\in G, z\in A_x\}$ and $r(x,z)>0$, we have 
    \begin{align*}
        	D_{KL}^\pi(P\|M) &= D_{KL}^\pi(P\|(P_{la})_{ra})+D_{KL}^\pi((P_{la})_{ra}\|M), \quad  M\in \mathcal{LI}(G)\cap \mathcal{RI}(G)\cap \mathcal{S}(\pi), \\
        	D_{KL}^\pi(P\|M) &= D_{KL}^\pi(P\|P_{la})+D_{KL}^\pi(P_{la}\|M), \quad  M\in \mathcal{LI}(G)\cap \mathcal{S}(\pi),\\
        	D_{KL}^\pi(P\|M) &= D_{KL}^\pi(P\|P_{ra})+D_{KL}^\pi(P_{ra}\|M), \quad  M\in \mathcal{RI}(G)\cap \mathcal{S}(\pi).
        \end{align*}
\end{corollary}
\begin{proof}
    Similar to the proof in Theorem \ref{thm:state-dependent kernels}, for $M$ satisfying \eqref{eq:continuous and discrete decomposition of M} and $M\in \mathcal{LI}(G)\cap \mathcal{RI}(G)\cap\mathcal{S}(\pi)$, we have
    \begin{align*}
         M_c(gx, hy)&=\frac{\pi(hy)}{\pi(y)}M_c(x,y), \quad \mathfrak{m}\text{-a.e. }x,y\in \mathcal{X}, \\
         M_d(gx, hy)&=\frac{\pi(hy)}{\pi(y)}M_d(x,y), \quad \forall x,y\in \mathcal{X},
    \end{align*}
    which is given by the following fact that can be easily verified: 
    \begin{equation}\label{eq:decomposition of QM=M}
        QM=M \iff QM_c=M_c, \enspace QM_d=M_d.
    \end{equation}
    Next, we need to write out the explicit form of $QP$, $PQ$ and $QPQ$. The continuous part is essentially the same with \eqref{eq:detailed form of QP and PQ} and \eqref{eq:detailed form of QPQ}, and we only need to deal with the discrete part. For any $f\in L^2(\pi)$, we have 
    \begin{align*}
        QP_d[f](x)&=\sum_{g\in G}\frac{P_d[f](gx)\pi(gx)}{Z_G(x)}=\sum_{g\in G}\sum_{z\in A_{gx}}\frac{\pi(gx)\rho(gx,z)f(z)}{Z_G(x)},\\
        P_dQ[f](x)&=\sum_{z\in A_x}\rho(x,z)Q[f](z)=\sum_{g\in G}\sum_{z\in A_x}\frac{\rho(x,z)\pi(gz)f(gz)}{Z_G(z)},\\
        QP_dQ[f](x)&=\sum_{g\in G}\frac{\pi(gx)P_dQ[f](gx)}{Z_G(x)}=\sum_{g,h\in G}\sum_{z\in A_{gx}}\frac{\pi(gx)\rho(gx,z)\pi(hz)f(hz)}{Z_G(x)Z_G(z)},
    \end{align*}
    hence recalling the assumption that $A_{gx}=gA_x$, we get
    \begin{align*}
        QP_d(x,dy)&=\sum_{g\in G}\sum_{z\in gA_{x}}\frac{\pi(gx)\rho(gx,z)}{Z_G(x)}\cdot \delta_z(dy),\\
        P_dQ(x,dy)&=\sum_{g\in G}\sum_{z\in A_x}\frac{\pi(gz)\rho(x,z)}{Z_G(z)}\delta_{gz}(dy),\\
        QP_dQ(x,dy)&=\sum_{g,h\in G}\sum_{z\in gA_x}\frac{\pi(gx)\rho(gx,z)\pi(hz)}{Z_G(x)Z_G(z)}\delta_{hz}(dy),
    \end{align*}
    then $QP(x,\cdot)$, $PQ(x,\cdot)$ and $QPQ(x,\cdot)$ are all absolutely continuous w.r.t. $M(x,\cdot)$ defined in \eqref{eq:continuous and discrete decomposition of M}, and $P(x,\cdot)$ is also absolutely continuous to these three. Therefore, we can proceed to use the decomposition of KL divergence and calculate the term similar to \eqref{eq:cross term in KL(P||M)}. For the case of $QPQ$, we have 
    \begin{align}
        &\quad \int_{\mathcal{X}\times \mathcal{X}}\pi(x)P(x,dy)\log \left(\frac{QPQ(x,dy)}{M(x,dy)}\right)\mathfrak{m}(dx)\nonumber\\
        &=\int_{\mathcal{X}\times \mathcal{X}}\pi(x)P_c(x,y)\log \left(\frac{QP_cQ(x,y)}{M_c(x,y)}\right)\mathfrak{m}(dx)\mathfrak{m}(dy)\label{eq:cross term continuous part}\\
        &\quad+\int_{\mathcal{X}\times \mathcal{X}}\pi(x)P_d(x,dy)\log\left(\frac{QPQ(x,dy)}{M(x,dy)}\right)\mathfrak{m}(dx)\label{eq:cross term discrete part},
    \end{align}
    and similar splitting holds for another term in \eqref{eq:cross term in KL(P||M)}, then it suffices to match the two parts respectively. The continuous part is direct via \eqref{eq:decomposition of QM=M} and the proof of Theorem \ref{thm:state-dependent kernels}, i.e.
    \begin{equation*}
        \int_{\mathcal{X}\times \mathcal{X}}\pi(x)QP_cQ(x,y)\log \left(\frac{QP_cQ(x,y)}{M_c(x,y)}\right)\mathfrak{m}(dx)\mathfrak{m}(dy)=\eqref{eq:cross term continuous part}.
    \end{equation*}
    For the discrete part, we have 
    \begin{align*}
        &\quad \int_{\mathcal{X}\times \mathcal{X}}\pi(x)QP_dQ(x,dy)\log\left(\frac{QP_dQ(x,dy)}{M_d(x,dy)}\right)\mathfrak{m}(dx)\\
        &=\int_{\mathcal{X}}\pi(x)\mathfrak{m}(dx)\sum_{y\in GA_x}\sum_{g,h\in G}\sum_{z\in gA_x}\frac{\pi(gx)\rho(gx,z)\pi(hz)}{Z_G(x)Z_G(z)}\mathbf{1}_{\{y=hz\}}\log\left(\frac{QP_dQ(x,y)}{M_d(x,y)}\right)\\
        &=\int_{\mathcal{X}}\pi(x)\mathfrak{m}(dx)\sum_{y\in GA_x}\sum_{g,h\in G}\frac{\pi(gx)\rho(gx,h^{-1}y)\pi(y)}{Z_G(x)Z_G(y)}\mathbf{1}_{\{y\in hgA_x\}}\log\left(\frac{QP_dQ(x,y)}{M_d(x,y)}\right)\\
        &=\int_{\mathcal{X}}\sum_{y\in GA_x}\sum_{g,h\in G}\frac{\pi(gx)\pi(x)\rho(x,y)\pi(hy)}{Z_G(x)Z_G(y)}\mathbf{1}_{\{y\in A_x\}}\log\left(\frac{QP_dQ(x,y)}{M_d(x,y)}\right)\mathfrak{m}(dx)\\
        &=\int_{\mathcal{X}}\sum_{y\in GA_x}\pi(x)\rho(x,y)\mathbf{1}_{\{y\in A_x\}}\log\left(\frac{QP_dQ(x,y)}{M_d(x,y)}\right)\mathfrak{m}(dx)\\
        &=\int_{\mathcal{X}}\pi(x)\mathfrak{m}(dx)\sum_{y\in A_x}\rho(x,y)\log\left(\frac{QP_dQ(x,y)}{M_d(x,y)}\right)=\eqref{eq:cross term discrete part},
    \end{align*}
    where in the third equality we have used change of variables and \eqref{eq:decomposition of QM=M}. For $QP$ and $PQ$, the argument is similar, then the result follows. 
\end{proof}

\subsection{Discussion of two methods}
In this subsection we discuss the advantages and disadvantages of these two methods proposed in the two previous subsections, and provide some guidelines on tuning the group $G$.

According to Corollary \ref{cor:larger group is better} and Theorem \ref{thm:state-dependent kernels} item (i), a larger group $G$ generally yields better improvement of the associated averaged kernels in terms of multiplicative spectral gap, so --- as a rule of thumb --- bigger is better for convergence performance. Yet each of two methods reacts differently to a large group:
\begin{itemize}
    \item First method (importance sampling correction).
        \begin{itemize}
             \item Advantage: We do not need to calculate the sum $\sum_{g\in G}$ through pseudo-marginal algorithms and thus straightforward to implement regardless of group size $|G|$.

             \item Disadvantage: Bias correction relies on the importance sampling step whose deviation to $I(f)$ (and hence the required sample size $N$ in \eqref{eq:sample size for importance sampling}) typically grows with $|G|$. For very large groups this extra sampling cost may erode the benefit in spectral gap.
         \end{itemize}

    \item Second method (state-dependent averaging).
         \begin{itemize}
             \item Advantage: Once $\pi(gx)/Z_G(x)$ in the $Q$-step \eqref{eq:jump step Q} is available (e.g. $G$ is small), no additional Monte Carlo resampling is required. 

             \item Disadvantage: Computing $\pi(gx)/Z_G(x)=\pi(gx)/\sum_{h\in G} \pi(hx)$ becomes hard when $G$ is exponentially large, making this strategy impractical in such cases. 
         \end{itemize}
\end{itemize}

Based on the discussion above, $G$ can be selected as follows. For the first method, one needs to achieve a trade-off between the spectral-gap improvement and the cost from sample-size requirement in importance sampling. This compromise is attractive in many statistical physics models whose target law already exhibits an ``approximate" symmetry, i.e. a $G$-invariant distribution can be found that closely matches the true target. Two notable papers fall into this direction \citep{ying2022annealed, ying2025multimodal}, where the Ising model with an external field is considered as an example, and an symmetric auxiliary distribution close to the target is paired in their annealed importance sampling (AIS) framework. In such settings, a carefully chosen $G$ delivers a significant spectral-gap improvement while keeping the extra budget from importance sampling within practical bounds.

For the second method, given the original Markov kernel $P$, one can try to minimize the distance of between $QP=Q(G,\pi)P$ and $\Pi$ under $\pi$-weighted KL divergence or squared-HS norm to select the optimal $G$ within some family $\mathcal{G}$ which contains moderate size of groups, i.e. to find
\begin{align*}
    G_{KL}^*=G_{KL}^*(P)&:=\argmin_{G\in \mathcal{G}}D_{KL}^{\pi}\left(Q(G,\pi)P\|\Pi\right),\\
    G_{\mathrm{HS}}^*=G_{\mathrm{HS}}^*(P)&:=\argmin_{G\in \mathcal{G}}\norm{Q(G,\pi)P-\Pi}_{\mathrm{HS}}^2,
\end{align*}
and according to Pythagorean identities in Theorem \ref{thm:state-dependent kernels}, this is equivalent to 
\begin{equation*}
    G_{KL}^*=\argmax_{G\in \mathcal{G}}D_{KL}^{\pi}\left(P\|Q(G,\pi)P\right), \quad G_{\mathrm{HS}}^*=\argmax_{G\in \mathcal{G}}\norm{P-Q(G,\pi)P}_{\mathrm{HS}}^2,
\end{equation*}
which is similar to the tuning strategy proposed in \citep[Section 6.1]{choi2025improving}. If the state space $\mathcal{X}$ is large, this optimization problem can still be challenging to solve computationally. 

Finally, we stress that there are other ways of state-dependent averaging. Another useful method is Metropolis-type averaging, while we call the method in Section \ref{subsec:state dependent averaging} as Gibbs-type averaging. To illustrate this, we consider the joint distribution $\widetilde{\pi}$ on extended state space $\mathcal{X}\times G$ introduced in \eqref{eq:joint distribution of (x,g)}. The step of randomly selecting group element can be seen as a transition kernel on $\mathcal{X}\times G$ targeting $\widetilde{\pi}$ with movement only on $G$, i.e. $(x,g)\to (x,g')$. Then the transition step of Section \ref{subsec:state dependent averaging} can be written as 
\begin{equation*}
    K_{\text{Gibbs}}\left((x,g),(x,g')\right)=\frac{\pi(g'x)}{Z_G(x)}=\widetilde{\pi}\left(g'|x\right).
\end{equation*}
The procedure of Metropolis-type is as follows:
\begin{enumerate}[label=(\roman*)]
    \item Starting from $(x,g)$, uniformly draw $g'\in G$;

    \item Accept with probability $\min\{1, \pi(g'x)/\pi(gx)\}$.
\end{enumerate}
Then the transition kernel is
\begin{align*}
    K_{\text{MH}}\left((x,g),(x,g')\right)&=
        \frac{1}{|G|}\min\left\{1, \frac{\pi(g'x)}{\pi(gx)}\right\}=\frac{1}{|G|}\min\left\{1, \frac{\widetilde{\pi}(x,g')}{\widetilde{\pi}(x,g)}\right\}, \quad g'\neq g,\\
    K_{\text{MH}}\left((x,g),(x,g)\right)&=1-\sum_{g'\in G, g'\neq g} K_{\text{MH}}\left((x,g),(x,g')\right).
\end{align*}
The advantage of Metropolis-type averaging is that it is still computationally feasible when $G$ is large. A similar form of this appears in \citep{ying2023double}, where the flipping group is taken for Ising models, and a Metropolis-type of move is made for choosing between the two elements in the group.

\section{Examples and applications}\label{sec:examples}

In this section, we highlight the practical value of our averaged kernels from the following complementary perspectives:
\begin{itemize}
    \item Algorithmic reformulation: Many modern sampling algorithms can be recast as some specific averaged kernels developed in previous sections. Viewing them through the lens of group symmetry not only reveals the key mechanism behind their acceleration, but also illustrates the broad applicability of our framework.

    \item Mixing enhancement: For some classical models, the technique of averaging can be applied on the standard samplers to improve the mixing time. Specifically, we shall consider improving the mixing time of Metropolis-Hastings from exponential to polynomial in the system size in a discrete bimodal V-shaped distribution in Section \ref{subsec:improveMH}.
\end{itemize}

\subsection{Algorithmic reformulation}

We consider several commonly used sampling algorithms and give their associated averaging ways to rewrite them in terms of group-averaged kernels. These examples unify disparate algorithms under a single framework, and provide practical templates for constructing and tuning $G$ in other problems.

\subsubsection{Swendsen-Wang algorithm}

The Swendsen-Wang algorithm introduced in \citep{swendsen1987nonuniversal} is the first non-local and cluster MCMC algorithm, and its numerous variants are widely used in statistical-physics simulation. The detailed procedure of this algorithm is as follows. Consider a $q$-Potts model of $n$-sites, let $(V,E)$ be the underlying graph where $|V|=n$ and $E$ is the undirected edge set. Let $\mathcal{X}=\{1,\dots,q\}^n$, for configuration $\sigma=(\sigma_1,\dots,\sigma_n)\in \mathcal{X}$, we define
\begin{equation*}
    \mathcal{H}(\sigma):=-\sum_{(i,j)\in E}J_{i,j}\mathbf{1}_{\{\sigma_i=\sigma_j\}}, \qquad \pi(\sigma)\propto e^{-\beta \mathcal{H}(\sigma)},
\end{equation*}
where $J_{i,j}>0$, and $\beta > 0$ is the inverse temperature. Starting from any configuration $\sigma$, we assign to each pair of vertices $i,j$ a Bernoulli random variable $b_{i,j}\in \{0,1\}$ following the rule:
\begin{gather*}
    \mathbb P\left(b_{i,j}=0|\sigma_i\neq \sigma_j\right)=1, \quad \mathbb P\left(b_{i,j}=1|\sigma_i\neq \sigma_j\right)=0,\\
    \mathbb P\left(b_{i,j}=0|\sigma_i= \sigma_j\right)=1-q_{i,j}, \quad \mathbb P\left(b_{i,j}=1|\sigma_i= \sigma_j\right)=q_{i,j},
\end{gather*}
where $q_{i,j}:=1-e^{-\beta J_{i,j}}$. If $b_{i,j}=1$, we say that there is a link between $i,j$, and for linked sites this defines a cluster. It is easy to see that sites in each cluster contain the same spin. Define $b:=(b_{i,j})_{n\times n}$ as the bond on the $n$ sites, $\mathcal{B}$ as the set of all possible bonds, and $\mathcal{C}(b)$ as the set of clusters induced by $b$. After the bond is updated, for each cluster, assign to the sites in it a new spin uniformly drawn from $\llbracket q\rrbracket$ and get a new configuration $\sigma'$. 
    
Now we show that on the extended state space $\mathcal{X}\times \mathcal{B}$, the transition kernel defined above can be written in the form of $P_{ra}=PQ$ where $Q$ is the state-dependent averaging introduced in Section \ref{subsec:state dependent averaging}. It is well known that the joint distribution of $(\sigma, b)$ is 
\begin{equation*}
    \widetilde{\pi}(\sigma,b)\propto \prod_{(i,j)\in E}\left(\left(1-q_{i,j}\right)^{1-b_{i,j}}\left(q_{i,j}\mathbf{1}_{\{\sigma_i=\sigma_j\}}\right)^{b_{i,j}}\right),
\end{equation*}
and the marginal in the $\sigma$-coordinate is $\pi$. The conditional distributions are
\begin{align*}
    \widetilde{\pi}(b|\sigma)&\propto \prod_{(i,j)\in E}\left(\left(1-q_{i,j}\right)^{1-b_{i,j}}\left(q_{i,j}\mathbf{1}_{\{\sigma_i=\sigma_j\}}\right)^{b_{i,j}}\right),\\
    \widetilde{\pi}(\sigma|b)&\propto \prod_{C\in \mathcal{C}(b)}\mathbf{1}_{\{\sigma_C\equiv\mathrm{const}\}},
\end{align*}
where $\sigma_C:=(\sigma_i)_{i\in C}$. Therefore, the algorithm can be interpreted as a Gibbs sampler targeting $\widetilde{\pi}$: from $(\sigma,b)$, draw $b'\sim \widetilde{\pi}(\cdot|\sigma)$ then $\sigma'\sim \widetilde{\pi}(\cdot|b')$. Next, take $P$ as the first step of Gibbs sampler, i.e.
\begin{equation*}
    P\left((\sigma,b),(\sigma',b')\right):=\widetilde \pi(b'|\sigma)\mathbf{1}_{\{\sigma=\sigma'\}},
\end{equation*}
and $G$ to be the direct product of $n$ permutation groups $S_q$, i.e.
\begin{equation*}
    G:=\prod_{i=1}^n S_q=S_q\times \cdots \times S_q.
\end{equation*}
For $g=(g_1,\cdots, g_n)\in G$, each $g_i$ is also a bijection on $\llbracket q\rrbracket$, then we extend the action of $g$ to $\mathcal{X}\times \mathcal{B}$ as
\begin{equation*}
    g\circ (\sigma,b):=(g\sigma, b), \quad \text{where }g\sigma:=(g_1\sigma_1,\cdots,g_n\sigma_n). 
\end{equation*}
Let 
\begin{equation*}
    G(b):=\left\{g\in G: g_C\equiv \text{const}, C\in \mathcal{C}(b)\right\},
\end{equation*}
then $G(b)\leq G$. Recalling for any $g\in G$,
\begin{equation*}
    \widetilde{\pi}(g\sigma,b)\propto \prod_{(i,j)\in E}\left(\left(1-q_{i,j}\right)^{1-b_{i,j}}\left(q_{i,j}\mathbf{1}_{\{g_i\sigma_i=g_j\sigma_j\}}\right)^{b_{i,j}}\right),
\end{equation*}
if assuming $\sigma_C\equiv \text{const}$ for $C\in \mathcal{C}(b)$ (this can be realized after $P$-step), we get
\begin{align*}
    \widetilde{\pi}(g\circ (\sigma,b))>0 &\iff g_i\sigma_i=g_j\sigma_j \text{ if } b_{i,j}>0\\
    &\iff g\in G(b),
\end{align*}
and that for $g,h\in G(b)$, 
\begin{equation*}
    \widetilde{\pi}(g\circ (\sigma,b))=\widetilde \pi(h\circ (\sigma,b)),
\end{equation*}
then the step of updating $\sigma'\sim \widetilde{\pi}(\cdot|b)$ is via randomly drawing $g\in G$ and setting $\sigma'=g\sigma$, where $g$ is selected according to the law
\begin{equation*}
    g|(\sigma,b) \sim \frac{\widetilde{\pi}(g\sigma,b)}{\sum_{g\in G(b)}\widetilde{\pi}(g\sigma,b)}=\frac{\widetilde{\pi}(g\circ (\sigma,b))}{\sum_{g\in G}\widetilde{\pi}(g\circ (\sigma,b))},
\end{equation*}
which is exactly the state-dependent averaging step introduced in Section \ref{subsec:state dependent averaging} with $Q=Q(G,\widetilde{\pi})$, then the associated Markov kernel of the algorithm on $\mathcal{X}\times \mathcal{B}$ is $P_{ra}=PQ$. One thing deserving noticing is that $P$ and $Q$ are non-ergodic on the extended state space $\mathcal{X}\times \mathcal{B}$, yet their composition is typically ergodic --- this demonstrates that the averaging technique can upgrade a Markov chain that is only stationary to one that is fully ergodic.

\subsubsection{Parallel tempering}

Parallel tempering (or replica exchange) algorithm, which evolves via an interacting particle system, is commonly used in molecular dynamics simulations and general optimization problems involving complex loss functions, see \citep{earl2005parallel} for a review. For a potential function $\mathcal{H}:\mathcal{X}\to \mathbb R$, under an inverse temperature $\beta>0$, we define its Gibbs distribution as 
\begin{equation*}
    \pi_{\beta}(x)\propto e^{-\beta \mathcal{H}(x)}, \quad x\in \mathcal{X}.
\end{equation*}
Given a sequence of inverse temperatures $0<\beta_1<\beta_2<\cdots<\beta_n:=\beta$, our target is to sample from the following distribution on $\mathcal{X}^n$:
\begin{equation*}
    \pi(x)\propto \prod_{i=1}^n e^{-\beta_i\mathcal{H}(x_i)}, \quad x=(x_1,\cdots,x_n)\in \mathcal{X}^n.
\end{equation*}
The algorithm in each iteration contains two steps:
\begin{enumerate}[label=(\roman*)]
    \item Level move: Uniformly choose a coordinate $i\in \llbracket n\rrbracket$ and update $x_i\to x_i'$ according to a Metropolis-Hastings move under inverse temperature $\beta_i$.

    \item Swap move: Uniformly choose $i\in \llbracket n-1\rrbracket$ and exchange their positions (i.e. $(x_i',x_{i+1}')\leftarrow (x_{i+1},x_i)$) according to an acceptance probability
    \begin{equation*}
        \alpha=\min \left\{1, \frac{\pi(x')}{\pi(x)}\right\}=\min \left\{1, e^{-(\beta_{i+1}-\beta_i)(\mathcal{H}(x_i)-\mathcal{H}(x_{i+1}))}\right\},
    \end{equation*}
    where $x=(x_1,\cdots,x_i,x_{i+1},\cdots,x_n)$ and $x'=(x_1,\cdots,x_{i+1},x_i,\cdots,x_n)$. 
\end{enumerate}

Next, denote the inverse temperature set by $\Lambda:=\{\beta_1,\cdots, \beta_n\}$, let $z_i:=(x_i,\omega_i)$, $\omega=(\omega_1,\cdots,\omega_n)\in \Lambda^n$ and $z=(z_1,\cdots,z_n)\in \mathcal{X}^n\times \Lambda^n$, define the joint distribution
\begin{equation*}
    \widetilde{\pi}(z)=\widetilde{\pi}(x,\omega)\propto \pi(x)\cdot \mathbf{1}_{\{\omega_i\neq \omega_j,\enspace \forall i\neq j\}},\quad x\in \mathcal{X}^n, \omega\in \Lambda^n,
\end{equation*}
we show that each step is some $P_{da}$ introduced in Section \ref{sec:prelim} targeting $\widetilde{\pi}$, and the whole algorithm on $\mathcal{X}^n\times \Lambda^n$ is the composition of two different $P_{da}$ that are both $\widetilde{\pi}$-stationary.

\textbf{Step (i)}: Let $P_1$ be the Markov chain that only changes the first coordinate, i.e. for $z=(z_1,\cdots,z_n)$ and $z'=(z_1',\cdots,z_n')$, define
\begin{equation*}
    P_1(z,z'):=q_{\omega_1}(x_1,x_1')\rho(x_1,x_1')\cdot \mathbf{1}_{\left\{z_{-1}=z_{-1}'\right\}}\cdot \mathbf{1}_{\left\{\omega_1=\omega_1'\right\}},
\end{equation*}
where $z_{-1}:=(x_i)_{i\neq 1}$ and $z_{-1}'=(x_i')_{i\neq 1}$, $q_{\omega_1}$ is some proposal chain on $\mathcal{X}$ under the inverse temperature $\omega_1$ and $\rho$ is the acceptance rate. Let 
\begin{align*}
    G&:=S_n, \\
    \nu_1(g,h)&:=\mu(g)\cdot \mathbf{1}_{\{h=g^{-1}\}},\quad g,h\in G,
\end{align*}
where $\mu(g)=\frac{1}{n!}$ is the uniform distribution on $G$. Define the group action
\begin{equation}\label{eq:group action in permutation}
    gz:=\left(z_{g^{-1}(1)},\cdots, z_{g^{-1}(n)}\right),
\end{equation}
then $\widetilde{\pi}$ is $G$-invariant, and the transition kernel corresponding to step (i) is 
\begin{equation*}
    K_1=\mathbb E_{(g,h)\sim \nu_1}\left(U_gP_1U_h\right)=\mathbb E_{g\sim \mu}\left(U_gP_1U_g^{-1}\right)=\overline {P_1}(G).
\end{equation*}

\textbf{Step (ii)}: Let $P_2$ be the Markov chain that swaps the first two coordinates, i.e. for $z=(z_1,\cdots,z_n)$ and $z'=(z_1',\cdots,z_n')$,
\begin{align*}
    P_2(z,z')&:=\left(\alpha(z_1,z_2)\cdot \mathbf{1}_{\left\{(x_1',x_2')=(x_2,x_1)\right\}}+(1-\alpha(z_1,z_2))\cdot \mathbf{1}_{\left\{(x_1',x_2')=(x_1,x_2)\right\}}\right)\\
    &\qquad \cdot \mathbf{1}_{\left\{z_{\llbracket n\rrbracket\backslash \{1,2\}}=z_{\llbracket n\rrbracket\backslash \{1,2\}}'\right\}}\cdot \mathbf{1}_{\left\{(\omega_1,\omega_2)=(\omega_1',\omega_2')\right\}},\\
    \text{where }& \quad \alpha(z_1,z_2)=\min \left\{1, e^{-(\omega_2-\omega_1)(\mathcal{H}(x_1)-\mathcal{H}(x_2)}\right\}.
\end{align*}
We still let $G=S_n$, and take
\begin{equation*}
    \nu_2(g,h):= \frac{1}{(n-1)!}\cdot \mathbf{1}_{\left\{g^{-1}(2)=g^{-1}(1)+1\right\}}\cdot \mathbf{1}_{\left\{h=g^{-1}\right\}}, \quad g,h\in G,
\end{equation*}
under the same action of \eqref{eq:group action in permutation}, the transition kernel of step (ii) is 
\begin{equation*}
    K_2=\mathbb E_{(g,h)\sim \nu_2}\left(U_gP_2U_h\right)=(P_2)_{da}(G,\nu_2).
\end{equation*}

Then, we can conclude that the Markov chain for the algorithm combining step (i) and (ii) is 
\begin{equation*}
    K=K_1K_2=\overline{P_1}(G)\cdot (P_2)_{da}(G,\nu_2).
\end{equation*}
This transition kernel is generally non-reversible. To get a reversible kernel, one can also consider 
\begin{equation*}
    K=\frac{1}{2}\left(K_1+K_2\right)=\frac{1}{2}\left(\overline{P_1}(G)+ (P_2)_{da}(G,\nu_2)\right).
\end{equation*}

\subsubsection{Hamiltonian Monte Carlo}

Hamiltonian Monte Carlo (HMC) is often used in sampling from continuous distribution on $\mathcal{X}=\mathbb R^d$, where a momentum variable is introduced. A recent survey can be found in \citep{neal2011mcmc}. For a potential function $U:\mathbb R^d\to \mathbb R$, the target distribution is
\begin{equation*}
    \pi(x)\propto e^{-U(x)}, \quad x\in \mathbb{R}^d.
\end{equation*}
HMC adds an auxiliary momentum variable $p\in \mathbb R^d$ following a Gaussian distribution $\mathcal{N}(0,M)$ with $M\succ 0$, and the joint distribution of $(x,p)\in \mathbb R^d\times \mathbb R^d$ is 
\begin{align*}
    \widetilde{\pi}(x,p)\propto \exp\left(-\mathcal{H}(x,p)\right), \quad \text{where } \mathcal{H}(x,p):=U(x)+\frac{1}{2}p^TM^{-1}p.
\end{align*}
Define the Hamiltonian flow as
\begin{equation*}
    \dot x=\nabla_p \mathcal{H}=M^{-1}p,\qquad 
    \dot p=-\nabla_x \mathcal{H}=-\nabla U(x),
\end{equation*}
let $\Phi_t(x,p):=(x_t,p_t)$ be the solution at time $t$ starting from initial point $(x,p)$, then it is well known that 
\begin{equation*}
    (x,p)\sim \widetilde{\pi} \quad \Longrightarrow \quad \Phi_t(x,p)\sim \widetilde{\pi}.
\end{equation*}
The algorithm is to use Leapfrog integrator as an approximation of $\Phi_t$, i.e. for fixed $\Delta T>0$, define 
\begin{equation*}
    \widehat \Phi_{\Delta T}(x,p):=\mathrm{Leapfrog}_{L,\varepsilon}(x,p),
\end{equation*}
where $\Delta T=L\varepsilon$ with $\varepsilon$ as the step size and $L$ as the step numbers in the Leapfrog integrator. Then the updating procedure is as follows:
\begin{enumerate}[label=(\roman*)]
    \item Starting from $(x,p)$, calculate $\widehat \Phi_{\Delta T}(x,p)$ and accept with probability 
    \begin{equation*}
        \alpha(x,p)=\min \left\{1, \exp\left(-\left(\mathcal{H}(\widehat \Phi_{\Delta T}(x,p))-\mathcal{H}(x,p)\right)\right)\right\}.
    \end{equation*}

    \item Refresh the momentum with $p'=-p$ or $p'=\xi$ with $\xi\sim \mathcal{N}(0,M)$ independent of $(x,p)$.
\end{enumerate}
We show that this procedure has the form of $P_{ra}=PQ$ where $Q$ is state-dependent averaging if $p'=\xi$, and the form $P_{da}$ if $p'=-p$. Take $P$ to represent the step (i), i.e.
\begin{equation*}
    P((x,p),(x',p')):=\alpha(x,p)\cdot \mathbf{1}_{\left\{(x',p')=\widehat \Phi_{\Delta T}(x,p)\right\}}+(1-\alpha(x,p))\cdot \mathbf{1}_{\left\{(x',p')=(x,p)\right\}}.
\end{equation*}
For step (ii), if the refreshed momentum is $p'=\xi\sim N(0,M)$, then take the group to be the translation group, i.e.
\begin{equation*}
    G_1:=\mathrm{Trans}(\mathbb R^d)=\left\{\tau_v: v\in \mathbb R^d\right\},
\end{equation*}
where $\tau_v: \mathbb R^d\times \mathbb R^d\to \mathbb R^d \times \mathbb R^d$ is translation map, i.e.
\begin{equation*}
    \tau_v(x,p):=(x,p+v), \quad (x,p)\in \mathbb R^d \times \mathbb R^d,
\end{equation*}
then $G_1$ is a locally compact group with Lebesgue measure as the Harr measure. We observe that 
\begin{equation*}
    \widetilde{\pi}(\tau_v(x,p))\propto \exp\left(-\frac{1}{2}(p+v)^TM^{-1}(p+v)\right), \quad \tau_v\in G_1,
\end{equation*}
then the state-dependent way of selecting $\tau_v\in G_1$ gives the conditional distribution
\begin{equation*}
    v|(x,p) \sim \mathcal{N}(-p,M),
\end{equation*}
in this case $p'=p+v\sim \mathcal{N}(0,M)$ and is independent of $(x,p)$, which is equivalent to taking $p'=\xi$. Therefore, the whole transition kernel is $PQ(G_1,\widetilde{\pi})$. Although $G_1$ is not a finite group, we stress that the state-dependent averaging technique in Section \ref{subsec:state dependent averaging} may extend readily to general groups under mild conditions.

If $p'=-p$ in step (ii), then we take the group to be the flipping group, i.e. 
\begin{equation*}
    G_2:=\mathbb Z_2=\{e,g_0\},\quad \text{where }g_0(x,p)=(x,-p), \quad (x,p)\in \mathbb R^d \times \mathbb R^d,
\end{equation*}
and take 
\begin{equation*}
    \nu(g,h):=\mathbf{1}_{\{g=e\}}\cdot\mathbf{1}_{\{h=g_0\}},\quad g,h\in G_2,
\end{equation*}
this is a deterministic jump, and $\widetilde{\pi}$ is $G_2$-invariant. Then the whole transition kernel is $P_{da}(G_2,\nu)$.

\subsubsection{Piecewise-deterministic Markov process}
Piecewise-deterministic Markov process (PDMP) \citep{davis1984piecewise} with velocity $v$ as the extended variable is a rejection-free sampler that alternates between a deterministic ODE flow and random jumps, which is similar to HMC but differs in two essential respects: the flow is simple and analytically solvable, and stationarity is enforced by random jumps instead of a Metropolis acceptance step. 

Starting from $(x,v)$, the updating procedure of PDMP is to first simulate a random jump time $\tau$ according to some prescribed distribution and calculate the flow up to time $\tau$ to reach $(x_\tau, v_\tau)$, then velocity $v$ jumps to $v'$ following some rule while maintaining $x$. For brevity, we skip the detailed constructions of the jump-time distribution and deterministic flow, and focus instead on the velocity-jumping step. Let $P$ represents the flow step of $(x,v)\to (x_\tau,v_\tau)$ which can be viewed as a discrete-time chain, and we show that for two standard PDMPs --- the bouncy particle sampler (BPS) \citep{bouchard2018bouncy} and the Zig-Zag process \citep{bierkens2019zig} --- their corresponding transition kernel can be written in the form of $P_{da}$. 

\textbf{Bouncy particle sampler}: For $(x,v)\in \mathbb R^d\times \mathbb R^d$, their joint distribution is 
\begin{equation*}
    \widetilde{\pi}(x,v)\propto \exp \left(-U(x)-\frac{1}{2}|v|^2\right),
\end{equation*}
where $U:\mathbb R^d\to \mathbb R$ is the potential function and $\pi(x)\propto e^{-U(x)}$ is the marginal. In the velocity-jumping step $(x,v)\to (x,v')$, the velocity is updated as 
\begin{equation*}
    v\to v':=v-2\cdot\frac{v^T\nabla U(x)}{|\nabla U(x)|^2}\cdot\nabla U(x),
\end{equation*}
which is a reflection on the hyperplane normal to the gradient. Similar to the case of HMC, we also take the flipping group:
\begin{equation*}
    G:=\mathbb Z_2=\{e,g_0\},
\end{equation*}
where 
\begin{equation*}
    g_0(x,v):=\left(x,v-2\cdot\frac{v^T\nabla U(x)}{|\nabla U(x)|^2}\cdot\nabla U(x)\right).
\end{equation*}
It is easy to see that this is well-defined (i.e. $g_0^2=e$ under such definition) and $\widetilde{\pi}$ is $G$-invariant. Define
\begin{equation*}
    \nu(g,h):=\mathbf{1}_{\{g=e\}}\cdot\mathbf{1}_{\{h=g_0\}},\quad g,h\in G,
\end{equation*}
then the transition kernel is $P_{da}(G,\nu)$. 

\textbf{Zig-Zag process}: For $(x,v)\in \mathbb R^d\times \{-1,1\}^d$, the joint distribution is 
\begin{equation*}
    \widetilde{\pi}(x,v)\propto e^{-U(x)},
\end{equation*}
which means $v$ follows the uniform distribution in $\{-1,1\}^d$. The velocity jumping is $(x',v')=(x,-v)$, hence it is direct to see that $G$ can also taken to be $\mathbb Z_2=\{e,g_0\}$ with $g_0(x,v)=(x,-v)$. We also define $\nu(g,h):=\mathbf{1}_{\{g=e\}}\cdot\mathbf{1}_{\{h=g_0\}}$, then the transition kernel is $P_{da}(G,\nu)$.

For other algorithms of PDMP such as Boomerang sampler \citep{bierkens2020boomerang} and event chain Monte Carlo \citep{krauth2021event}, one can easily construct the group and averaged kernels to characterize the samplers in an analogous way of the above two examples. 

\subsubsection{Markov chains with deterministic jumps}

Adding a deterministic jump before each step of a Markov chain can remarkably accelerate mixing. Apart from the samplers listed before, a notable breakthrough is \citep{chatterjee2021correction}, which shows that on finite state space $\mathcal{X}=\llbracket n\rrbracket$, for most of the permutation matrices $S$ on $\llbracket n\rrbracket$, the chain $SP$ mixes much faster than $P$ (both stationary w.r.t. the uniform distribution on $\mathcal{X}$). Under the same setting, for $\pi$ as the uniform distribution and $P$ as a $\pi$-stationary Markov chain, \citep{bordenave2019cutoff} gives a sharp characterization of the worst-case TV mixing time of $SP$: for most $S$, $SP$ exhibits the cutoff phenomenon with cutoff time at 
\begin{equation*}
    t=\frac{\log n}{\mathfrak{h}}, \quad \text{where }\mathfrak{h}=\log n-D_{KL}^{\pi}\left(P\|\Pi\right).
\end{equation*}
In particular, if $P$ is a simple random walk on $\llbracket n\rrbracket$, then its mixing time is $\Theta(n^2)$, while the above two references both show that $SP$ can mix in $\mathcal{O}(\log n)$ steps for most choices of $S$. 

A naive construction of $G$ to fit the above framework is to take $G=S_n$ the permutation group on $\llbracket n\rrbracket$, and define 
\begin{equation*}
    g_0x:=s(x),\quad \text{where }S(x,s(x))=1, \enspace x\in \llbracket n\rrbracket,
\end{equation*}
so that $SP=U_{g_0}P$, a special case of $P_{da}$. However, this choice offers little practical implication because $S_n$ is too large ($|S_n|=n!$). If one can identify a much smaller subgroup $G_1\leq S_n$ such that $g_0\in G_1$, then $SP$ can be further improved via $QP$ the uniform averaging over $G_1$, which is computationally feasible. The reason that we take the group containing $g_0$ instead of arbitrary groups of similar size comes from an heuristic perspective: if the jump $g_0$ is already known to accelerate the chain, one can intuitively expect its iterates $g_0^k$ to be similarly useful (this is exemplified by \eqref{eq:chung-diaconis-graham linear} and references below), thus it is prudent to secure the gains via averaging over the group containing these, such as the cyclic group $\langle g_0\rangle$ generated by $g_0$, while the benefits of unrelated groups can be uncertain.

Now we provide some examples where a small subgroup containing $g_0$ can be found. We consider the Chung-Diaconis-Graham chain \citep{chung1987random} and its many variants to sample from the uniform distribution $\pi$ on finite state space $\mathcal{X}=\llbracket n\rrbracket$. For $k\geq 0$ and $a\in \llbracket n\rrbracket$, let
\begin{equation}\label{eq:chung-diaconis-graham linear}
    X_{k+1}=aX_k+\varepsilon_{k+1} \enspace (\text{mod } n),
\end{equation}
where $\varepsilon_k\sim \pi_0$ are i.i.d. random variables. This defines a non-reversible chain admitting $\pi$ as the stationary distribution if $\text{gcd}(a,n)=1$. If $a=2$ and $\pi_0$ is the uniform distribution on $\{-1,0,1\}$, \eqref{eq:chung-diaconis-graham linear} covers the classical chain in \citep{chung1987random} with mixing time of $\Theta(\log n)$ for almost all odd $n$. For most of $n$ such that $\text{gcd}(a,n)=1$ with $a\geq 2$, it is shown in \citep{eberhard2021mixing} that \eqref{eq:chung-diaconis-graham linear} exhibits a cutoff with cutoff time of order $\Theta(\log n)$. In this case, let $P$ denotes the transition kernel corresponding to taking $a=1$ in \eqref{eq:chung-diaconis-graham linear}, and define
\begin{equation*}
    g_0x:=ax\enspace (\text{mod }n), \quad x\in \llbracket n\rrbracket,
\end{equation*}
then the transition kernel of \eqref{eq:chung-diaconis-graham linear} is $U_{g_0}P$. It is easy to see that 
\begin{equation*}
    g_0^{\varphi(n)}=e,
\end{equation*}
where $\varphi (n)$ is Euler's totient function. Therefore, $g_0$ belongs to the cyclic group 
\begin{equation*}
    G_1:=\langle g_0\rangle=\left\{g_0^k: 1\leq k\leq \varphi(n)\right\}.
\end{equation*}
It is well known that $\varphi(n)\leq n-1$, hence $|G_1|=\text{ord}(g_0)\leq n-1\ll |S_n|$. Averaging \eqref{eq:chung-diaconis-graham linear} over $G_1$ yields a kernel $K$, i.e.
\begin{equation*}
    K=QU_{g_0}P=\frac{1}{|G_1|}\sum_{g\in G_1}U_gU_{g_0}P=\frac{1}{\varphi(n)}\sum_{k=1}^{\varphi(n)}U_{g_0^k} P,
\end{equation*}
which has a better multiplicative spectral gap than \eqref{eq:chung-diaconis-graham linear}. 

More generally, we may allow a non-linear jump at each step, i.e.
\begin{equation}\label{eq:chung non-linear}
    X_{k+1}=f(X_k)+\varepsilon_{k+1}\enspace (\text{mod }n),
\end{equation}
where $f:\llbracket n\rrbracket \to \llbracket n\rrbracket$ is a bijection. We then define $g_0x:=f(x)$. For an arbitrary $f$, the stationary distribution for \eqref{eq:chung non-linear} can be hard to identify, so we restrict our attention to a few representative choices of $f$ for which stationary distribution can be explicitly established. Now we assume $n$ is a \textbf{prime}. If $f=f_1/f_2$ for some coprime $f_1,f_2\in \mathbb F_n[x]$ such that $f$ is a bijection and not a linear function, then for some certain $\pi_0$ the distribution of $\varepsilon_k$, \citep{he2022markov} shows that the lazified version of \eqref{eq:chung non-linear} has the mixing time of $\mathcal{O}(n^{1+\varepsilon})$ for any $\varepsilon>0$ (although stationary distribution may not be uniform). Here is an example of such $f$:
\begin{itemize}
    \item $f(x)=ax^k$: For $a\in \mathbb F_n^{\times}$ and $\text{gcd}(k,n-1)=1$, define $m\in \llbracket n\rrbracket$ such that $mk=1 \enspace (\text{mod }n-1)$, then $f$ is a bijection with $f^{-1}(x)=(a^{-1}x)^m$. For the $g_0$ induced by $f$, one can readily verify that $g_0$ belongs to the following group
    \begin{align*}
        G_1&:=\left\{f_{a,k}:x\to ax^k \big | a\in \mathbb F_n^{\times}, \text{gcd}(k,n-1)=1\right\} \\
        & \cong \mathbb F_{n}^{\times}\rtimes R_{n-1},
    \end{align*}
    where $R_{n-1}:=\left\{k\in \mathbb Z/(n-1)\mathbb Z:\text{gcd}(k,n-1)=1\right\}$ is the reduced residue system modulo $n-1$, and the semi-direct product is defined to be $(a,u)\cdot (b,v):=(ab^u,uv)$ for $a,b\in \mathbb F_{n}^{\times}$ and $u,v\in R_{n-1}$. Then $|G_1|=(n-1)\varphi(n-1)$, which is also much smaller than $|S_n|$. Since $\langle g_0\rangle\leq G_1$, we have $|\langle g_0\rangle|\leq (n-1)\varphi(n-1)$, and thus one can similarly average over $\langle g_0\rangle$ to get an improved kernel, i.e.
    \begin{equation*}
        K=\frac{1}{|\langle g_0\rangle|}\sum_{g\in \langle g_0\rangle}U_gU_{g_0}P=\frac{1}{(n-1)\varphi(n-1)}\sum_{k=1}^{(n-1)\varphi(n-1)}U_{g_0^k}P,
    \end{equation*}
    where $P$ is the transition kernel taking $f=\text{id}.$ in \eqref{eq:chung non-linear}.
\end{itemize}

\subsubsection{A counter-example}
In previous sections, we have shown that when $\pi$ is $G$-invariant, uniform averaging over $G$ can be optimal in enlarging the (multiplicative) spectral gap. However, spectral gap may even remain zero after averaging, and it is still far from precisely characterizing mixing times. The Diaconis-Holmes-Neal sampler \citep{diaconis2000analysis} illustrates this possible phenomenon. This non-reversible chain can be written as some $P_{da}$, and its uniformly right-averaged kernel $(P_{da})_{ra}$ has the multiplicative spectral gap of $0$, just like $P_{da}$. To be worse, $(P_{da})_{ra}$ mixes much more slowly than $P_{da}$: the worst-case TV mixing time deteriorates from order $n$ to $n^2$.

On the $2n$-cycle $\mathcal{X}= \llbracket 2n \rrbracket$, the chain $K$ is defined to be 
\begin{gather*}
    K(x,x+1)=1-\frac{1}{n},\quad K(x,2n-x)=\frac{1}{n}, \quad x \in \llbracket 2n-1 \rrbracket,\\
    K(2n,2n)=K(n,n)=\frac{1}{n}, \quad K(2n,1)=K(n,n+1)=1-\frac{1}{n}.
\end{gather*}
Now, we take $P$ to be the deterministic move on the cycle, i.e.
\begin{equation*}
    P(x,x+1)=1, \quad x \in \llbracket 2n-1 \rrbracket, \quad \text{and}\quad P(2n,1)=1. 
\end{equation*}
Let $G:=\mathbb Z_2=\{e,g_0\}$ be the flipping group, where 
\begin{equation*}
    g_0x:=2n+1-x, \quad x \in \llbracket 2n \rrbracket,
\end{equation*}
and define
\begin{equation*}
    \nu(g,h):=\mathbf{1}_{\{g=e\}}\cdot \left(\frac{1}{n}\cdot\mathbf{1}_{\{h=g_0\}}+\left(1-\frac{1}{n}\right)\cdot \mathbf{1}_{\{h=e\}}\right),
\end{equation*}
then 
\begin{equation*}
    K=P_{da}(G,\nu)=\frac{1}{n}PU_{g_0}+\left(1-\frac{1}{n}\right)P. 
\end{equation*}
To get a spectral improvement, one can take the uniformly right-averaged kernel, i.e.
\begin{equation*}
    K_{ra}=KQ=\frac{1}{2}PU_{g_0}+\frac{1}{2}P,
\end{equation*}
which is equivalent to substituting the change rate from $1/n$ to $1/2$ in $K$. The method to calculate the multiplicative spectral gaps of $K$ and $K_{ra}$ is similar to \citep{diaconis2000analysis}. Let $p$ be the change rate, i.e. $p=1/n$ for $K$ and $p=1/2$ for $K_{ra}$. Take the Fourier basis $\{u_h: -(n-1)\leq h\leq n\}$:
\begin{gather*}
    u_h(x)=\frac{1}{\sqrt{2n}}e^{i\theta_hx},\quad u_{-h}(x)=\frac{1}{\sqrt{2n}}e^{-i\theta_hx}, \quad \text{where }\theta_h=\frac{\pi h}{n}, \quad 1\leq h\leq n-1,\\
    u_0(x)=\frac{1}{\sqrt{2n}}, \quad u_n(x)=\frac{1}{\sqrt{2n}}(-1)^x,
\end{gather*}
under the basis $\{u_h,u_{-h}\}$ for $1\leq h\leq n-1$, the corresponding diagonalized block is 
\begin{equation*}
    K_p(h)=\begin{pmatrix}
        (1-p)e^{i\theta_h} & p\\[3pt]
        p & (1-p)e^{-i\theta_h}
    \end{pmatrix},
\end{equation*}
and the eigenvalues associated to $u_0$ and $u_n$ are $1$ and $2p-1$ respectively. Moreover, 
\begin{equation*}
    K_p(h)^*K_p(h)=\begin{pmatrix}
        p^2+(1-p)^2 & 2p(1-p)e^{-i\theta_h} \\[3pt]
        2p(1-p)e^{i\theta_h} & p^2+(1-p)^2
    \end{pmatrix},
\end{equation*}
whose eigenvalues are $1$ and $(2p-1)^2$. Therefore, the multiplicative spectral gaps of $K$ and $K_{ra}$ are both $0$. 

According to \citep{diaconis2000analysis}, the worst-case TV mixing time of $K$ is $\Theta(n)$. Now we show that the mixing time of $K_{ra}$ is at least of order $n^2$. We rewrite $\mathcal{X}=\llbracket n\rrbracket \times \{-1,1\}$ to represent the state space, and use $X_t=(Y_t,D_t)\in \llbracket n\rrbracket\times \{-1,1\}$ to denote the chain corresponding to $K_{ra}$, where $D_t$ can be understood as the direction. Since $Y_t$ is a function of $X_t$, the TV mixing time of $X_t$ is lower bounded by that of $Y_t$ with $\pi_1(x)=1/n$ as its stationary distribution on $\llbracket n\rrbracket$, where we have used the data processing inequality (DPI) for TV distance, and a most related form of DPI can be found in \citep[Lemma A.2]{boursier2023universal}. The update of $Y_t$ follows:
\begin{equation*}
    Y_{t+1}=\begin{cases}
        Y_t+D_t, &\text{if } 2\leq Y_t\leq n-1,\\
        Y_t, &\text{if } Y_t=1, D_t=-1, \\
        Y_t+1, &\text{if } Y_t=1, D_t=1,\\
        Y_t, &\text{if } Y_t=n, D_t=1,\\
        Y_t-1, &\text{if } Y_t=n, D_t=-1,
    \end{cases}
\end{equation*}
and $\{D_t\}_{t=0}^\infty$ are i.i.d. random variables with equal probability of $-1$ and $1$. Then, $Y_t\in \sigma(D_0,D_1,\cdots, D_{t-1})$ the sigma algebra generated by $D_0,\cdots,D_{t-1}$, and $Y_t$ and $D_t$ are independent. Therefore, $\{Y_t\}_{t=0}^\infty$ has the same distribution with the simple random walk on $\llbracket n\rrbracket$ with reflection on boundaries, whose TV mixing time is well known to be $\Theta(n^2)$. Thus we can conclude that $X_t$ mixes in at least order $n^2$ steps under TV distance.

\subsection{Achieving $P_{la} = P_{ra} = (P_{la})_{ra} = \Pi$ for discrete uniform $\pi$}\label{subsec:achievediscreteuniformpi}

This subsection shows that on a finite state space with discrete uniform $\pi$, it is possible to achieve $P_{la} = P_{ra} = (P_{la})_{ra} = \Pi$ when $P$ is any $\pi$-stationary Markov kernel with a suitable choice of the group $G$. Specifically, let $n \in \mathbb{N}$ and without loss of generality we write $\mathcal{X} = \llbracket n \rrbracket$. We define $g$ to be, for $x \in \mathcal{X}$,
\begin{align*}
	g x = x + 1, g n = 1,
\end{align*}
that is, the right shift by one action with a periodic boundary condition at $n$. Clearly, $g^n = e$, and we consider the group $G$ generated by $g$ such that
\begin{align}\label{eq:uniformG}
	G = \{e, g, g^2, \ldots, g^{n-1}\}.
\end{align}

We now state the main results of this subsection:

\begin{proposition}\label{prop:uniform}
	Let $\pi$ be the discrete uniform distribution on $\mathcal{X} = \llbracket n \rrbracket$, and consider the group $G$ as in \eqref{eq:uniformG}. For $P \in \mathcal{S}(\pi)$, we have
	$$P_{la} = \Pi.$$
	Consequently, $P_{ra} = (P_{la})_{ra} = \Pi$, and hence, for any $\varepsilon > 0$ and $1 \leq p \leq \infty$,
	\begin{align*}
		t_{\mathrm{mix},p}(P_{la},\varepsilon) = t_{\mathrm{mix},p}(P_{ra},\varepsilon) = t_{\mathrm{mix},p}((P_{la})_{ra},\varepsilon) = 1.
	\end{align*}
\end{proposition}

\begin{proof}
	First, if $P_{la} = \Pi$, then it is immediate to see that $(P_{la})_{ra} = \Pi_{ra} = \Pi$. By replacing $P$ with $P^* = P^T$ which is also $\pi$-stationary, we also have $(P^*)_{la} = \Pi$. Using Proposition \ref{prop:inheritance of properties} we thus arrive at $\Pi = \Pi^* = ((P^*)_{la})^* = P_{ra}$. The mixing time statements are obvious as the Markov kernels are exactly $\Pi$.
	
	It thus suffices to prove $P_{la} = \Pi$. For any $x,y \in \llbracket n \rrbracket$, we have
	\begin{align*}
		P_{la}(x,y) = \dfrac{1}{n} \sum_{i = 0}^{n-1} P(g^i x , y) = \dfrac{1}{n} \sum_{i = 0}^{n-1} P^*(y,g^i x) = \dfrac{1}{n} \sum_{z \in \mathcal{X}} P^*(y , z) = \dfrac{1}{n},
	\end{align*}
	and hence $P_{la} = \Pi$.
\end{proof}

We discuss three remarks concerning Proposition \ref{prop:uniform}. 

First, with the choice of $G$ as stated this result applies to any discrete uniform $\pi$ and $\pi$-stationary $P$, showing that it is possible to achieve stationarity in only one projection, and hence the mixing times are precisely one, which is independent of $n$. As a concrete example, this result can be applied to the Diaconis-Holmes-Neal sampler \cite{diaconis2000analysis}, thus improving its mixing time from linear in $n$ to one.

Second, the essence of $P_{la}$ and the group $G$ chosen is that it permutes the initial state into a randomized state over the entire state space $\mathcal{X}$. Thus, to simulate $P_{la}$ in this context, we would need to draw uniformly at random an element from $G$. In other words, we need to sample from the discrete uniform $\pi$ in order to simulate $P_{la}$.

Third, on a finite state space we recall that the projections studied in \cite{choi2025improving} is trace-preserving, thus stationarity can be achieved through projections are limited to $P$ such that $\mathrm{Tr}(P) = 1$. On the other hand, as demonstrated in Proposition \ref{prop:inheritance of properties} and its following remark, $P_{la}, P_{ra}, (P_{la})_{ra}$ do not necessarily preserve the trace of $P$, and hence stationarity can possibly be achieved through projections even for $P$ such that $\mathrm{Tr}(P) \neq 1$.

\subsection{Improving Metropolis-Hastings on a discrete bimodal V-shaped distribution}\label{subsec:improveMH}

A common benchmark target distribution on finite state space is the bimodal V-shaped Gibbs distribution $\pi_\beta$ as studied in the swapping algorithm \cite{madras2003swapping} and the Diaconis-Holmes-Neal sampler \cite{diaconis2000analysis}. This subsection demonstrates that it is possible to improve the Metropolis-Hastings sampler for such target from exponential to polynomial mixing time in the size of the state space, see Proposition \ref{prop:improveMH} below.

Let $n \in \mathbb{N}$ and consider the state space $\mathcal{X} = \llbracket -n,n \rrbracket$, the Hamiltonian function $\mathcal{H}(x) := -|x|$ for $x \in \mathcal{X}$ and its associated Gibbs distribution at inverse temperature $\beta \geq 0$:
\begin{align*}
	\pi_\beta(x) \propto e^{-\beta \mathcal{H}(x)},
\end{align*}
with $Z_\beta := \sum_{x \in \mathcal{X}} e^{-\beta \mathcal{H}(x)}$ being the normalization constant. Let $P_0$ be the proposal Markov kernel with $P_0(n,n) = P_0(-n,-n) = 1/2$, $P_0(x,x+1) = P_0(x+1,x) = 1/2$ for $x \in \llbracket -n,n-1\rrbracket$, a nearest-neighbour simple random walk. The Metropolis-Hastings Markov kernel $P_\beta$ with such proposal $P_0$ and target $\pi_\beta$ is defined to be, for $x \in \llbracket -n,n-1\rrbracket$,
\begin{align*}
	P_\beta(x,x+1) = \dfrac{1}{2} e^{-\beta (\mathcal{H}(x+1) - \mathcal{H}(x))_+}, \quad P_\beta(x+1,x) = \dfrac{1}{2} e^{-\beta (\mathcal{H}(x) - \mathcal{H}(x+1))_+},
\end{align*}
and the diagonal entries of $P_\beta$ are such that each row sums to one.

In the above context, a natural group is given by $G = \{e,g\}$ where $gx := -x$ for all $x \in \mathcal{X}$. It can be readily verified that $\pi_\beta$ is $G$-invariant, and that
\begin{align*}
	P_\beta = U_g P_\beta U_g^{-1} = U_g P_\beta U_g.
\end{align*}
Consequently, we note that
\begin{align*}
	P_\beta = \overline{P_\beta} = \widetilde{P_\beta},
\end{align*}
that is, $P_\beta \in \mathcal{L}(G,G) \cap \mathcal{L}(G,G^{-1})$. As such the theory and techniques developed in \cite{choi2025improving} yield no improvement. On the other hand, we compute that
\begin{align*}
	(P_\beta)_{la} &= \dfrac{1}{2}(P_\beta + U_gP_\beta) \neq P_{\beta},\\
	(P_\beta)_{ra} &= \dfrac{1}{2}(P_\beta + P_\beta U_g) \neq P_{\beta}, \\
	((P_\beta)_{la})_{ra} &= \dfrac{1}{2}P_\beta + \dfrac{1}{4} U_g P_\beta + \dfrac{1}{4} P_\beta U_g \neq P_{\beta}.
\end{align*}

One of the main results of this section gives a polynomial in $n$ upper bound on the relaxation time based on the right spectral gap:

\begin{proposition}\label{prop:spectralgapVshape}
	In the setting of this subsection, we have
	\begin{align*}
		\lambda(((P_\beta)_{la})_{ra}) \geq \dfrac{1-e^{-\beta}}{36n^3}.
	\end{align*}
	where we recall $\lambda(P)$ is the right spectral gap of $P$ as defined in \eqref{eq:rightspectralgap}.
\end{proposition}

\begin{proof}
	For $x \neq y \in \mathcal{X}$, let $(p^{x,y}_i)_{i=1}^{n(x,y)}$ be a path from $p^{x,y}_1 = x$ to $p^{x,y}_{n(x,y)} = y$ of length $n(x,y)$. We select the paths in the following manner:
	\begin{itemize}
		\item Case $1$: $x \neq 0$ and $xy \geq 0$.
		In this case, we have either $x < 0, y \leq 0$ or $x > 0, y \geq 0$. If $\mathcal{H}(x) \geq \mathcal{H}(y)$ (resp.~$\mathcal{H}(x) < \mathcal{H}(y)$), we follow the descent (resp.~ascent) path using $P_{\beta}$, leading to
		\begin{align*}
			n(x,y) &\leq n,\quad \max_{i \in \llbracket n(x,y) \rrbracket} \mathcal{H}(p^{x,y}_i) \leq \max\{\mathcal{H}(x), \mathcal{H}(y)\}, \\
			\pi_\beta(p^{x,y}_i) ((P_\beta)_{la})_{ra}(p^{x,y}_i,p^{x,y}_{i+1}) &\geq \dfrac{1}{2Z_\beta} e^{-\beta\max\{\mathcal{H}(x), \mathcal{H}(y)\}}.
		\end{align*}
		
		\item Case $2$: $x \neq 0$ and $xy < 0$.
		In this case, we have either $x < 0, y > 0$ or $x > 0, y < 0$. We first consider $U_g P_\beta$ to flip from $x$ to $-x$, followed by the descent or ascent path using $P_{\beta}$, leading to
		\begin{align*}
			n(x,y) &\leq n,\quad \max_{i \in \llbracket n(x,y) \rrbracket} \mathcal{H}(p^{x,y}_i) \leq \max\{\mathcal{H}(x), \mathcal{H}(y)\}, \\
			\pi_\beta(p^{x,y}_i) ((P_\beta)_{la})_{ra}(p^{x,y}_i,p^{x,y}_{i+1}) &\geq \dfrac{1}{4Z_\beta} e^{-\beta\max\{\mathcal{H}(x), \mathcal{H}(y)\}}(1-e^{-\beta}).
		\end{align*}
		
		\item Case $3$: $x = 0$.
		In these cases, we consider the descent path using $P_{\beta}$, leading to
		\begin{align*}
			n(x,y) &\leq n,\quad \max_{i \in \llbracket n(x,y) \rrbracket} \mathcal{H}(p^{x,y}_i) = \mathcal{H}(0) \leq \max\{\mathcal{H}(x), \mathcal{H}(y)\}, \\
			\pi_\beta(p^{x,y}_i) ((P_\beta)_{la})_{ra}(p^{x,y}_i,p^{x,y}_{i+1}) &\geq \dfrac{1}{2Z_\beta} e^{-\beta\max\{\mathcal{H}(x), \mathcal{H}(y)\}}.
		\end{align*}
	\end{itemize}
	Let $f \in L^2_0(\pi_\beta)$, and $\chi_{z,w}(x,y)$ be $1$ if there exists some $i \in \llbracket n(x,y) \rrbracket$ such that $p^{x,y}_{i} = z, p^{x,y}_{i+1} = w$ and $0$ otherwise. We compute that
	\begin{align*}
		\langle f,f \rangle_{\pi_{\beta}} &= \dfrac{1}{2} \sum_{x,y} (f(y) - f(x))^2 \pi_\beta(y) \pi_\beta(x) \\
		&= \dfrac{1}{2} \sum_{x,y} \left(\sum_{i=1}^{n(x,y)} f(p^{x,y}_{i+1}) - f(p^{x,y}_{i})\right)^2 \pi_\beta(y) \pi_\beta(x) \\
		&\leq \dfrac{n}{2} \sum_{x,y} \sum_{i=1}^{n(x,y)} \left(f(p^{x,y}_{i+1}) - f(p^{x,y}_{i})\right)^2 \pi_\beta(y) \pi_\beta(x) \\
		&\leq \dfrac{n}{2} \sum_{x,y} \sum_{z,w} \chi_{z,w}(x,y)\left(f(w) - f(z)\right)^2 \pi_\beta(z) ((P_\beta)_{la})_{ra}(z,w) \dfrac{4Z_{\beta} e^{\beta\max\{\mathcal{H}(x), \mathcal{H}(y)\}}\pi_\beta(y) \pi_\beta(x)}{1-e^{-\beta}} \\
		&\leq n \left(\max_{z,w} \sum_{x,y} \chi_{z,w}(x,y) \dfrac{4Z_{\beta} e^{\beta\max\{\mathcal{H}(x), \mathcal{H}(y)\}}\pi_\beta(y) \pi_\beta(x)}{1-e^{-\beta}} \right) \\
		&\quad \times \left(\dfrac{1}{2}\sum_{z,w} \left(f(w) - f(z)\right)^2 \pi_\beta(z) ((P_\beta)_{la})_{ra}(z,w)\right) \\
		&\leq n \left((2n+1)^2 \dfrac{4}{1-e^{-\beta}}\right) \langle f, (I - ((P_\beta)_{la})_{ra})[f]\rangle_{\pi_\beta} \\
		&\leq \dfrac{36n^3}{1-e^{-\beta}} \langle f, (I - ((P_\beta)_{la})_{ra})[f]\rangle_{\pi_\beta}.
	\end{align*}
	Rearranging gives the desired inequality.
\end{proof}

Denote the lazy Markov kernel of $((P_\beta)_{la})_{ra}$ to be 
\begin{align*}
	L_\beta := \dfrac{1}{2}(I + ((P_\beta)_{la})_{ra}).
\end{align*}

Another main result of this section demonstrates that $L_\beta$ enjoys rapid (i.e. polynomial in $n$) mixing time while $P_\beta$ has a torpid (i.e. exponential in $n$) mixing time.

\begin{proposition}\label{prop:improveMH}
	In the setting of this subsection, for $\varepsilon > 0$ we have
	\begin{align*}
		t_{\mathrm{mix},1}(L_\beta,\varepsilon) &\leq \dfrac{72n^3}{1-e^{-\beta}}\left(\beta n +  \log \left(\dfrac{2n+1}{\varepsilon}\right)\right), \\
		t_{\mathrm{mix},1}(P_\beta,\varepsilon) &\geq \left(\dfrac{e^{\beta n}}{(2n+1)^2} - 1\right)\log \left(\dfrac{1}{\varepsilon}\right).
	\end{align*}
\end{proposition} 

\begin{proof}
	Using Proposition \ref{prop:spectralgapVshape}, we see that
	\begin{align*}
		\lambda(L_\beta) \geq \dfrac{1-e^{-\beta}}{72n^3}.
	\end{align*}
	Making use of \cite[Theorem $12.4$]{levin2017markov}, the worst-case $L^1$ mixing time of $L_\beta$ is
	\begin{align*}
		t_{\mathrm{mix},1}(L_\beta,\varepsilon) \leq \dfrac{1}{\lambda(L_\beta)} \log \left(\dfrac{1}{\varepsilon \min_x \pi_\beta(x)}\right) \leq \dfrac{72n^3}{1-e^{-\beta}}\left(\beta n +  \log \left(\dfrac{2n+1}{\varepsilon}\right)\right).
	\end{align*}
	On the other hand, by noting that the so-called critical height of $P_\beta$ is $n$, applying \cite[Lemma $2.3$]{holley1988simulated} leads to
	\begin{align*}
		\lambda(P_\beta) \leq (2n+1)^2 e^{-\beta n},
	\end{align*}
	and by \cite[Theorem $12.5$]{levin2017markov}, we arrive at
	\begin{align*}
		t_{\mathrm{mix},1}(P_\beta,\varepsilon) \geq \left(\dfrac{e^{\beta n}}{(2n+1)^2} - 1\right)\log \left(\dfrac{1}{\varepsilon}\right).
	\end{align*}
\end{proof}

\subsubsection{Improving Metropolis-Hastings on a non-symmetric discrete V-shaped distribution via state-dependent averaging and group planting}

In the previous subsection, we consider a V-shaped Gibbs distribution $\pi_\beta$ which is $G$-invariant, where $G$ is the group generated by the action of multiplying by negative one. In this subsection, we consider a Hamiltonian $\mathcal{H}_\delta$ which is perturbed by a parameter $\delta$, making its associated Gibbs distribution to be non-$G$-invariant. To overcome this, we apply the state-dependent averaging technique by planting the group $G$ as discussed in Section \ref{subsec:state dependent averaging}. We show that the resulting Markov kernel has a polynomial mixing time in the system size in Proposition \ref{prop:improveMHperturb} below.

Consider the state space $\mathcal{X} = \llbracket -n,n \rrbracket$ with $n \in \mathbb{N}$, the Hamiltonian function 
\begin{align*}
	\mathcal{H}_\delta(x) := -|x+\delta|
\end{align*} 
for $x \in \mathcal{X}$, $\delta \in (0,\frac{1}{2})$ and its associated Gibbs distribution at inverse temperature $\beta \geq 0$:
\begin{align*}
	\pi_{\beta,\delta}(x) \propto e^{-\beta \mathcal{H}_\delta(x)},
\end{align*}
with $Z_{\beta,\delta} := \sum_{x \in \mathcal{X}} e^{-\beta \mathcal{H}_\delta(x)}$ being the normalization constant. We use the same $P_0$, a nearest-neighbour simple random walk, as the proposal kernel. The Metropolis-Hastings Markov kernel $P_{\beta,\delta}$ with such proposal $P_0$ and target $\pi_{\beta,\delta}$ is defined to be, for $x \in \llbracket -n,n-1\rrbracket$,
\begin{align*}
	P_{\beta,\delta}(x,x+1) = \dfrac{1}{2} e^{-\beta (\mathcal{H}_\delta(x+1) - \mathcal{H}_\delta(x))_+}, \quad P_{\beta,\delta}(x+1,x) = \dfrac{1}{2} e^{-\beta (\mathcal{H}_\delta(x) - \mathcal{H}_\delta(x+1))_+},
\end{align*}
and the diagonal entries of $P_{\beta,\delta}$ are such that each row sums to one.

We consider the same group $G$ as in the previous subsection, which is given by $G = \{e,g\}$ where $gx := -x$ for all $x \in \mathcal{X}$. However, $\pi_{\beta,\delta}$ is in general non-$G$-invariant. Also, there may not exist equi-probability pair of states as in \cite{choi2025improving}. On the other hand, we compute the state-dependent averaging Markov kernels to be, for $x,y \in \mathcal{X}$,
\begin{align*}
	(P_{\beta,\delta})_{la}(x,y) &= \dfrac{\pi_{\beta,\delta}(x)}{\pi_{\beta,\delta}(x) + \pi_{\beta,\delta}(-x)}P_{\beta,\delta}(x,y) + \dfrac{\pi_{\beta,\delta}(-x)}{\pi_{\beta,\delta}(x) + \pi_{\beta,\delta}(-x)}P_{\beta,\delta}(-x,y),\\
	(P_{\beta,\delta})_{ra}(x,y) &= \dfrac{\pi_{\beta,\delta}(y)}{\pi_{\beta,\delta}(y) + \pi_{\beta,\delta}(-y)}P_{\beta,\delta}(x,y) + \dfrac{\pi_{\beta,\delta}(-y)}{\pi_{\beta,\delta}(y) + \pi_{\beta,\delta}(-y)}P_{\beta,\delta}(x,-y), \\
	((P_{\beta,\delta})_{la})_{ra}(x,y) &= \dfrac{\pi_{\beta,\delta}(x)\pi_{\beta,\delta}(y)}{(\pi_{\beta,\delta}(x) + \pi_{\beta,\delta}(-x))(\pi_{\beta,\delta}(y) + \pi_{\beta,\delta}(-y))}P_{\beta,\delta}(x,y) \\
	&\quad + \dfrac{\pi_{\beta,\delta}(x)\pi_{\beta,\delta}(-y)}{(\pi_{\beta,\delta}(x) + \pi_{\beta,\delta}(-x))(\pi_{\beta,\delta}(y) + \pi_{\beta,\delta}(-y))}P_{\beta,\delta}(x,-y) \\
	&\quad +
	\dfrac{\pi_{\beta,\delta}(-x)\pi_{\beta,\delta}(y)}{(\pi_{\beta,\delta}(x) + \pi_{\beta,\delta}(-x))(\pi_{\beta,\delta}(y) + \pi_{\beta,\delta}(-y))}P_{\beta,\delta}(-x,y) \\
	&\quad + \dfrac{\pi_{\beta,\delta}(-x)\pi_{\beta,\delta}(-y)}{(\pi_{\beta,\delta}(x) + \pi_{\beta,\delta}(-x))(\pi_{\beta,\delta}(y) + \pi_{\beta,\delta}(-y))}P_{\beta,\delta}(-x,-y).
\end{align*}

One of the main results of this subsection gives a polynomial in $n$ upper bound on the relaxation time based on the right spectral gap:

\begin{proposition}\label{prop:spectralgapVshapeperturb}
	In the setting of this subsection, we have
	\begin{align*}
		\lambda(((P_{\beta,\delta})_{la})_{ra}) \geq \dfrac{1-e^{-\beta}}{36n^3e^{\beta 2\delta}}.
	\end{align*}
	where we recall $\lambda(P)$ is the right spectral gap of $P$ as defined in \eqref{eq:rightspectralgap}.
\end{proposition}

\begin{proof}
	For $x \neq y \in \mathcal{X}$, let $(p^{x,y}_i)_{i=1}^{n(x,y)}$ be a path from $p^{x,y}_1 = x$ to $p^{x,y}_{n(x,y)} = y$ of length $n(x,y)$. We select the paths in the following manner:
	\begin{itemize}
		\item Case $1$: $x > 0, y \geq 0$.
		If $\mathcal{H}_\delta(x) \geq \mathcal{H}_\delta(y)$ (resp.~$\mathcal{H}_\delta(x) < \mathcal{H}_\delta(y)$), we follow the descent (resp.~ascent) path using $U_e P_{\beta,\delta} U_e^{-1}$, leading to
		\begin{align*}
			n(x,y) &\leq n,\quad \max_{i \in \llbracket n(x,y) \rrbracket} \mathcal{H}_\delta(p^{x,y}_i) \leq \max\{\mathcal{H}_\delta(x), \mathcal{H}_\delta(y)\}, \\
			\pi_{\beta,\delta}(p^{x,y}_i) ((P_{\beta,\delta})_{la})_{ra}(p^{x,y}_i,p^{x,y}_{i+1}) &\geq \dfrac{1}{4Z_{\beta,\delta}} e^{-\beta\max\{\mathcal{H}_\delta(x), \mathcal{H}_\delta(y)\}}.
		\end{align*}
		
		\item Case $2$: $x < 0, y \leq 0$. We first consider $U_g P_{\beta,\delta}$ to flip from $x$ to $-x$, then if $\mathcal{H}_\delta(x) \geq \mathcal{H}_\delta(y)$ (resp.~$\mathcal{H}_\delta(x) < \mathcal{H}_\delta(y)$) we follow the descent (resp.~ascent) path using $U_e P_{\beta,\delta} U_e^{-1}$ to $-y$, then we flip from $-y$ to $y$ using $U_g P_{\beta,\delta}$, leading to
		\begin{align*}
			n(x,y) &\leq n+1,\quad \max_{i \in \llbracket n(x,y) \rrbracket} \mathcal{H}_\delta(p^{x,y}_i) \leq \max\{\mathcal{H}_\delta(x), \mathcal{H}_\delta(y)\}, \\
			\pi_{\beta,\delta}(p^{x,y}_i) ((P_{\beta,\delta})_{la})_{ra}(p^{x,y}_i,p^{x,y}_{i+1}) &\geq \dfrac{1}{4Z_{\beta,\delta}} e^{-\beta\max\{\mathcal{H}_\delta(x), \mathcal{H}_\delta(y)\} - \beta 2\delta}(1-e^{-\beta}).
		\end{align*}
		
		\item Case $3$: $x > 0, y < 0$. We first consider $U_e P_{\beta,\delta} U_e^{-1}$ to move from $x$ to $-y$, then we flip from $-y$ to $y$ using $U_g P_{\beta,\delta}$, leading to
		\begin{align*}
			n(x,y) &\leq n,\quad \max_{i \in \llbracket n(x,y) \rrbracket} \mathcal{H}_\delta(p^{x,y}_i) \leq \max\{\mathcal{H}_\delta(x), \mathcal{H}_\delta(y)\}, \\
			\pi_{\beta,\delta}(p^{x,y}_i) ((P_{\beta,\delta})_{la})_{ra}(p^{x,y}_i,p^{x,y}_{i+1}) &\geq \dfrac{1}{4Z_{\beta,\delta}} e^{-\beta\max\{\mathcal{H}_\delta(x), \mathcal{H}_\delta(y)\} - \beta 2\delta}(1-e^{-\beta}).
		\end{align*}
		
		\item Case $4$: $x < 0, y > 0$.
		We first flip from $x$ to $-x$ using $U_g P_{\beta,\delta}$ then we consider $U_e P_{\beta,\delta} U_e^{-1}$ to move from $-x$ to $y$, leading to
		\begin{align*}
			n(x,y) &\leq n,\quad \max_{i \in \llbracket n(x,y) \rrbracket} \mathcal{H}_\delta(p^{x,y}_i) \leq \max\{\mathcal{H}_\delta(x), \mathcal{H}_\delta(y)\}, \\
			\pi_{\beta,\delta}(p^{x,y}_i) ((P_{\beta,\delta})_{la})_{ra}(p^{x,y}_i,p^{x,y}_{i+1}) &\geq \dfrac{1}{4Z_{\beta,\delta}} e^{-\beta\max\{\mathcal{H}_\delta(x), \mathcal{H}_\delta(y)\}}(1-e^{-\beta}).
		\end{align*}
		
		\item Case $5$: $x = 0$. If $y > 0$ we consider $U_e P_{\beta,\delta} U_e^{-1}$ to move from $0$ to $y$. If $y < 0$, we consider $U_e P_{\beta,\delta} U_e^{-1}$ to move from $0$ to $-y$ then we flip from $-y$ to $y$ using $U_g P_{\beta,\delta}$, leading to
		\begin{align*}
			n(x,y) &\leq n,\quad \max_{i \in \llbracket n(x,y) \rrbracket} \mathcal{H}_\delta(p^{x,y}_i) \leq \max\{\mathcal{H}_\delta(x), \mathcal{H}_\delta(y)\}, \\
			\pi_{\beta,\delta}(p^{x,y}_i) ((P_{\beta,\delta})_{la})_{ra}(p^{x,y}_i,p^{x,y}_{i+1}) &\geq \dfrac{1}{4Z_{\beta,\delta}} e^{-\beta\max\{\mathcal{H}_\delta(x), \mathcal{H}_\delta(y)\} - \beta 2\delta}(1-e^{-\beta}).
		\end{align*}
	\end{itemize}
	Let $f \in L^2_0(\pi_{\beta,\delta})$, and $\chi_{z,w}(x,y)$ be $1$ if there exists some $i \in \llbracket n(x,y) \rrbracket$ such that $p^{x,y}_{i} = z, p^{x,y}_{i+1} = w$ and $0$ otherwise. We compute that
	\begin{align*}
		\langle f,f \rangle_{\pi_{\beta,\delta}} &= \dfrac{1}{2} \sum_{x,y} (f(y) - f(x))^2 \pi_{\beta,\delta}(y) \pi_{\beta,\delta}(x) \\
		&= \dfrac{1}{2} \sum_{x,y} \left(\sum_{i=1}^{n(x,y)} f(p^{x,y}_{i+1}) - f(p^{x,y}_{i})\right)^2 \pi_{\beta,\delta}(y) \pi_{\beta,\delta}(x) \\
		&\leq \dfrac{n}{2} \sum_{x,y} \sum_{i=1}^{n(x,y)} \left(f(p^{x,y}_{i+1}) - f(p^{x,y}_{i})\right)^2 \pi_{\beta,\delta}(y) \pi_{\beta,\delta}(x) \\
		&\leq \dfrac{n}{2} \sum_{x,y} \sum_{z,w} \chi_{z,w}(x,y)\left(f(w) - f(z)\right)^2 \pi_{\beta,\delta}(z) ((P_{\beta,\delta})_{la})_{ra}(z,w) \\
		&\quad \times \dfrac{4Z_{\beta,\delta} e^{\beta\max\{\mathcal{H}_\delta(x), \mathcal{H}_\delta(y)\}+\beta 2\delta}\pi_{\beta,\delta}(y) \pi_{\beta,\delta}(x)}{1-e^{-\beta}} \\
		&\leq n \left(\max_{z,w} \sum_{x,y;~x\neq y} \chi_{z,w}(x,y) \dfrac{4Z_{\beta,\delta} e^{\beta\max\{\mathcal{H}_\delta(x), \mathcal{H}_\delta(y)\}+ \beta 2 \delta}\pi_{\beta,\delta}(y) \pi_{\beta,\delta}(x)}{1-e^{-\beta}} \right) \\
		&\quad \times \left(\dfrac{1}{2}\sum_{z,w} \left(f(w) - f(z)\right)^2 \pi_{\beta,\delta}(z) ((P_{\beta,\delta})_{la})_{ra}(z,w)\right) \\
		&\leq n \left((2n+1)^2 \dfrac{4e^{\beta 2\delta}}{1-e^{-\beta}}\right) \langle f, (I - ((P_{\beta,\delta})_{la})_{ra})[f]\rangle_{\pi_{\beta,\delta}} \\
		&\leq \dfrac{36n^3e^{\beta 2\delta}}{1-e^{-\beta}} \langle f, (I - ((P_{\beta,\delta})_{la})_{ra})[f]\rangle_{\pi_{\beta,\delta}}.
	\end{align*}
	Rearranging gives the desired inequality.
\end{proof}

Denote the lazy Markov kernel of $((P_{\beta,\delta})_{la})_{ra}$ to be 
\begin{align*}
	L_{\beta,\delta} := \dfrac{1}{2}(I + ((P_{\beta,\delta})_{la})_{ra}).
\end{align*}

Another main result of this subsection demonstrates that $L_{\beta,\delta}$ enjoys rapid (i.e. polynomial in $n$) mixing time while $P_{\beta,\delta}$ has a torpid (i.e. exponential in $n$) mixing time.

\begin{proposition}\label{prop:improveMHperturb}
	In the setting of this subsection, for $\varepsilon > 0$ we have
	\begin{align*}
		t_{\mathrm{mix},1}(L_{\beta,\delta},\varepsilon) &\leq \dfrac{72n^3e^{\beta 2\delta}}{1-e^{-\beta}}\left(\beta n +  \log \left(\dfrac{2n+1}{\varepsilon}\right)\right), \\
		t_{\mathrm{mix},1}(P_{\beta,\delta},\varepsilon) &\geq \left(\dfrac{e^{\beta (n-2\delta)}}{(2n+1)^2} - 1\right)\log \left(\dfrac{1}{\varepsilon}\right).
	\end{align*}
\end{proposition} 

\begin{proof}
	By Proposition \ref{prop:spectralgapVshape}, we have
	\begin{align*}
		\lambda(L_{\beta,\delta}) \geq \dfrac{1-e^{-\beta}}{72n^3 e^{\beta 2\delta}}.
	\end{align*}
	In view of \cite[Theorem $12.4$]{levin2017markov}, the worst-case $L^1$ mixing time of $L_{\beta,\delta}$ is
	\begin{align*}
		t_{\mathrm{mix},1}(L_{\beta,\delta},\varepsilon) \leq \dfrac{1}{\lambda(L_{\beta,\delta})} \log \left(\dfrac{1}{\varepsilon \min_x \pi_{\beta,\delta}(x)}\right) \leq \dfrac{72n^3e^{\beta 2\delta}}{1-e^{-\beta}}\left(\beta n +  \log \left(\dfrac{2n+1}{\varepsilon}\right)\right).
	\end{align*}
	On the other hand, by noting that the so-called critical height of $P_{\beta,\delta}$ is $n-2\delta$, applying \cite[Lemma $2.3$]{holley1988simulated} leads to
	\begin{align*}
		\lambda(P_{\beta,\delta}) \leq (2n+1)^2 e^{-\beta (n-2\delta)},
	\end{align*}
	and by \cite[Theorem $12.5$]{levin2017markov}, we arrive at
	\begin{align*}
		t_{\mathrm{mix},1}(P_{\beta,\delta},\varepsilon) \geq \left(\dfrac{e^{\beta (n-2\delta)}}{(2n+1)^2} - 1\right)\log \left(\dfrac{1}{\varepsilon}\right).
	\end{align*}
\end{proof}

\subsection{Improving the simple random walk on the $n$-cycle}

In this subsection, we consider $P$ as the simple random walk on the $n$-cycle with the state space $\mathcal{X} = \llbracket n \rrbracket$ and discrete uniform $\pi$, where $n = 2^k$ for $k \in \mathbb{N}$. We have seen in Section \ref{subsec:achievediscreteuniformpi} that using a group of size linear in $n$ allows one to achieve exactly $\Pi$ in one projection step. In this subsection, we shall demonstrate that using a group of size in the order of $\log_2 n$ can lead to a worst-case $L^1$-mixing time of the order of polynomial in $\log_2 n$ (see Proposition \ref{prop:improvencycle} below), while the original $P$ exhibits diffusive behaviour with a mixing time of the order of $n^2$. It was also the original motivation in \cite{diaconis2000analysis} to propose non-reversible samplers that aim at overcoming this diffusive property.

The proof of the results rely on partitioning the state space recursively into a half, as inspired by the examples in \cite{jerrum2004elementrary}.

With such choice of $P$, we recall the notions of ``projection" and ``restriction" chain as investigated in \cite{jerrum2004elementrary}. For $a < b$ with $a,b \in \llbracket n-1 \rrbracket$, we write $P^{\llbracket a,b \rrbracket}$ to be the restriction chain of $P$ on the state space $\llbracket a,b \rrbracket$, that is,
\begin{align*}
	P^{\llbracket a,b \rrbracket}(x,x+1) &= \dfrac{1}{2}, \quad x \in \llbracket a,b-1\rrbracket, \\
	P^{\llbracket a,b \rrbracket}(x,x-1) &= \dfrac{1}{2}, \quad x \in \llbracket a+1,b\rrbracket, \\
	P^{\llbracket a,b \rrbracket}(a,a) &= P^{\llbracket a,b \rrbracket}(b,b) = \dfrac{1}{2},
\end{align*}
while the projection chain of $P$ induced by the partition $\llbracket a,c \rrbracket \cup \llbracket b,d \rrbracket$ to be $P^{\llbracket a,c \rrbracket, \llbracket b,d \rrbracket}$. Observe that $P^{\llbracket a,c \rrbracket, \llbracket b,d \rrbracket}$ is a two-state Markov chain, in which we label the states as $1,2$ in which the left partition $\llbracket a,c \rrbracket$ is state $1$ while the right partition $\llbracket b,d \rrbracket$ is state $2$.

We now define involutive permutations:
\begin{definition}[Block-reversal involutions on $\llbracket n \rrbracket$]
	Let $n=2^k$ with $k\in\mathbb{N}$. For each $j\in \llbracket k \rrbracket$ define
	a permutation $\sigma^{(j)}$ of $\llbracket n \rrbracket$ by
	\[
		\sigma^{(j)}(i) := q\,2^{j} + \bigl(2^{j}-1-r\bigr) + 1, \quad \text{where } i-1 = q\,2^{j} + r \quad (q\in \mathbb{N}\cup \{0\},\ 0\le r < 2^j).
	\]
	Equivalently, partition $\llbracket n \rrbracket$ into consecutive blocks of length $2^{j}$:
	\[
	\llbracket 1,2^{j} \rrbracket,\ \llbracket 2^{j}+1,2\cdot 2^{j} \rrbracket,\ \dots,\ \llbracket (m-1)2^{j}+1,m2^{j}\rrbracket,\ \dots
	\]
	with $m \in \llbracket 2^{k-j} \rrbracket$ and within each block reverse the order, leaving different blocks disjoint.
\end{definition}

\begin{remark}[Involution and structure]
	For every $j$, $\sigma^{(j)}$ is an involution:
	\[
	\bigl(\sigma^{(j)}\bigr)^2= e.
	\]
	Indeed, in each block the map is $r\mapsto 2^{j}-1-r$, whose self-composition is the identity.
	Hence $\sigma^{(j)}$ is a disjoint product of transpositions within each length-$2^j$ block.
	
	There are $k=\log_2(n)$ such involutions, indexed by $j=1,2,\ldots,k$.
\end{remark}

\begin{example}[Example: $n=32$ ($k=5$)]

We now write down the family $\{\sigma^{(j)}\}_{j=1}^{5}$ as illustrations:

\begin{itemize}
	\item \textbf{Top split.} Take $j=5$ (block size $2^5=32$), so there is a single
	block $\llbracket 1,32 \rrbracket$ and
	\[
	\sigma^{(5)}(i)=33-i,\qquad i=1,\dots,32.
	\]
	In particular $1\leftrightarrow 32,\ 2\leftrightarrow 31,\ \dots,\ 16\leftrightarrow 17$.
	
	\item \textbf{Next split into halves of $16$ and then into $8$.}
	For $j=3$, we have
	\[
	\sigma^{(3)}(i)=
	\begin{cases}
		9-i, & i\in\llbracket 1,8 \rrbracket,\\
		25-i, & i\in \llbracket 9,16 \rrbracket,\\
		41-i, & i\in \llbracket 17,24 \rrbracket,\\
		57-i, & i\in \llbracket 25,32 \rrbracket.
	\end{cases}
	\]
	Concretely: $\sigma^{(3)}(1)=8,\ \sigma^{(3)}(2)=7,\ldots,\sigma^{(3)}(8)=1$;
	$\sigma^{(3)}(9)=16,\ldots,\sigma^{(3)}(16)=9$; and similarly on $\llbracket 17,24 \rrbracket$ and $\llbracket 25,32 \rrbracket$.
	
	\item \textbf{Bottom split into pairs.}
	Taking $j=2$ (block size $4$) reverses each $4$-block:
	\[
	\sigma^{(2)} \text{ swaps } (1\,4)(2\,3)\,(5\,8)(6\,7)\,(9\,12)(10\,11)\,\dots\,(29\,32)(30\,31).
	\]
	For instance, $1\leftrightarrow 4,\ 2\leftrightarrow 3$ and $5\leftrightarrow 8,\ 6\leftrightarrow 7$.
\end{itemize}
	
\end{example}

We define $\sigma^{(0)} := e$, the identity. We now consider a finite group $G$ generated by $\{\sigma^{(j)}\}_{j=0}^{k}$, equipped with the discrete probability distribution $\nu$ given by
\begin{align*}
	\nu(j) = \dfrac{1}{k+1}, \quad j \in \{0\} \cup \llbracket k \rrbracket.
\end{align*}
With the above choices of $G$ and $\nu$, we consider
\begin{align*}
	P_{da} &= P_{da}(G,\nu \otimes \nu) = \dfrac{1}{(k+1)^2} \sum_{i,j=0}^k U_{\sigma^{(i)}} P U_{\sigma^{(j)}} 
\end{align*}
Clearly, $P_{da} \in \mathcal{L}(\pi)$ is $\pi$-reversible. The following results relate the spectral gap of the projection and restriction chains:

\begin{proposition}\label{prop:spectralgapdecomp}
	Let $k \geq 2$. For $j \in \llbracket 2,k \rrbracket$ and $m \in \llbracket 2^{k-j}\rrbracket$, we have
	\begin{align*}
		\lambda(P_{da}^{\llbracket (m-1)2^j +1, m2^j \rrbracket}) &= \min\bigg\{\lambda(P_{da}^{\llbracket ((2m-1)-1)2^{j-1} +1, (2m-1)2^{j-1} \rrbracket, \llbracket (2m-1)2^{j-1} + 1, m2^j \rrbracket}), \\
		&\quad \lambda(P_{da}^{\llbracket ((2m-1)-1)2^{j-1} +1, (2m-1)2^{j-1} \rrbracket}) \bigg\}, \\
		\lambda(P_{da}^{\llbracket ((2m-1)-1)2^{j-1} +1, (2m-1)2^{j-1} \rrbracket, \llbracket (2m-1)2^{j-1} + 1, m2^j \rrbracket}) &\geq \dfrac{1}{(k+1)^2}, 
	\end{align*}
\end{proposition}


\begin{proof}
	The first equality is a direct application of \cite[Corollary $3$]{jerrum2004elementrary}: by the symmetry of the $n$-cycle and the fact that for each $x \in \llbracket ((2m-1)-1)2^{j-1} +1, (2m-1)2^{j-1} \rrbracket$, there are precisely four paths to go to $y \in \llbracket (2m-1)2^{j-1} + 1, m2^j \rrbracket$ using $P_{da}^{\llbracket (m-1)2^j +1, m2^j \rrbracket}$ with probability $\frac{1}{2(k+1)^2}$, leading to $\hat{\eta} = 0$ in \cite[Corollary $3$]{jerrum2004elementrary}.
	
	For the second inequality, first we note that for stochastic matrices $M$ of the form
	\[
	M = \begin{pmatrix}
		1-b & b \\[3pt]
		b & 1-b
	\end{pmatrix},
	\]
	we have $\lambda(M) = 2b$. In our context, by taking $M = P_{da}^{\llbracket ((2m-1)-1)2^{j-1} +1, (2m-1)2^{j-1} \rrbracket, \llbracket (2m-1)2^{j-1} + 1, m2^j \rrbracket}$, recall the definition of $\sigma^{(j)}$ we note that
	\begin{align*}
		P_{da}^{\llbracket ((2m-1)-1)2^{j-1} +1, (2m-1)2^{j-1} \rrbracket, \llbracket (2m-1)2^{j-1} + 1, m2^j \rrbracket}(1,2) \geq \dfrac{1}{2(k+1)^2},
	\end{align*}
	and hence the desired result follows.
\end{proof}

By defining, for $j \in \llbracket n \rrbracket$,
\begin{align*}
	f(j) := \lambda(P_{da}^{\llbracket j \rrbracket}),
\end{align*}
using the symmetry of the $n$-cycle together with Proposition \ref{prop:spectralgapdecomp} we arrive at the recursion, for $l \in \llbracket 2,k \rrbracket$,
\begin{align*}
	f(2^{l}) \geq \min\bigg\{\dfrac{1}{(k+1)^2}, f\left(2^{l-1}\right)\bigg\},
\end{align*}
with the initial condition that $f(2) \geq \frac{1}{(k+1)^2}$. We thus have
\begin{align*}
	\lambda(P_{da}) = f(n) \geq \dfrac{1}{(k+1)^2}.
\end{align*}
Denote the lazy Markov kernel of $P_{da}$ to be
\begin{align*}
	L_{da} := \dfrac{1}{2}(I + P_{da})
\end{align*}
In view of \cite[Theorem $12.4$]{levin2017markov}, the worst-case $L^1$ mixing time of $L_{\beta,\delta}$ is
\begin{align*}
	t_{\mathrm{mix},1}(L_{da},\varepsilon) \leq \dfrac{1}{\lambda(L_{da})} \log \left(\dfrac{1}{\varepsilon \min_x \pi(x)}\right) \leq 2 (\log_2 n + 1)^2\log \left(\dfrac{n}{\varepsilon}\right).
\end{align*}

We collect the above result together with the well-known result that the mixing time of $P$ is $n^2$ (see \cite[Section $5.3.2$]{levin2017markov})

\begin{proposition}\label{prop:improvencycle}
	In the setting of this subsection, for $\varepsilon > 0$ we have
	\begin{align*}
		t_{\mathrm{mix},1}(L_{da},\varepsilon) &\leq 2 (\log_2 n + 1)^2\log \left(\dfrac{n}{\varepsilon}\right), \\
		t_{\mathrm{mix},1}\left(P,\dfrac{1}{8}\right) &\geq \dfrac{n^2}{32}.
	\end{align*}
\end{proposition}

\section*{Acknowledgements}

We thank Persi Diaconis and Lexing Ying for pointers to the literature. Michael Choi acknowledges the financial support of the projects A-8001061-00-00, NUSREC-HPC-00001, NUSREC-CLD-00001, A-0000178-01-00, A-0000178-02-00 and A-8003574-00-00 at National University of Singapore. Youjia Wang gratefully acknowledges the financial support from National University of Singapore via the Presidential Graduate Fellowship.

\bibliographystyle{abbrvnat}
\bibliography{ref}

\end{document}